\documentclass[10pt,reqno]{amsart}
\makeatletter\def\@@and{}\makeatother 

\usepackage[a4paper,left=35mm,right=35mm,top=30mm,bottom=30mm,marginpar=25mm]{geometry}

\usepackage{amsmath}
\usepackage{amssymb}
\usepackage{amsthm}
\usepackage[latin1]{inputenc}
\usepackage{eurosym}
\usepackage[dvips]{graphics}

\usepackage{graphicx}

\usepackage{epsfig}
\usepackage{hyperref}
\usepackage{dsfont}

\usepackage[displaymath,mathlines]{lineno}

\allowdisplaybreaks

\usepackage{hyperref}

\usepackage{ifthen}

\makeindex

\usepackage{graphicx}
\usepackage{caption}
\usepackage{subcaption}
\usepackage[nocompress, space]{cite}

\newcommand{\tr}{\operatorname{tr}}

\renewcommand{\div}{\operatorname{div}}

\newcommand{\Rr}{{\mathbb{R}}}

\newcommand{\Nn}{{\mathbb{N}}}

\newcommand{\Tt}{{\mathbb{T}}}

\newcommand{\Hh}{{\overline{H}}}

\newcommand{\Ll}{{\mathcal{L}}}

\newcommand{\Aa}{{\mathcal{A}}}

\newcommand{\Mm}{{\mathcal{M}}}

\newcommand{\bx}{{\bf x}}

\newcommand{\bdx}{\dot{\bf x}}

\newcommand{\balpha}{{\bar{\alpha}}}
\newcommand{\talpha}{{\tilde{\alpha}}}

\newcommand{\epsi}{\varepsilon}

\newcommand\sizefigure{0.40}

\def\leq{\leqslant}
\def\geq{\geqslant}

\let\weakly\rightharpoonup
\def\weaklystar{\buildrel{\hskip-.6mm\star}\over\weakly}

\numberwithin{equation}{section}

\newtheoremstyle{thmlemcorr}{10pt}{10pt}{\itshape}{}{\bfseries}{.}{10pt}{{\thmname{#1}\thmnumber{
#2}\thmnote{ (#3)}}}
\newtheoremstyle{thmlemcorr*}{10pt}{10pt}{\itshape}{}{\bfseries}{.}\newline{{\thmname{#1}\thmnumber{
#2}\thmnote{ (#3)}}}
\newtheoremstyle{defi}{10pt}{10pt}{\itshape}{}{\bfseries}{.}{10pt}{{\thmname{#1}\thmnumber{
#2}\thmnote{ (#3)}}}
\newtheoremstyle{remexample}{10pt}{10pt}{}{}{\bfseries}{.}{10pt}{{\thmname{#1}\thmnumber{
#2}\thmnote{ (#3)}}}
\newtheoremstyle{ass}{10pt}{10pt}{}{}{\bfseries}{.}{10pt}{{\thmname{#1}\thmnumber{
A#2}\thmnote{ (#3)}}}

\theoremstyle{thmlemcorr}
\newtheorem{theorem}{Theorem}
\numberwithin{theorem}{section}

\newtheorem{corollary}[theorem]{Corollary}
\newtheorem{proposition}[theorem]{Proposition}

\theoremstyle{thmlemcorr*}
\newtheorem{theorem*}{Theorem}
\newtheorem{lemma*}[theorem]{Lemma}
\newtheorem{corollary*}[theorem]{Corollary}
\newtheorem{proposition*}[theorem]{Proposition}
\newtheorem{problem*}[theorem]{Problem}
\newtheorem{conjecture*}[theorem]{Conjecture}

\theoremstyle{defi}

\newtheorem{hyp}{Assumption}

\theoremstyle{remexample}
\newtheorem{remark}[theorem]{Remark}

\newtheorem{lem}[theorem]{Lemma}

\theoremstyle{ass}

\usepackage{color}

\begin{document}

\title[First-order, stationary MFGs with congestion]{First-order, stationary mean-field games with congestion}

\author[D. Evangelista]{David Evangelista}
\address[D. Evangelista]{
        King Abdullah University of Science and Technology (KAUST), CEMSE
Division , Thuwal 23955-6900. Saudi Arabia, and  
        KAUST SRI, Center for Uncertainty Quantification in Computational
Science and Engineering.}
\email{david.evangelista@kaust.edu.sa}

\author[R. Ferreira]{Rita Ferreira}
\address[R. Ferreira]{
        King Abdullah University of Science and Technology (KAUST), CEMSE Division, Thuwal 23955-6900. Saudi Arabia, and  
        KAUST SRI, Center for Uncertainty Quantification in Computational Science and Engineering.}
\email{rita.ferreira@kaust.edu.sa}
\author[D. A. Gomes]{Diogo A. Gomes}
\address[D. A. Gomes]{
        King Abdullah University of Science and Technology (KAUST), CEMSE Division, Thuwal 23955-6900. Saudi Arabia, and  
        KAUST SRI, Center for Uncertainty Quantification in Computational Science and Engineering.}
\email{diogo.gomes@kaust.edu.sa}

\author[L. Nurbekyan]{Levon Nurbekyan}
\address[L. Nurbekyan]{
        King Abdullah University of Science and Technology (KAUST), CEMSE Division, Thuwal 23955-6900. Saudi Arabia, and  
        KAUST SRI, Center for Uncertainty Quantification in Computational Science and Engineering.}
\email{levon.nurbekyan@kaust.edu.sa}

\author[V. Voskanyan]{Vardan Voskanyan}
\address[V. Voskanyan]{
        King Abdullah University of Science and Technology (KAUST), CEMSE Division, Thuwal 23955-6900. Saudi Arabia, and  
        KAUST SRI, Center for Uncertainty Quantification in Computational Science and Engineering.}
\email{vardan.voskanyan@kaust.edu.sa}


\thanks{
        The authors were partially supported by King Abdullah University of Science and Technology (KAUST) baseline and start-up funds.  
}
\date{\today}

\begin{abstract}

Mean-field games (MFGs) are models for large populations of competing rational agents
that seek to optimize a suitable functional. In the case of congestion, 
this functional takes into account the difficulty of moving in high-density areas. 

Here, we study stationary MFGs with congestion with quadratic or power-like Hamiltonians. 
First, using explicit examples, we illustrate two main difficulties: the lack of classical solutions and the existence of 
areas with vanishing density. 
Our main contribution is 
a new variational formulation for MFGs with congestion. This formulation
was not previously known, and, thanks 
to 
it, we prove the existence and uniqueness of solutions. 
Finally, we consider applications to numerical methods. 
\end{abstract}

\maketitle

\section{Introduction}

Mean-field games (MFGs) is a branch of game theory that studies systems with a large number of competing agents. These games were introduced in \cite{ll1,ll2,ll3} (see also \cite{LCDF}) and \cite{Caines2,Caines1} motivated by problems arising in  population dynamics, mathematical economics, social sciences,  and engineering.  MFGs have been the focus of intense study in the last few years and substantial progress has been achieved. Congestion problems, which arise in models
where the motion of agents in high-density regions is expensive, are a challenging class of 
MFGs. 
Many MFGs 
are determined by a system of a Hamilton--Jacobi equation coupled with a transport or Fokker--Planck equation. In congestion problems, these equations have singularities and, thus, 
their analysis requires particular care.

Here, we study first-order stationary MFGs with congestion. Our main example is
the  system
\begin{equation}
\label{main}
\begin{cases}
\frac{|P+Du|^\gamma}{\gamma m^\alpha}+V(x)=g(m)+\Hh\\
-\div(m^{1-\alpha} |P+Du|^{\gamma-2}(P+Du))=0,  
\end{cases}
\end{equation}
where $x$ takes values on the $d$-dimensional torus, $\Tt^d$, and the
unknowns are  $u, m:\Tt^d\to \Rr$ and $\Hh\in \Rr$, with $m\geq 0$ and $\int_{\Tt^d} m\,dx=1$.  Here, 
 $1\leq\alpha\leq \gamma<\infty$, 
$V:\Tt^d\to \Rr$, $V\in C^\infty(\Tt^d)$, and $g:\Rr^+\to \Rr$ with 
$g(m)=G'(m)$ for some convex function $G:\Rr_0^+\to \Rr$ with $G\in C^\infty(\Rr^+)\cap C(\Rr_0^+)$.
In particular, the convexity of $G$ gives that 
 $g$ is monotonically increasing. 
 
The preceding MFG is a model where agents incur in a large cost if moving in regions with a high agent
density. The constant $-\Hh$ is the average cost per unit of time corresponding to the Lagrangian 
\[
L(x,v, m)=m^\alpha \frac{|v|^{\gamma'}}{\gamma'}+v\cdot P-V(x)+g(m),
\]
where $\frac 1 \gamma +\frac 1 {\gamma'}=1$. More precisely, the typical agent seeks to minimize the long-time average 
cost
\[
\lim_{T\to \infty}\frac 1 T \int_0^T \bigg(m^\alpha \frac{|\bdx(s)|^{\gamma'}}{\gamma'}+\bdx(s)\cdot P-V(\bx(s))+g(m(\bx(s)))\bigg)ds .
\]
Due to this optimization process, agents avoid moving in high-density regions. Further, because
$g$ is  increasing, agents prefer to remain in low-density regions
rather than in high-density regions. Finally, we observe that \(P\) determines the preferred direction of motion.

In the stationary case, the theory for second-order MFGs without singularities is well understood. For example, the papers \cite{GM, GPM1, GPatVrt, PV15, GR} address the existence of classical solutions
and weak solutions were examined in \cite{bocorsporr}.
In dimension one, a characterization of solutions for stationary MFGs 
was developed in \cite{Gomes2016b} (including non-monotone MFGs) and, in the case of congestion, in \cite{GNPr216, nurbekyan17}. The theory of weak solutions was considered in \cite{FG2}, where a general existence result was proven using a monotonicity argument. 
The monotonicity structure that many MFGs enjoy has important applications to numerical methods, see \cite{AFG}. 
A review of  MFG models can be found in \cite{GJS2}
and a survey of regularity results in \cite{GPV}. 

The congestion problem was first introduced in \cite{LCDF} where a uniqueness condition was established.
Next, the existence problem for stationary MFGs  with congestion,  positive viscosity, and a quadratic Hamiltonian was proved in \cite{GMit}.
Subsequently, this problem was examined in more generality in \cite{GE}. The time-dependent case was considered in \cite{GVrt2} (classical solutions) 
and \cite{Graber2} (weak solutions). Later, \cite{Achdou2016} examined weak solutions for time-dependent problems. 

Apart from the results in \cite{FG2},  the one-dimensional examples in  \cite{GNPr216, nurbekyan17},
and the radial cases in \cite{EvGomNur17}, little is known about first-order MFGs with congestion. The critical difficulties stem from two issues: first-order Hamilton--Jacobi equations provide little a priori regularity; second, because the transport equation is a first-order equation, we cannot use Harnack-type results and, thus, we cannot bound $m$ by below by a positive constant. 
Indeed, as we show in Section \ref{lcs}, $m$ can vanish. 
In many MFG problems, the regularity follows from a priori bounds that combine both equations in \eqref{main}. Here, with standard methods, we can only get relatively weak bounds, see Remark~\ref{rmk:onfoeinic}. For example, if  \(0<\alpha\leq 1\), then  there exists a constant, $C$, such that for any regular enough solution of \eqref{main}, 
we have
\[
\int_{\Tt^d}
\left[\left(\frac{|P+Du|^\gamma}{m^\alpha}\right) (1+m) + (m-1) g(m)\right]dx\leq C.   
\]
In 
Section \ref{apbsec}, 
we examine a priori bounds for a class of MFGs that generalize \eqref{main}.

While the bounds from Section \ref{apbsec} are interesting on their own, they are not enough to prove the existence of solutions. 
In the case of MFGs without congestion, a number of variational principles
have been proposed, see \cite{ll1, LCDF}. These are not only of
independent interest but also have important applications, see, for example,
\cite{MR3195846} for a study of efficiency loss in oscillator synchronization games, 
the recent results in \cite{GraCard} where variational principles are used to prove the existence
of solutions for first-order MFGs, and 
\cite{MR3644590} where optimal transport methods are used to examine constrained MFGs, which is
an alternative approach to model congestion. 
Hard congestion problems
can be modeled by variational problems, see for example
\cite{San12} or \cite{San16}. However, 
soft-congestion models such as \eqref{main} do not fit this framework. 
In Section \ref{vp}, we  study a new variational problem for which \eqref{main}
is the corresponding Euler--Lagrange equation.
More concretely, 
let $G:\Rr_0^+\to \Rr$ with $G'=g$. Then, \eqref{main} is the Euler--Lagrange equation of the functional 
\begin{equation}
\label{nvp}
J[u,m]=\int_{\Tt^d} \left(\frac{|P+Du|^\gamma}{\gamma (\alpha-1) m^{\alpha-1}}-V m +G(m)\right) dx;
\end{equation}
that is, if $(u,m)$, with $u,m:\Tt^d\to \Rr$ and $m>0$, is a smooth enough minimizer of $J$ under the constraint 
\[
\int_{\Tt^d} m \,dx=1, 
\]
then $(u,m)$ solves \eqref{main}.
The existence of a minimizer of \eqref{nvp} is
addressed as follows. First, the  $\alpha=1$ case is 
examined in Section \ref{ccsec}. The 
$1<\alpha<\gamma$ case and the case where $\alpha=\gamma$ with $g=m^\theta$, $\theta>0$, are addressed in 
Theorem \ref{thm:exist1}. Finally, the  $\alpha=\gamma$  with more general 
assumptions on $g$ case is considered in Theorem \ref{thm:exist2}.
The uniqueness of a solution 
is shown in Theorem \ref{thm:uniqmin}.
Our new variational principle provides a new construction of weak solutions for MFGs that does not rely on the high-order regularizations in \cite{FG2} nor requires ellipticity as in \cite{ GPatVrt,GPM1, GR, GM,  PV15}. Moreover, our methods  suggest an alternative computational approach for stationary MFGs with congestion that complements the existing ones, see \cite{AFG}. 

Our variational methods do not apply for $0<\alpha<1$
nor to second-order MFGs. Sections~\ref{2dcase} and 
\ref{tsom} are devoted to the study of these cases. 
In Section \ref{2dcase}, we examine
first-order MFGs with $0<\alpha<1$
in the two-dimensional case. There, we perform a
change of variables for which our variational methods can be used.
Next, in Section \ref{tsom}, we study various cases where
we can
reduce second-order MFG systems
to scalar equations. Moreover, we show that these equations 
are equivalent to Euler--Lagrange equations of suitable functionals.  
Finally, in Sections \ref{num}--\ref{num9}, we use our results to develop 
numerical methods for MFGs.

\section{Some explicit examples}
\label{ee}

Before developing the general theory, we consider three examples that illustrate some of the properties of \eqref{main}. First, we prove
that \eqref{main} may fail to have classical solutions. Next, we examine
the critical congestion case, $\alpha=1$. In this case, $u$ is constant and the  existence of solutions to 
\eqref{main} can be addressed by solving algebraic equations. 
%

\subsection{Lack of classical solutions}
\label{lcs}

In general, \eqref{main} may not have classical solutions. To illustrate this behavior, we consider the case  when $P=0,$  where the analysis  is elementary.
Here, to simplify the presentation, we take $\gamma=2$ and 
$g(m)=m$, but the analysis is similar for the general case. 
By adding a constant to $V$, we can assume without loss of generality that
\begin{equation}
\label{nv}
\int_{\Tt^d} \,Vdx=0. 
\end{equation}
In this case,  \eqref{main} becomes
\begin{equation}
\label{main2}
\begin{cases}
\frac{|Du|^2}{2 m^\alpha}+V(x)=m+\Hh\\
-\div(m^{1-\alpha} Du)=0.  
\end{cases}
\end{equation}

Now, we assume that \((u,m, \overline H)\) is a classical
solution to \eqref{main2} with \(m>0\)
 and \(\int_{\Tt^d} m\,dx =1\).
Then, multiplying the second equation by $u$ and integrating over $\Tt^d$, we have
\[
\int_{\Tt^d} m^{1-\alpha}|Du|^2\,dx=0.
\] 
Hence, because $m$ does not vanish, $u$ is constant. Accordingly, the first equation in \eqref{main2} becomes
\[
m=-\Hh + V(x).
\]
Using $\int_{\Tt^d}m \,dx=1$ and  \eqref{nv}, we obtain
\begin{equation}\label{eq:lcs}
m=1+V(x). 
\end{equation}
However, without further assumptions,  $1+V$ may take negative values and, thus,  \eqref{main} may not have a classical solution with $m>0$.  

For  a general MFG of the form
\[
\begin{cases}
m^\alpha H(\frac{Du}{m^\alpha}, x)=g(m)+\Hh\\
\div(D_pH(\frac{Du}{m^\alpha}, x) m)=0
\end{cases}
\]
with a Hamiltonian $H:\Rr^d\times \Tt^d\to
\Rr$ satisfying $D_pH (p,x)p>0$ for $p\in \Rr^d \backslash \{0\}$, a similar
argument yields $u$
constant.

\subsection{Critical congestion $\alpha=1$}
\label{ccsec}

If $\alpha=1$ and $\gamma=2$, the second equation in \eqref{main} becomes $\Delta u=0$. Hence, $u$ is constant. Therefore, 
the first equation in \eqref{main} is the following algebraic equation for $m$:
\[
\frac{|P|^2}{2 m}-g(m)=\Hh-V(x).  
\]
Suppose that $g$ is increasing and that $P\neq 0$. Then, for each $x$ and for each fixed $\Hh$, the preceding equation has at most one solution, $m(x)>0$. Furthermore, the constant $\Hh$ is 
determined by the normalization condition on $m$. 

The  $\gamma\neq 2$ case is similar; the second equation in  \eqref{main} is
\[
\div(|P+Du|^{\gamma-2}(P+Du))=0. 
\]
The prior equation is the $\gamma$-Laplacian equation for the function $P\cdot x+u(x)$ and, again, $u$ is constant due to the periodicity.

\section{Some formal a priori bounds}
\label{apbsec}
In this section, we prove some a priori bounds for the following more general version of  \eqref{main}:
\begin{equation}\label{mainGen}
        \begin{cases}
        m^\balpha H\big(\frac{P+Du}{m^\balpha}\big)+V(x)=g(m)+\Hh\\
        -\div\left(mD_pH\big(\frac{P+Du}{m^\balpha}\big) \right)=0,
        \end{cases}
\end{equation}
where \(H:\Rr^d\to\Rr\) is the Hamiltonian and \(\bar \alpha\) 
is the congestion parameter. We note that for 
\begin{equation}\label{hamilt1}
        H(p)=\frac{|p|^\gamma}{\gamma}, 
\end{equation}
\eqref{mainGen} reduces to \eqref{main}  by setting $\balpha:=\frac{\alpha}{\gamma-1}$.

The bounds established  next give partial regularity for solutions of  \eqref{main}.  Unfortunately, this regularity is not sufficient to ensure the existence of solutions. We examine the existence of solutions in the next section using methods from the calculus of variations. 

As before, we assume that:
\begin{hyp}\label{A0}
The congestion parameter, $\balpha$, is non-negative, $V\in C^\infty(\Tt^d)$, and $g:\Rr^+\to \Rr$ is a $C^\infty$ monotonically increasing function. 
\end{hyp}

Regarding the Hamiltonian,  we work under the following assumptions:
 \begin{hyp}\label{A4}
        The Hamiltonian, $H:\Rr^d\to\Rr$, is a $C^\infty$ function. Moreover,  there exists a constant, $C>0$, such that
        \begin{equation*}
        D_pH(p)\cdot p - H(p)\geq \frac 1 C H(p)-C
        \end{equation*}
        for all $p\in \mathbb R^d$.
\end{hyp}


\begin{hyp}\label{As1.2} 
        There exist $\gamma>1$ and a constant, $C>0$, such that
        \[
        \frac1C|p|^{\gamma}- C\leq H(p)\leq  C\big(|p|^{\gamma}+1\big) 
        \]
        for all $p\in \mathbb R^d$. 
\end{hyp}

\begin{remark}\label{A3}
        Note that Assumptions~\ref{A4} and \ref{As1.2}  imply that there exists a constant,  $\tilde C>0$, such that
        \begin{equation*}
        D_pH(p)\cdot p - H(p)\geq \frac1{\tilde C}|p|^{\gamma}- \tilde C
        \end{equation*}
        for all $p\in \mathbb R^d$.
\end{remark}

\begin{hyp}\label{ADpH} 
        For all \(\delta>0\), there exists a constant, $C_\delta>0$, such         that
        \[
        |D_pH(p)| \leq C_\delta + \delta H(p) 
        \]
        for all $p\in \mathbb R^d$. 
\end{hyp}

\begin{hyp}\label{ADpH2} 
        There exists a constant, $C>0$, such
        that
        \[
        |D_pH(p)| \leq C(|p|^{\gamma-1}+1)
        \]
        for all $p\in \mathbb R^d$. 
\end{hyp}
        
\begin{hyp}\label{A5}
        There exist a constant,  $\varepsilon\in (0, 1)$, such that for any symmetric matrix $M\in \Rr^{d\times d}$ and vector $p\in \Rr^d$, we have
        \[
        \bar{\alpha}|D^2_{pp}H(p)M|^2|p|^2\leq 4(1-\epsi)(D_pH(p)\cdot p-H(p)) \tr(D^2_{pp} H(p)MM),
        \]
        where $|A|=\sup\limits_{x\neq 0} \frac{|Ax|}{|x|}$ for a matrix $A$.
\end{hyp}
The preceding assumptions are similar to the ones in \cite{GE}, where examples of  Hamiltonians  satisfying them are discussed. For example, it is not hard to check that $H(p)=\frac{|p|^{\gamma}}{\gamma}+b \cdot p$ with \(b\in\Rr^d\) satisfies  the prior assumptions, including Assumption \ref{A5} for $\bar{\alpha}< 4/{\gamma}$ and $\gamma\in[1,2]$.

Next,  we prove  two a priori estimates for solutions of \eqref{mainGen}, the first of which holds only for \(\balpha \leq 1\).

\begin{proposition}\label{prop:foe} 
Suppose that Assumptions~\ref{A0}--\ref{ADpH}
hold and that \(\balpha \leq 1\).  Then, there exists a constant, $C>0$, such that, for any 
smooth solution of \eqref{mainGen}, $(u,m, \Hh)$ with \(m>0\) and \(\int_{\Tt^d} m\, dx =1\), we have
\begin{equation}
\label{eq:foe0}
\begin{aligned}
\int_{\Tt^d}
\left[\Big|\frac{P+Du}{m^{\bar \alpha}}\Big|^\gamma ( m^{\bar \alpha } +m^{\bar \alpha +1}) + (m-1) g(m)\right]dx\leq C\bigg(1+ \int_{\Tt^d} m^{\bar \alpha +1}\, dx\bigg) . 
\end{aligned}
\end{equation}
\end{proposition}

\begin{remark}\label{rmk:onfoe}
We observe that for several examples of \(g\), \eqref{eq:foe0}
has encoded \textit{better} integrability conditions. For instance, if there exist \(c>0\) and  \(\theta>\bar\alpha\) such that  \(\frac1c m^\theta - c \leq g(m) \leq c(1+ m^\theta)\)  in
the previous proposition, then \eqref{eq:foe0} can be replaced by 
\begin{equation}
\label{eq:foe00}
\begin{aligned}
\int_{\Tt^d}
\left[\Big|\frac{P+Du}{m^{\bar \alpha}}\Big|^\gamma ( m^{\bar \alpha } +m^{\bar
\alpha +1}) + m^{\theta  +1}\right]dx\leq C.  
\end{aligned}
\end{equation}

\end{remark}

\begin{proof}[Proof of Proposition~\ref{prop:foe}]

Subtracting  the first equation in \eqref{mainGen} multiplied by $(m-1)$ from the second equation in \eqref{mainGen} multiplied by $u$ first, and then integrating the resulting equation over \(\Tt^d\), we obtain
\begin{equation}\label{eq:foe1}
\begin{aligned}
&\int_{\Tt^d} m^{\bar \alpha +1} \bigg(D_p H\Big(\frac{P+Du}{m^\balpha} \Big)\cdot \frac{P+Du}{m^\balpha}- H\Big( \frac{P+Du}{m^\balpha} \Big)\bigg)dx + \int_{\Tt^d}  m^{\bar \alpha}  H\Big(\frac{P+Du}{m^\balpha} \Big)\,dx  \\
&\quad - \int_{\Tt^d}  m D_p H\Big(\frac{P+Du}{m^\balpha} \Big)\cdot
P\,dx+ \int_{\Tt^d}  (m-1) g(m) \,dx =  \int_{\Tt^d} V(x) (m-1)\,dx\leq 2\Vert V\Vert_\infty ,
\end{aligned}
\end{equation}
where we used integration by parts and the condition  \(\int_{\Tt^d} m\, dx =1\). 

By Remark~\ref{A3}, we have
%
\begin{align*}
&\int_{\Tt^d} m^{\bar \alpha +1} \bigg(D_p H\Big(\frac{P+Du}{m^\balpha} \Big)\cdot
\frac{P+Du}{m^\balpha}- H\Big( \frac{P+Du}{m^\balpha} \Big)\bigg)dx\\
&\quad \geq \frac1{\tilde C} \int_{\Tt^d} m^{\bar \alpha +1} 
 \Big|\frac{P+Du}{m^\balpha} \Big|^\gamma \,dx - \tilde C
\int_{\Tt^d} m^{\bar \alpha +1} \, dx.
\end{align*}
%
Moreover, by
Assumption~\ref{As1.2}, we have
\begin{equation}
\label{eq:foe3}
\begin{aligned}
 \int_{\Tt^d}
 m^{\bar \alpha}  H\Big(\frac{P+Du}{m^\balpha} \Big)\,dx \geq \frac1{ C} \int_{\Tt^d} m^{\bar \alpha } 
 \Big|\frac{P+Du}{m^\balpha} \Big|^\gamma \,dx -  C
\int_{\Tt^d} m^{\bar \alpha } \, dx.
\end{aligned}
\end{equation}
Next, using Assumptions~\ref{As1.2} and~\ref{ADpH} with \(\delta=\frac{1}{2C(\tilde C+ C)|P|}\) and recalling that  \(\int_{\Tt^d} m\,
dx =1\), we can find a   constant, \(\bar C>0\),
such that
\begin{equation}
\label{eq:foe5}
\begin{aligned}
 -\int_{\Tt^d}
 m D_p H\Big(\frac{P+Du}{m^\balpha} \Big)\cdot P\,dx \geq -\frac1{2(C+\tilde C)}
\int_{\Tt^d} m
 \Big|\frac{P+Du}{m^\balpha} \Big|^\gamma \,dx -  \bar C.
\end{aligned}
\end{equation}
Finally, we observe that if \(0<\balpha \leq 1\), then
\(m\leq m^\balpha\) if \(m\leq1\) and 
\(m\leq m^{\balpha+1}\) if \(m\geq1\). Hence, 
\begin{equation*}
\begin{aligned}
&-\frac1{2(C+\tilde
C)}
\int_{\Tt^d} m
 \Big|\frac{P+Du}{m^\balpha} \Big|^\gamma \,dx\\
 &\qquad \geq -\frac1{2C}
\int_{\Tt^d} m^\balpha
 \Big|\frac{P+Du}{m^\balpha} \Big|^\gamma \,dx-\frac1{2\tilde
C}
\int_{\Tt^d} m^{\balpha + 1}
 \Big|\frac{P+Du}{m^\balpha} \Big|^\gamma \,dx,
\end{aligned}
\end{equation*}
which, together with \eqref{eq:foe1}--\eqref{eq:foe5} and the estimate \(m^\balpha \leq m^{\balpha+1} + 1\), yields
 \eqref{eq:foe0}.
\end{proof}

\begin{proposition}\label{prop:foe2} 
Suppose that Assumptions~\ref{A0}--\ref{As1.2} and Assumption~\ref{ADpH2}
hold.  Then, there exists a constant, $C>0$,
such that, for any 
smooth solution, $(u,m, \Hh)$, of \eqref{mainGen} with \(m>0\) and \(\int_{\Tt^d}
m\, dx =1\), we have
\begin{equation}
\label{eq:foe02}
\begin{aligned}
\int_{\Tt^d}
\left[\Big|\frac{P+Du}{m^{\bar \alpha}}\Big|^\gamma ( m^{\bar \alpha } +m^{\bar
\alpha +1}) + (m-1) g(m)\right]dx\leq C\bigg(1+ \int_{\Tt^d} m^{1-\bar 
\alpha(\gamma-1)}\, dx\bigg) . 
\end{aligned}
\end{equation}
\end{proposition}

\begin{proof} We proceed exactly as in the proof of Proposition~\ref{prop:foe} up to the estimate \eqref{eq:foe5}. Here, to estimate the term on the left-hand side of \eqref{eq:foe5}, we argue as follows. Using Assumption~\ref{ADpH2} and the condition  \(\int_{\Tt^d} m\,
dx =1  \) first, and then the fact that \(|p|^{\gamma-1} \leq
\frac{1}{2\tilde CC|P|} |p|^\gamma + \bar C\) for all \(p\in\Rr^d\) and for
some positive constant \(\bar C\) independent of \(p\),  we obtain
\begin{equation}\label{eq:foe62}
\begin{aligned}
&-\int_{\Tt^d}
 m D_p H\Big(\frac{P+Du}{m^\balpha} \Big)\cdot P\,dx \\
 \geq& - C|P|\int_{\Tt^d} m^{1+\balpha}
 \frac{|P+Du|^{\gamma-1}}{m^{\balpha\gamma}}\,dx -C|P|\\
\geq&-\frac1{2\tilde
C}
\int_{\Tt^d} m^{1+\balpha}
 \Big|\frac{P+Du}{m^\balpha} \Big|^\gamma \,dx -C|P|\bigg(1+ \bar C \int_{\Tt^d} m^{1-\bar 
\alpha(\gamma-1) }\, dx\bigg) . 
\end{aligned}
\end{equation}
 From  \eqref{eq:foe1}--\eqref{eq:foe3}, \eqref{eq:foe62}, and using  the estimate \(m^\balpha
\leq m^{\balpha+1} + 1\), we deduce
 \eqref{eq:foe02}.
\end{proof}

\begin{remark}\label{rmk:onfoe2}
If \(\balpha(\gamma-1) \leq1\) and, for instance,   there
exist \(c>0\) and  \(\theta>0\) such that  \(\frac1c m^\theta - c \leq
g(m) \leq c(1+ m^\theta)\)  in
the previous proposition, then \eqref{eq:foe02} can be replaced by  \eqref{eq:foe00}.
\end{remark}

\begin{remark}\label{rmk:onfoeinic}
As we mentioned before, \eqref{mainGen} reduces to \eqref{main}  for $\balpha=\frac{\alpha}{\gamma-1}$ and \(H\) given by \eqref{hamilt1}.   In this case, the condition
\(0<\balpha\leq 1\) is equivalent to \(0<\alpha\leq\gamma-1\), while  \(\balpha(\gamma-1) \leq1\) is equivalent to \(0<\alpha\leq 1\). For \(\alpha\) and \(\gamma\) in this range, Propositions~\ref{prop:foe} and \ref{prop:foe2} and Remarks~\ref{rmk:onfoe} and \ref{rmk:onfoe2} provide a priori estimates for smooth solutions of  \eqref{main}.

\end{remark}

The following proposition gives an a priori second-order estimate.
\begin{proposition} 
        Suppose that Assumptions~\ref{A0} and  \ref{A5} hold.  Then, there exists a constant, $C>0$, such that, for any 
smooth solution of \eqref{mainGen}, $(u,m, \Hh)$ with \(m>0\) and \(\int_{\Tt^d}
m\, dx =1\), we have
        \begin{equation}\label{SecOrdEst}
        \int_{\Tt^d}\tr\Big(D^2_{pp}H\Big(\frac{P+Du}{m^\balpha} \Big)D^2uD^2u\Big)m^{1-\balpha} \,dx+\int_{\Tt^d} g'(m)|Dm|^2\,dx
        \leq C.
        \end{equation}
\end{proposition}
\begin{proof}
        For simplicity, we omit the argument,   $\frac{P+Du}{m^\balpha}$, of the Hamiltonian and its derivatives. Differentiating the first equation in \eqref{mainGen} with respect to $x_k$ and using Einstein summation convention, we have
        \begin{equation}\label{HJdiff}
        \balpha m^{\balpha-1}m_{x_k}H +D_{p_i}H u_{x_ix_k}-\balpha \frac{ (P_i+u_{x_i})D_{p_i} H m_{x_k}}{m} +V_{x_k}(x)=g'(m)m_{x_k}.
        \end{equation}
        Next, we note that
        \[
        (D_{p_i}H u_{x_ix_k})_{x_k}=\frac{D^2_{p_ip_j}H u_{x_jx_k}u_{x_ix_k}}{m^\balpha}-\frac{\balpha D^2_{p_ip_j}H(P_j+u_{x_j})m_{x_k}u_{x_i x_k}}{m^{\balpha+1}} +D_{p_i}H (u_{x_kx_k})_{x_i}.
        \]
        By differentiating \eqref{HJdiff} with respect to $x_k$ and using the previous equality, we obtain
        \begin{align}\label{HJ2diff}
        \begin{split}
        (\balpha m^{\balpha-1}m_{x_k}H)_{x_k}+&  
        \frac{D^2_{p_ip_j}H u_{x_jx_k}u_{x_ix_k}}{m^\balpha}
        -\frac{\balpha D^2_{p_ip_j}H(P_j+u_{x_j})m_{x_k}
        u_{x_i x_k}}{m^{\balpha+1}}\\ +D_{p_i}H (u_{x_kx_k})_{x_i} 
        &-\left(\balpha \frac{(P_i+u_{x_i})
        D_{p_i}H m_{x_k}}{m}\right)_{x_k} 
        +V_{x_kx_k}(x)=(g'(m)m_{x_k})_{x_k}.
        \end{split}
        \end{align}
        Now, we multiply the second equation in \eqref{mainGen} by $u_{x_kx_k}$ and integrate by parts. Accordingly, we get the identity       
         \[
        0=\int_{\Tt^d}-\div(mD_pH)u_{x_kx_k}\,dx=\int_{\Tt^d}m D_{p_i}H(u_{x_kx_k})_{x_i}\,dx.
        \]
        Next, we multiply \eqref{HJ2diff} by $m$, integrate by parts, and use the prior identity to derive 
        \begin{align*}
        &\int_{\Tt^d} \tr (D^2_{pp}H D^2uD^2u) m^{1-\balpha}\,dx+\int_{\Tt^d}g'(m)|Dm|^2 \,dx\\
&\quad\leq \int_{\Tt^d}\balpha m^{\balpha-1}H|Dm|^2\,dx-\int_{\Tt^d}\balpha \frac{(P+Du)\cdot D_pH|Dm|^2}{m}\,dx-\int_{\Tt^d}m \Delta V \,dx\\
         &\qquad+\int_{\Tt^d}\frac{\balpha |D^2_{pp}H D^2u||P+Du| |Dm| }{m^{\balpha}}\,dx\\
        &\quad\leq \int_{\Tt^d}\balpha m^{\balpha-1}|Dm|^2\left(H-D_pH\cdot \frac{P+Du}{m^\balpha}\right) dx +\int_{\Tt^d}\frac{\balpha |D^2_{pp}H D^2u||P+Du||Dm|}{m^{\balpha}}dx \\ & \qquad+ \Vert V\Vert_{C^2(\Tt^d)},
        \end{align*} 
where in the last inequality we used that $m$ is a probability density. Setting $Q:=\frac{P+Du}{m^\balpha}$ and \(C:= \Vert V\Vert_{C^2(\Tt^d)}\), so far we proved that         
\begin{equation}\label{soe1}
\begin{aligned}
        &\int_{\Tt^d}\tr(D^2_{pp}H(Q)D^2uD^2u)m^{1-\balpha} \,dx+\int_{\Tt^d} g'(m)|Dm|^2\,dx\\
&\quad\leq \int_{\Tt^d}\balpha m^{\balpha-1}|Dm|^2\left(H(Q)-D_pH(Q)\cdot Q\right)
\,dx +\int_{\Tt^d}\balpha |D^2_{pp}H(Q) D^2u||Q||Dm|dx + C.
        \end{aligned}
   \end{equation}     
  Finally, from Assumption~\ref{A5} and Cauchy's inequality, we have
        \begin{align*}
&\balpha |D_{pp}^2H(Q) D^2u ||Q||Dm|\\
&\quad\leq 2\balpha \sqrt{(D_pH(Q)\cdot Q-H(Q))}m^{\frac{\balpha-1}{2} } |Dm|\sqrt{
(1-\epsi)\tr(D^2_{pp} H(Q)D^2u D^2u) } m^{\frac{1-\balpha}{2} }  \\
&\quad\leq \bar{\alpha}(D_pH(Q)\cdot
Q -H(Q)) m^{\bar{\alpha}-1} |Dm|^2  + (1-\epsi) \tr(D_{pp}^2H(Q)D^2uD^2u) 
m^{1-\balpha },
        \end{align*}
which together with \eqref{soe1} gives
        \[
        \int_{\Tt^d}\varepsilon \tr(D^2_{pp}H(Q)D^2uD^2u)m^{1-\balpha} dx+\int_{\Tt^d} g'(m)|Dm|^2dx\leq C.\qedhere
        \]        
\end{proof}

It is worth mentioning that the estimate \eqref{SecOrdEst} is an analog  of the second-order estimates proved in
\cite{ll2}. 
For $g(m)=m^\theta$, \eqref{SecOrdEst} can be combined with the Sobolev theorem 
to yield improved integrability for $m$.

\section{A variational problem}
\label{vp}

In this section, we  study
 a minimization problem associated
with the functional in \eqref{nvp}, $J$. To incorporate the case in
which \(m\) is zero on a set of positive
measure, we consider the following extension
of   \(J\). Let \(\bar
J\) be the functional defined by
\begin{equation}
\label{nvpnew}
\bar J[u,m]=\int_{\Tt^d} \left[\bar
f(\nabla u,
m)-V m +G(m)\right]
dx,
\end{equation}
where, for \((p,m)\in\Rr^d \times \Rr^+_0\),
\begin{equation}\label{barf}
\begin{aligned}
\bar f (p,m) = \begin{cases}
\frac{|P+p|^\gamma}{\gamma
(\alpha-1) m^{\alpha-1}} & \text{if
} m\not=0,\\
+\infty & \text{if } m=0 \text{ and } p\not= -P,\\
0 &\text{if }  m=0 \text{ and
} p= -P.
\end{cases}
\end{aligned}
\end{equation}

We aim at proving the existence and uniqueness
of solutions to the variational problem
\begin{equation}
\label{mmz}
\min_{(u,m)\in \Aa_{q,r}} \bar J[u,m],
\end{equation}
where $\bar J$ is given by \eqref{nvpnew} and
$\Aa_{q,r}$ is the set
\[
\Aa_{q,r}=\left\{(u,m)\in W^{1,q}(\Tt^d)
\times 
L^r(\Tt^d):\int_{\Tt^d} u \,dx=0,\int_{\Tt^d} m \,dx=1,m\geq 0\right\},
\]
with $q\geq1$ and \(r\geq 1\)  to be chosen later. 

To this end, we prove in
Section~\ref{auxformin} that
 \(\bar f \) is convex and lower semi-continuous.
These properties entail the sequential,
weakly lower semi-continuity of  \(\bar J\) in an appropriate
function space. This result is a key ingredient in the proof of the existence
of solutions to \eqref{mmz}, which is presented in Section~\ref{secexist}. Next, in  Section~\ref{uniqmin},
we discuss the uniqueness of these solutions. Finally, in Section~\ref{varexplicit}, we  further characterize the solutions in  the  \(P=0\) case. This characterization will be useful to validate our numerical methods in Sections~\ref{num}--\ref{num9}.

\subsection{Lower semi-continuity properties of
$\bar J$}\label{auxformin}

Here, we study the lower
semi-continuity 
of the functional
$\bar J$ given by \eqref{nvpnew}. We
first prove that  \(\bar f \) defined in \eqref{barf} is convex and lower semi-continuous.

\begin{lem}
\label{barfcxlsc}
Suppose that $1<\alpha\leq\gamma$. Then,
$\bar f$  given by \eqref{barf} is convex and lower semi-continuous in \(\Rr^d \times \Rr^+_0\).         
\end{lem}
\begin{proof}
We begin by proving that \(\bar f\) is convex
in \(\Rr^d\times
\Rr^+\).
Without
loss of generality, 
we assume that \(P=0\). Fix \(\alpha\) and \(\gamma\)  such
that  $1<\alpha\leq\gamma$,
and set \(\kappa_1 = \gamma - \alpha
+1\) and \(\kappa_2
= \frac{\gamma}{\gamma - \alpha +1}\).
Note that \(1 \leq \kappa_1 <\gamma
\) and \(1 < \kappa_2 \leq \gamma\).
Moreover, for \((p,m)\in\Rr^d \times \Rr^+\),
we may rewrite \(\bar f \) as
\[
\bar f(p,m) =\frac{1}{\gamma(\alpha-1)}  \left( m \left|\frac{p}{m}\right|^{\kappa_2}
\right) ^{\kappa_1} = \phi\left( m\,\psi\left(
\frac{p}{m}
\right)\right),
\]
where \(\phi(t) = \frac{1}{\gamma(\alpha-1)}t^{\kappa_1}\) for
\(t \in \Rr^+_0\) and \(\psi(p)=|p|^{\kappa_2}\)
for \(p\in\Rr^d\). Because \(\psi\)
is a convex function in \(\Rr^d\), its perspective function,
defined by \(\bar \psi  (p,m)=  m\,\psi\left(
\frac{p}{m}
\right)\), is  convex in \(\Rr^d\times
\Rr^+\) (see, for instance, \cite[Lemma~2]{DaMa08}).
Because \(\phi\) is an
increasing convex function in \(\Rr^+_0\),
 \(\bar f\) is a convex function in \(\Rr^d\times
\Rr^+\).

Next, we prove that \(\bar f\) is convex
in \(\Rr^d\times
\Rr^+_0\). Let \(\lambda \in (0,1)\),
\(p_1,\, p_2 \in \Rr^d\), \(m_1,\, m_2
\in \Rr^+_0\). We want to show that
\begin{equation}\label{barfcx}
\begin{aligned}
\bar f(\lambda (p_1,m_1) + (1-\lambda) (p_2,
m_2))
\leq \lambda\bar f (p_1,m_1) + (1-\lambda)
\bar f(p_2, m_2).
\end{aligned}
\end{equation}
We are only left to prove that \eqref{barfcx}
holds when either \(m_1=0\) or \(m_2=0\).
Consider first the \(m_1=0 \)
case. If \(p_1\not=-P\) or \(m_2=0\) and  \(p_2\not=-P\), then
the right-hand side of \eqref{barfcx}
equals to \(\infty\); thus, \eqref{barfcx}
holds in these sub-cases. If \(p_1=-P\)  and  \(p_2=-P\),
then  \eqref{barfcx} reduces to the condition
 \(0\leq 0\); thus, \eqref{barfcx}
holds in this sub-case.
Finally, 
if \(p_1=-P\),    \(p_2\not=-P\), and \(m_2\not=0\),
 then  \eqref{barfcx} reduces to the condition
 \((1-\lambda)^{\gamma-\alpha}\leq 1\); thus, \eqref{barfcx}
also holds in this sub-case  because \(1-\lambda \in (0,1)\)
and \(\gamma-\alpha\geq0\).
 The  \(m_2=0\) case is analogous.

Finally, we prove that  $\bar f$  is lower semi-continuous in
\(\Rr^d \times \Rr^+_0\). This amounts
to showing that if \((p,m),\, (p_j,m_j)
\in \Rr^d \times \Rr^+_0\), \(j\in\Nn\),
are such that \((p_j,m_j) \to (p,m)\)
in \(\Rr^d \times \Rr^+_0\) as \(j\to\infty\), then
\begin{equation}\label{barflsc}
\begin{aligned}
\bar f(p,m) \leq \liminf_{j\to\infty}
\bar f(p_j,m_j).
\end{aligned}
\end{equation}

Because  \(\bar f\) is convex
in \(\Rr^d\times
\Rr^+_0\), it is continuous in the interior
of its effective domain. Thus, we are
left to prove that \eqref{barflsc} holds
when \(m=0\).

Assume that \(m=0\). If \(p=-P\), then \eqref{barflsc}
holds because   \(\bar
f(m,p) = \bar f(0,-P) =0 \) in this case. If \(p\not=-P\),
then \(p_j\not=-P\) for all \(j\in\Nn\)
sufficiently large. For  any such \(j\), we have
\begin{equation*}
\begin{aligned}
\bar f (p_j,m_j) = \begin{cases}
\frac{|P+p_j|^\gamma}{\gamma
(\alpha-1) m_j^{\alpha-1}} & \text{if
} m_j\not=0,\\
+\infty & \text{if } m_j=0 .
\end{cases}
\end{aligned}
\end{equation*}
Define \(S=\{j\in\Nn\!: \, m_j \not =0\}\). If \(S\) has finite cardinality, then \(\bar f(p_j,m_j) = \infty\)  for all \(j\in\Nn\)
sufficiently large; thus  \eqref{barflsc}
holds. If \(S\)
has infinite cardinality, then 
\[\liminf_{j\to\infty}
\bar f(p_j,m_j) = \liminf_{j\to\infty\atop j\in S}
\bar f(p_j,m_j) = \liminf_{j\to\infty\atop j\in S} \frac{|P+p_j|^\gamma}{\gamma
(\alpha-1) m_j^{\alpha-1}} = \frac{|P+p|^\gamma}{0^+}=  \infty;\]
    thus,  \eqref{barflsc}
holds.
\end{proof}

The following  proposition
is a simple consequence of \cite[Theorem~5.14]{FoLe07}
and is closely related to the sequential,
weakly lower semi-continuity of \(\bar
J\). 

\begin{proposition}\label{slscLq}
Let \(\bar f\) be the
function given by \eqref{barf} with
\(1<\alpha \leq \gamma\). Then, the functional
\[
(v_1,v_2) \in L^1(\Tt^d;\Rr^d) \times L^1(\Tt^d; \Rr^+_0) \mapsto
\int_{\Tt^d} \left[ \bar f(v_1(x), v_2(x)) + G(v_2(x)) \right]\, dx
\]
is sequentially lower semi-continuous with respect to
the weak convergence in \(L^1(\Tt^d) \times L^1(\Tt^d; \Rr^+_0)\).
\end{proposition}
\begin{proof}  Because \(G\) is
a real-valued, convex function on \(\Rr^+_0\), we have 
\(G(m)\geq -C_0(1+ m)\) for all \(m\in\Rr^+_0\) and
for some positive constant \(C_0\) independent of  \(m\).
Then, by Lemma~\ref{barfcxlsc} and the non-negativeness of \(\bar
f\),  the mapping
\begin{equation*}
\begin{aligned}
(p,m)\in \Rr^d \times \Rr^+_0 \mapsto \bar f(p,m) + G(m)
\end{aligned}
\end{equation*}
is convex and  lower semi-continuous in \(\Rr^d \times \Rr^+_0\) and bounded from below by \(-C_0(1+ |m|)\) 
for all \((p,m)\in \Rr^d \times \Rr^+_0\). Consequently,  Proposition~\ref{slscLq}
is an immediate consequence of \cite[Theorem~5.14]{FoLe07}.
\end{proof}

As a simple corollary to the previous
proposition, we obtain the following
lower semi-continuity result on \(\bar
J\).

\begin{corollary}\label{barJwlsc}
Let \(q, \, r, r'\geq
1\) be such that 
\(\frac{1}{r} + \frac{1}{r'}= 1\). Then,
the functional  
$\bar J$   given by \eqref{nvpnew} with
\(V\in L^{r'}(\Tt^d)\) is
sequentially, weakly lower semi-continuous
in \(W^{1,q}(\Tt^d)
\times 
L^r(\Tt^d;\Rr^+_0)\).
\end{corollary}

\begin{proof} 
We observe first that
the functional
\begin{equation*}
\begin{aligned}
m\mapsto \int_{\Tt^d} V(x) m(x)\, dx
\end{aligned}
\end{equation*}
is continuous with respect to the weak convergence
in \(L^r(\Tt^d)\) because \(V\in L^{r'}(\Tt^d)\). To conclude,  we invoke
Proposition~\ref{slscLq} and recall that because $\Tt^d$ is compact, sequential weak lower semi-continuity in \(L^1 \times L^1\) implies sequential weak lower semi-continuity
in \(L^q \times L^r\).    
\end{proof}

Next, we prove the lower semi-continuity in the sense
of measures of
the first integral term in \(\bar J\). This result
will be useful to proving existence
of solutions to \eqref{mmz} when \(\alpha=\gamma\).

We recall that   if \((X,\mathfrak{M})\)
is a measurable space and \(\vartheta:
\mathfrak{M} \to \Rr^n\) is a vectorial
measure, then the total variation of
\(\vartheta\) is the measure \(\Vert
\vartheta\Vert: \mathfrak{M} \to[0,\infty)\)
defined for all \(E\subset \mathfrak{M}\)
by
\begin{equation}\label{totalvar}
\begin{aligned}
\Vert
\vartheta\Vert(E)= \sup \bigg\{ \sum_{i=1}^\infty
|\vartheta(E_i)|\!: \, \{E_i\}\subset\mathfrak{M}  \text{ is a partition of } E \bigg\}.
\end{aligned}
\end{equation}
Moreover, given \(\tilde E \in \mathfrak{M} \), we denote by \(\vartheta\lfloor\tilde
E\)  the restriction of \(\vartheta\) to \(\tilde {E}\),  which is the measure given by \(\big(\vartheta\lfloor\tilde
E\big)(E)=\vartheta(E\cap \tilde E)\) for \(E\in \mathfrak{M}\).

\begin{remark}\label{ontotalvar}
Observe that if \((X,\mathfrak{M})\)
is a measurable space and \(\vartheta:
\mathfrak{M} \to \Rr^n\) is a vectorial
measure, then 
\begin{equation*}
\begin{aligned}
\frac{1}{n}\sum_{j=1}^n \Vert \vartheta_j
\Vert (E) \leq \Vert 
\vartheta\Vert(E) \leq \sum_{j=1}^n \Vert \vartheta_j
\Vert (E)
\end{aligned}
\end{equation*}
for all \(E\subset \mathfrak{M}\), where
each \(\Vert \vartheta_j
\Vert (E)\) is given by \eqref{totalvar}
(with \(n=1\)).
\end{remark}

In what follows, \(\mathcal{L}^d\) stands for the \(d\)-dimensional
Lebesgue
 measure.

\begin{proposition}\label{slscM}
Let \(\bar f\) be the
function given by \eqref{barf} with
\(\alpha = \gamma\). If
 \((v_1^n,v_2^n)_{n\in\Nn}
\subset L^1(\Tt^d;\Rr^d) \times
L^1(\Tt^d; \Rr^+_0)\) and \(\vartheta\in
\Mm(\Tt^d;\Rr^d \times \Rr^+_0) \) are such that 
\begin{equation*}
\begin{aligned}
(v_1^n,v_2^n) \mathcal{L}^d\lfloor{\Tt^d}
\weaklystar \vartheta \enspace \text{weakly-$\star$
in } \Mm(\Tt^d;\Rr^d \times \Rr^+_0),
\end{aligned}
\end{equation*}
then
\begin{equation*}
\begin{aligned}
\liminf_{n\to\infty} \int_{\Tt^d} \bar
f(v_1^n(x),v_2^n(x))\, dx \geq \int_{\Tt^d}
\bar
f\left( \frac{d\vartheta}{d \mathcal{L}^d}
(x)\right) dx + \int_{\Tt^d}
\bar
f^\infty\left( \frac{d\vartheta_s}{d \Vert\vartheta_s\Vert}
(x)\right) d \Vert\vartheta_s\Vert(x),
\end{aligned}
\end{equation*}
where
\(\vartheta = \frac{d\vartheta}{d \mathcal{L}^d}
 \mathcal{L}^d\lfloor\Tt^d +\vartheta_s
 \) is the Lebesgue--Besicovitch decomposition of
 \(\vartheta\) with respect to the \(d\)-dimensional
Lebesgue
 measure, \(\frac{d\vartheta_s}{d \Vert\vartheta_s\Vert}\)
 is the Radon--Nikodym derivative of
\(\vartheta_s\) with respect to its total
variation, and  \(\bar f^\infty:
 \Rr^d \times \Rr^+_0 \to [0,\infty]
 \) is the recession function of \(\bar
 f\); that is, for \((p,m)\in\Rr^d \times
\Rr^+_0\),
\begin{equation}\label{barfrec}
\begin{aligned}
\bar f^\infty (p,m) = \begin{cases}
\frac{|p|^\gamma}{\gamma
(\gamma-1) m^{\gamma-1}} & \text{if
} m\not=0,\\
+\infty & \text{if } m=0 \text{ and
} p\not= 0,\\
0 &\text{if }  m=0 \text{ and
} p= 0.
\end{cases}
\end{aligned}
\end{equation}\end{proposition}

\begin{proof} 
In view of Remark~\ref{ontotalvar},
 \cite[Theorem~5.19]{FoLe07}
holds when we consider the total
variation defined by  \eqref{totalvar}
(compare with \cite[Definition~1.183]{FoLe07}).
 To conclude, it suffices to use \cite[Theorem~4.70]{FoLe07}
to characterize the recession function
of \(\bar f\), taking into account that \(\bar f\) 
is proper, convex, and lower semi-continuous on \(\Rr^d\times \Rr^+_0\). Thus, we
obtain
\begin{equation*}
\begin{aligned}
\bar f^\infty (p,m) = \lim_{t\to\infty}
\frac{\bar f((-P,0)+t(p,m)) - \bar f(-P,0)}{t},
\end{aligned}
\end{equation*}
from which we derive \eqref{barfrec}.
\end{proof}

\subsection{Existence of solutions
}\label{secexist}

Here, we examine the existence of solutions
to the minimization problem \eqref{mmz}.
From Theorems~\ref{thm:exist1}
and \ref{thm:exist2} below, it follows that
 for all \(1< \alpha\leq\gamma\), 
 there exists a solution to this problem.

\begin{theorem}\label{thm:exist1}
Assume that  \(V\in L^\infty(\Tt^d)\) and  
\begin{itemize}
\item[($\mathcal{G}$1)] \(G\) is coercive; that is,  \(\displaystyle\lim_{z\to+\infty}\frac{G(z)}{z} = +\infty\).
\end{itemize}
Then, the minimization problem \eqref{mmz} has a solution
\((u,m) \in \Aa_{
\gamma/\alpha,1}\) for all
\(1<\alpha< \gamma\). Moreover, if 
\begin{itemize}
\item[($\mathcal{G}$2)] there exist positive constants, \(\theta\) and \(C\), such that
\(\displaystyle  G(z)\geq\frac1C z^{\theta+1} - C
 \text{ for all } z>0 ,
\)
\end{itemize}
then the minimization problem \eqref{mmz} has a solution
\((u,m) \in \Aa_{
{\gamma(1+\theta)}/{(\alpha + \theta)},1+\theta}\) for all 
\(1<\alpha\leq\gamma\). \end{theorem}

\begin{proof}
We start by observing that for any  $q\geq 1$ and \(r\geq 1\), we have   
\begin{equation}
\label{mmzfinite}
-\sup_{\Tt^d} V + G(1) \leq \inf_{(u,m)\in \Aa_{q,r}} \bar J[u,m] \leq \frac{|P|^\gamma}{\gamma
(\alpha-1) } + \Vert V\Vert_{L^1(\Tt^d)} + G(1), 
\end{equation}
using the condition \(\int_{\Tt^d} m \, dx =1\),
 the convexity of \(G\) together with Jensen's inequality,
 and the non-negativeness of \(\bar f\)
to obtain the lower bound; to obtain the upper bound, we use  \(u=0\) and \(m=1\) as test functions.

Let  \(1<\alpha\leq\gamma\), and set \(q=\tfrac{\gamma}{\alpha}\)
and \(r=1\).  Let \((u_n,m_n)_{n\in\Nn}\subset  \Aa_{q,1}\)
be an infimizing sequence for \eqref{mmz}; that is, a
sequence \((u_n,m_n)_{n\in\Nn}\subset  \Aa_{q,1}\) such
that
\begin{equation}\label{infseq}
\begin{aligned}
\liminf_{n\to\infty} \bar J[u_n,m_n] = \inf_{(u,m)\in \Aa_{q,1}}
\bar J[u,m].
\end{aligned}
\end{equation}
Extracting a subsequence if necessary, we may assume
that the lower limit on the left-hand side of \eqref{infseq}
is a limit and, in view of \eqref{mmzfinite},
\begin{equation*}
\begin{aligned}
\sup_{n\in\Nn} \left|\bar J[u_n,m_n]\right| \leq C
\end{aligned}
\end{equation*}
for some positive constant \(C\). This estimate, the
condition \(\int_{\Tt^d} m_n \, dx = 1\), and the non-negativeness
of \(\bar f\) and of \(\int_{\Tt^d}\left[G(m_n) - G(1)\right]dx\) yield
\begin{equation}
\label{boundsDum}
\begin{aligned}
\int_{\Tt^d} \bar f(\nabla u_n, m_n)\,dx \leq C+|G(1)|+ \Vert V\Vert_\infty
\enspace \text{ and } \int_{\Tt^d}  G(m_n)\,dx \leq C+2|G(1)|+ \Vert V\Vert_\infty
\end{aligned}
\end{equation}
for all \(n\in\Nn\). Recalling the definition of \(\bar
f\), the first condition in \eqref{boundsDum} implies
that \(m_n>0\) a.e.\! in \(U_n=\{x\in\Tt^d\!: \, \nabla u_n\not=-P\}\) and 
\begin{equation*}
\begin{aligned}
\int_{\Tt^d} \bar f(\nabla u_n, m_n)\,dx = \int_{U_n}
\frac{|P+\nabla u_n|^\gamma}{\gamma
(\alpha-1) m_n^{\alpha-1}}  \, dx \leq C+|G(1)|+ \Vert V\Vert_\infty.
\end{aligned}
\end{equation*}
Recalling that \(q=\tfrac{\gamma}{\alpha}\)
and using the preceding estimate and  Young's inequality,
we obtain
\begin{equation}
\label{Dun1}
\begin{aligned}
\int_{\Tt^d} |P+\nabla u_n|^q \, dx&= \int_{U_n} |P+\nabla u_n|^{\frac{\gamma}{\alpha}} \, dx = \int_{U_n} \frac{|P+\nabla u_n|^{\frac{\gamma}{\alpha}}}{m_n^{\frac{\alpha-1}{\alpha}}}
m_n^{\frac{\alpha-1}{\alpha}}\, dx\\ 
& \leq \frac{1}{\alpha}
\int_{U_n} 
\frac{|P+\nabla u_n|^\gamma}{ m_n^{\alpha-1}}  \, dx + \frac{\alpha-1}{\alpha}
\int_{\Tt^d} m_n \, dx \\
&\leq \frac{\alpha-1}{\alpha} \left[\gamma
(C+|G(1)|+ \Vert V\Vert_\infty ) + 1\right].
\end{aligned}
\end{equation}
Assume now that \(G\) satisfies
($\mathcal{G}$1) and that   \(\alpha<\gamma\).  Note that \(q>1\) in this case. 
 Extracting a subsequence if necessary,  from \eqref{Dun1} together with Poincar\'e--Wirtinger's inequality and
from the
second estimate in  \eqref{boundsDum} together with ($\mathcal{G}$1) and  De la Vall\'ee Poussin's criterion, there exists 
\((\bar u,\bar m) \in \Aa_{
q,1}\) such that
\begin{equation*}
\begin{aligned}
u_n \rightharpoonup \bar u \text{ weakly in } W^{1,q}(\Tt^d)
\enspace \text{ and }\enspace  m_n \rightharpoonup \bar m \text{ weakly in } L^1(\Tt^d).
\end{aligned}
\end{equation*}
Thus, invoking Corollary~\ref{barJwlsc}, we have
\begin{equation*}
\begin{aligned}
\inf_{(u,m)\in
\Aa_{q,1}}
\bar J[u,m] \leq \bar J[\bar u,\bar m] &\leq\liminf_{n\to\infty}
 \bar J[u_n,m_n] = \inf_{(u,m)\in
\Aa_{q,1}}
\bar J[u,m].
\end{aligned}
\end{equation*}
Hence, recalling that \(q=\gamma/\alpha\), we conclude that
\((\bar u,\bar m) \in \Aa_{
\gamma/\alpha,1}\) satisfies
\begin{equation*}
\begin{aligned}
\bar J[\bar u,\bar m] = \min_{(u,m)\in
\Aa_{\gamma/\alpha,1}}
\bar J[u,m].
\end{aligned}
\end{equation*}

Assume now that  ($\mathcal{G}$2) holds. Then, in
particular, ($\mathcal{G}$1) holds.
Let  
\(1<\alpha\leq\gamma\), and set
\(q=\gamma({1+\theta})/({\alpha+\theta})\)
and \(r=1+\theta\); note that \(q>\tfrac{\gamma}{\alpha}\geq
1\). 
Let  \((u_n,m_n)_{n\in\Nn}\subset
 \Aa_{q,r}\)
be an infimizing sequence for \eqref{mmz}. Arguing as above, we conclude
that \eqref{boundsDum}
holds. Hence,
using the definition of \(\bar f\) and ($\mathcal{G}$2), we obtain%
\begin{equation}\label{betterDunmn}
\begin{aligned}
\sup_{n\in\Nn}  \int_{U_n}
\frac{|P+\nabla u_n|^\gamma}{ m_n^{\alpha-1}}  \, dx <\infty\enspace \text{
and } \enspace\int_{\Tt^d}  m_n^{\theta+1}\,dx
<\infty,
\end{aligned}
\end{equation}
where, as before,  \(U_n=\{x\in\Tt^d\!: \, \nabla u_n\not=-P\}\)
and  \(m_n>0\) a.e.\! in \(U_n\).
     
Set \(a=\frac{(\alpha-1)(1+\theta)}{\alpha+\theta}\),
\(b=\frac{\alpha+\theta}{1+\theta}\),
and \(b'=\frac{b}{b-1} = \frac{\alpha+\theta}{\alpha-1}\).
Note that \(b,\, b'>1\), \(ab = \alpha-1\),
\(qb = \gamma\),
and \(ab'=r\). Then, arguing as in \eqref{Dun1}
and using \eqref{betterDunmn},
we obtain
\begin{equation*}
\begin{aligned}
\sup_{n\in\Nn} \int_{\Tt^d} |P+\nabla u_n|^q \, dx &= \sup_{n\in\Nn} \left(
\int_{U_n} \frac{|P+\nabla u_n|^{q}}{m_n^a} 
m_n^{a}\, dx\right) \\ 
&\leq \sup_{n\in\Nn}
\left( \frac{1}{b}
\int_{U_n} 
\frac{|P+\nabla u_n|^{qb}}{ m_n^{ab}}
 \, dx + \frac{1}{b'}
\int_{\Tt^d} m_n^{ab'} \, dx \right) <\infty.
\end{aligned}
\end{equation*}
Reasoning once more as in the preceding
case, we conclude that there exists \((\bar u,\bar m) \in \Aa_{
{\gamma(1+\theta)}/{(\alpha + \theta)},1+\theta}\) satisfying 
\begin{equation*}
\begin{aligned}
\bar J[\bar u,\bar m] = \min_{(u,m)\in
\Aa_{
{\gamma(1+\theta)}/{(\alpha + \theta)},1+\theta}}
\bar J[u,m]. 
\end{aligned} \qedhere
\end{equation*}

\end{proof}

\begin{remark}\label{onV}
If ($\mathcal{G}$2) holds, then
 Theorem~\ref{thm:exist1} remains valid under the weaker
 assumption \(V\in L^{\frac{\theta+1}{\theta}}(\Tt^d)\)
 with \(\sup_{\Tt^d} V \in \Rr\).
\end{remark}

Before proving the existence of solutions
in the case in which \(\alpha=\gamma\)
and \(G\) satisfies
Assumption~($\mathcal{G}$1) in Theorem~\ref{thm:exist1},
we briefly recall some properties of
the space, \(BV(\Tt^d)\), of functions of bounded variation in \(\Tt^d\). 

We
say that \(u\in BV(\Tt^d) \) if \(u\in
L^1(\Tt^d)\) and its distributional derivative, \(Du\),  belongs to \(\Mm(\Tt^d;\Rr^d)\);
that is, there exists a Radon measure,
\(Du \in \Mm(\Tt^d;\Rr^d) \), such that
for all \(\phi\in C^1(\Tt^d)\), we have
\begin{equation*}
\begin{aligned}
\int_{\Tt^d} u(x) \nabla \phi(x)\, dx
= - \int_{\Tt^d} \phi(x) \, dDu(x).
\end{aligned}
\end{equation*}
The space \(BV(\Tt^d)\) is a Banach
space when endowed with the norm \(\Vert
u \Vert_{BV(\Tt^d)}= \Vert u\Vert_{L^1(\Tt^d)}
+\ \Vert Du\Vert\). Moreover,
we have the following compactness property.
If \((u_n)_{n\in\Nn}\) is such that
\(\sup_{n\in\Nn} \Vert
u_n \Vert_{BV(\Tt^d)} < \infty\), then,
extracting a subsequence if necessary,
there exists \(u\in BV(\Tt^d)\) such
that \((u_n)_{n\in\Nn}\) weakly-\(\star\)
converges to \(u\) in \(BV(\Tt^d)\),
written \(u_n\weaklystar u\) weakly-\(\star\)
in \(BV(\Tt^d)\); that is, \(u_n \to
u \) (strongly) in \(L^1(\Tt^d)\) and
\(Du_n\weaklystar Du\) weakly-\(\star\)
in \(\Mm(\Tt^d;\Rr^d)\).

Given \(u\in BV(\Tt^d)\), the Radon-Nikodym derivative of \(Du\)
with respect to the \(d\)-dimensional
Lebesgue
 measure is denoted by \(\nabla u\)
 and the singular part of the Lebesgue--Besicovitch
decomposition
of
 \(D u\) with respect to the \(d\)-dimensional
Lebesgue
 measure is denoted by \(D^su\);
 thus, 
\begin{equation*}
\begin{aligned}
Du= \nabla u \Ll^d\lfloor \Tt^d + D^s
u
\end{aligned}
\end{equation*}
stands for the Lebesgue--Besicovitch
decomposition
of
 \(D u\) with respect to the \(d\)-dimensional
Lebesgue
 measure. By the Polar decomposition
theorem, we have that \(\big|\frac{dD^s
u}{d
\Vert D^s u\Vert}(x)\big|=1\) for
\(\Vert D^s u\Vert\)-a.e.\! \(x\in\Tt^d\).
 
Finally, we note  that   \(u\in
BV(\Tt^d) \) belongs to 
 \( W^{1,1}(\Tt^d)\) if and only if
 \(D^s u \equiv0\). In this case,   \(Du = \nabla u \Ll^d\lfloor
\Tt^d\), where \(\nabla u\) is the usual (weak) gradient
of \(u\); moreover, in that case, \(\Vert Du\Vert(\Tt^d)
= \int_{\Tt^d} |\nabla u|\,dx\).

\begin{theorem}\label{thm:exist2}
Assume that \(\alpha=\gamma\),   \(V\in L^\infty(\Tt^d)\),
and  \(G\) satisfies Assumption~($\mathcal{G}$1) in Theorem~\ref{thm:exist1}.
Then, the minimization problem \eqref{mmz}
has a solution
\((u,m) \in \Aa_{1,1}\).
\end{theorem}

\begin{proof}
As at the beginning of the proof of
Theorem~\ref{thm:exist1}, let \((u_n,m_n)_{n\in\Nn}\subset
 \Aa_{1,1}\)
be an infimizing sequence for \eqref{mmz};
that is, a
sequence \((u_n,m_n)_{n\in\Nn}\subset
 \Aa_{1,1}\) satisfying \eqref{infseq}.
 Observe that \eqref{boundsDum} and
 \eqref{Dun1} are valid for \(q=1\)
 and \(r=1\). Thus, extracting a subsequence if necessary,
 from \eqref{Dun1} together with Poincar\'e--Wirtinger's
inequality and
from the
second estimate in  \eqref{boundsDum}
together with ($\mathcal{G}$1), there exists 
\((\bar u,\bar m) \in BV(\Tt^d)\times
L^1(\Tt^d)\) such that  
\begin{equation*}
\begin{aligned}
u_n \weaklystar \bar u \text{ weakly-\(\star\)
in }BV(\Tt^d)
\enspace \text{ and }\enspace  m_n \rightharpoonup
\bar m \text{ weakly in } L^1(\Tt^d).
\end{aligned}
\end{equation*}

We claim that \((\bar u,\bar m) \in\Aa_{1,1}
\). We first observe that because \((u_n,m_n)_{n\in\Nn}\subset
 \Aa_{1,1}\),
  the above
weak convergences imply that \(\bar
m \geq 0\) a.e.\! in \(\Tt^d\), \(\int_{\Tt^d}
\bar m\, dx =1\), \(\int_{\Tt^d} \bar
u\,
dx =0\), and \((\nabla u_n,m_n)\mathcal{L}^d\lfloor
\Tt^d \weaklystar (D\bar u,\bar m\mathcal{L}^d\lfloor
\Tt^d) \) weakly-\(\star\) in \(\Mm(\Tt^d;\Rr^d\times\Rr^+_0)\).
We further observe  that the Lebesgue--Besicovitch decomposition
of
 \((D\bar u,\bar m\mathcal{L}^d\lfloor
\Tt^d)\) with respect to the \(d\)-dimensional
Lebesgue
 measure is  
\begin{equation*}
\begin{aligned}
(D\bar u,\bar m\mathcal{L}^d\lfloor
\Tt^d) = (\nabla\bar  u, \bar m )\mathcal{L}^d\lfloor
\Tt^d + (D^s\bar u,0)
\end{aligned}
\end{equation*}
and that \(\Vert (D^su,0)\Vert = \Vert D^su\Vert \). Thus, by Proposition~\ref{slscM},
it follows that
\begin{equation*}
\begin{aligned}
\liminf_{n\to\infty} \int_{\Tt^d} \bar
f(\nabla u_n,m_n)\, dx \geq \int_{\Tt^d}
\bar
f\left( \nabla\bar  u, \bar m\right) dx + \int_{\Tt^d}
\bar
f^\infty\left( \frac{dD^s\bar u}{d
\Vert D^s\bar u\Vert}
,0\right) d \Vert D^s\bar u\Vert(x).
\end{aligned}
\end{equation*}
Because \(\bar f\) and \(\bar f^\infty\)
are nonnegative functions, from the
first
uniform estimate in \eqref{boundsDum},
it follows that   
\begin{equation*}
\begin{aligned}
 \int_{\Tt^d}
\bar
f^\infty\left( \frac{dD^s\bar u}{d
\Vert D^s\bar u\Vert}
,0\right) d \Vert D^s\bar u\Vert(x)
\leq C+|G(1)|+ \Vert V\Vert_\infty.
\end{aligned}
\end{equation*}
In view of the definition of \(\bar
f^\infty\) and the fact that \(\big|\frac{dD^s\bar u}{d
\Vert D^s\bar u\Vert}(x)\big|=1\) for
\(\Vert D^s\bar u\Vert\)-a.e.\! \(x\in\Tt^d\), this last estimate
is only possible if \( \Vert D^s\bar u\Vert \equiv 0\). Hence, also \(D^s\bar
u\equiv
0\). This proves that  \(\bar u\in W^{1,1}(\Tt^d)\).
Thus,  \((\bar u,\bar m) \in\Aa_{1,1}
\) and 
\begin{equation*}
\begin{aligned}
\liminf_{n\to\infty} \int_{\Tt^d} \bar
f(\nabla u_n,m_n)\, dx \geq \int_{\Tt^d}
\bar
f\left( \nabla\bar  u, \bar m\right)
dx.
\end{aligned}
\end{equation*}
Because we also have
\begin{equation*}
\begin{aligned}
\liminf_{n\to\infty} \int_{\Tt^d} \left[
-V m_n +G(m_n)\right] dx \geq \int_{\Tt^d}
\left[-V m +G(m)\right] dx,
\end{aligned}
\end{equation*}
arguing as in the proof of
Theorem~\ref{thm:exist1}, we  conclude that
\((\bar u,\bar m) \in \Aa_{
1,1}\) satisfies
\begin{equation*}
\begin{aligned}
\bar J[\bar u,\bar m] = \min_{(u,m)\in
\Aa_{1,1}}
\bar J[u,m]. 
\end{aligned}\qedhere
\end{equation*}
\end{proof}

\subsection{Uniqueness of minimizers}
\label{uniqmin}
In this subsection, we  study
the uniqueness of solutions 
to the minimization problem \eqref{mmz}.
We show that, in particular,  the solutions provided
by  Theorems~\ref{thm:exist1}
and \ref{thm:exist2} are unique.

\begin{theorem}\label{thm:uniqmin}
Let \(1<\alpha\leq\gamma\) and \(q,r \geq 1\). Assume that \(G\) is strictly
convex in \(\Rr^+_0\) and that \(V\in
L^{\tfrac{r}{r-1}}(\Tt^d)\) is such that \(\sup_{\Tt^d} V \in \Rr\).
Then, there is at most one solution
to \eqref{mmz}.
\end{theorem}

\begin{proof} Assume that \((u_1,m_1),(u_2,m_2) \in \Aa_{q,r} \) are such that
\begin{equation*}
\begin{aligned}
\bar J[u_1,m_1] = \bar J[u_2,m_2] =\min_{(u,m)\in \Aa_{q,r}} \bar J[u,m].
\end{aligned}
\end{equation*}
We want to show that \(u_1=u_2\) and
\(m_1=m_2\) a.e.\! in \(\Tt^d\).

 Due to \eqref{mmzfinite}, we
have that \(j_0:= \min_{(u,m)\in \Aa_{q,r}} \bar J[u,m]\in \Rr\). Then, using H\"older's inequality and Jensen's inequality together with the convexity of \(G\) and the
condition \(\int_{\Tt^d} m_1\, dx =1\),
it follows that 
\begin{equation*}
\begin{aligned}
0\leq\int_{\Tt^d} \bar f (Du_1, m_1)\, dx
\leq j_0 + \Vert V\Vert_{L^{\tfrac{r}{r-1}}(\Tt^d)}
 \Vert m_1\Vert_{L^{r}(\Tt^d)} - G(1).
\end{aligned}
\end{equation*}
Thus, \(\bar f (Du_1, m_1)<\infty\)
a.e.\! in \(\Tt^d\). In particular,
\(Du_1= - P\) a.e.\! in \(\{x\in\Tt^d\!:
\, m_1=0\}\). Similarly, \(\bar f (Du_2, m_2)<\infty\)
a.e.\! in \(\Tt^d\) and 
\(Du_2= - P\) a.e.\! in \(\{x\in\Tt^d\!:
\, m_2=0\}\). 

Set \(u=\tfrac{u_1 + u_2}{2}\)
and \(m=\tfrac{m_1 + m_2}{2}\). In view of the convexity
of the function \((p,m)\in\Rr^d \times
\Rr^+_0 \mapsto \bar
f(p,m) - V(x)m + G(m)\) (see Lemma~\ref{barfcxlsc}),
we have
\begin{equation*}
\begin{aligned}
j_0\leq \bar J[u,m] \leq \frac12 \bar J[u_1,m_1]  + \frac12 \bar J[u_2,m_2] = \frac12 j_0
 + \frac12 j_0 = j_0.
\end{aligned}
\end{equation*}
Consequently, \(\bar J[u,m] = j_0\)
and
\begin{equation*}
\begin{aligned}
0&= \frac12 \bar
J[u_1,m_1]  + \frac12 \bar J[u_2,m_2]
- \bar J[u,m]\\
& = \int_{\Tt^d} \Big(\frac12
\bar f(\nabla u_1, m_1) + \frac12
\bar f(\nabla u_2, m_2) - 
\bar f(\nabla u, m)  + \frac12 G(m_1)
+ \frac12 G(m_2) - G(m) \Big)\, dx.
 \end{aligned}
\end{equation*}
Because of the  convexity of  \((p,m)\in\Rr^d \times
\Rr^+_0 \mapsto \bar
f(p,m)  + G(m)\), the integrand
in the last integral is nonnegative.
Hence,
\begin{equation*}
\begin{aligned}
\frac12
\bar f(\nabla u_1, m_1) + \frac12
\bar f(\nabla u_2, m_2) - 
\bar f(\nabla u, m)  + \frac12 G(m_1)
+ \frac12 G(m_2) - G(m)=0
 \end{aligned}
\end{equation*}
a.e. in \(\Tt^d\). Invoking the convexity
of \(\bar f\) and \(G\) once more, the
previous equality implies that
\begin{equation}\label{uniqbarf}
\begin{aligned}
\begin{cases}
\frac12
\bar f(\nabla u_1, m_1) + \frac12
\bar f(\nabla u_2, m_2) - 
\bar f(\nabla u, m) =0 \\
\frac12 G(m_1)
+ \frac12 G(m_2) - G(m)=0
\end{cases}
\end{aligned}
\end{equation}
a.e.\! in \(\Tt^d\). Because \(G\) is
strictly convex, it follows from the
second identity in \eqref{uniqbarf} that \(m_1
= m_2\)  a.e.\! in \(\Tt^d\). Consequently,
the first identity in \eqref{uniqbarf}
reduces to
\begin{equation*}
\begin{aligned}
\begin{cases}
\displaystyle
\frac{\frac12|P+ \nabla u_1|^\gamma
+ \frac12|P+ \nabla u_2|^\gamma - |P+ \nabla u|^\gamma}{\gamma
(\alpha-1) {m_1}^{\alpha-1} } =0  & \text{a.e.\! in
} \{x\in\Tt^d\!:
\, m_1\not=0\}\\
\displaystyle
\nabla u_1= \nabla u_2 = -P & \text{a.e.\! in
} \{x\in\Tt^d\!:
\, m_1=0\}.
\end{cases}
\end{aligned}
\end{equation*}
Because \(\gamma>1\), we conclude that
\(\nabla u_1= \nabla u_2\) a.e.\! in \(\Tt^d\), which, together with the
condition \(\int_{\Tt^d} u_1\,dx = \int_{\Tt^d} u_2\,dx = 1\), yields \(u_1= u_2\)  a.e.\! in
\(\Tt^d\). 
\end{proof}

As an immediate consequence of Theorem~\ref{thm:uniqmin},
we obtain the following result.

\begin{corollary}\label{Cor:uniqmin}
If, in addition to the hypotheses of
Theorem~\ref{thm:exist1} (respectively,
Theorem~\ref{thm:exist2}), we assume
that \(G\) is strictly
convex in \(\Rr^+_0\), then the solution
to \eqref{mmz} provided by Theorem~\ref{thm:exist1} (respectively,
Theorem~\ref{thm:exist2}) is unique.
\end{corollary}

\subsection{The  \(P=0\) case}\label{varexplicit}
Here,  we further characterize the  solutions of \eqref{mmz} when \(P=0\). 

Assume that \(P=0\) and that \((\bar u, \bar m) \in \Aa_{q,r}\) satisfies%
\begin{equation*}
\bar J[\bar u, \bar m] = \min_{(u,m)\in \Aa_{q,r}} \bar J[u,m].
\end{equation*}
Because this minimum is finite, the definition of \(\bar f\) yields \(\bar m >0\) a.e.\!~in the set \(\{x\in\Tt^d\!: \, \nabla\bar  u \not= 0\}\). Consequently, the inequality \(\bar J[\bar u, \bar m] \leq \bar J[0, \bar m]\) gives
\begin{equation*}
\begin{aligned}
\int_{\{x\in\Tt^d\!: \, \nabla\bar  u \not= 0 \}} \frac{|\nabla \bar u|^\gamma}{\gamma
(\alpha-1) \bar m^{\alpha-1}} \, dx \leq 0.
\end{aligned}
\end{equation*}
This last estimate is possible only if the set \(\{x\in\Tt^d\!: \, \nabla\bar  u \not= 0\}\) has zero measure. Therefore, we conclude that \(\nabla \bar u = 0\) a.e.\!~in \(\Tt^d\), which, together with the restriction \(\int_{\Tt^d} \bar u \, dx = 0\), implies that \(\bar u =0\) a.e.\!~in \(\Tt^d\). Moreover,
we have%
\begin{equation*}
\begin{aligned}
\bar J[0, \bar m] = \min_{(u,m)\in \Aa_{q,r}} \bar J[u,m] \leq \inf_{m\in L^r(\Tt^d), m\geq 0, \atop \int_{\Tt^d} m \, dx =1} \bar J[0,m]  \leq \bar J[0, \bar m].
\end{aligned}
\end{equation*}
Thus, \(\bar m\) satisfies
\begin{equation}\label{mwhenP0}
\begin{aligned}
\tilde J [\bar m] = \min_{m\in
L^r(\Tt^d), m\geq 0, \atop \int_{\Tt^d} m \, dx =1} \tilde J[m], 
\end{aligned}
\end{equation}
where
\begin{equation*}
\begin{aligned}
\tilde J[m] =\int_{\Tt^d} \big(
-V m +G(m)\big)\,
dx. 
\end{aligned}
\end{equation*}

Furthermore, we observe that if $G \in C^{\infty}(\Rr^+)\cap C(\Rr_0^+)$ is coercive and strictly convex then \eqref{mwhenP0} can be solved explicitly. More specifically, the solution, $\bar{m}$, is given by
\begin{equation}\label{eq:explweaksolHgeneral}
	\bar{m}(x)=(G^*)'(V(x)-\Hh),\quad x\in \Tt^d,
\end{equation}
where $\Hh \in \Rr$ is the unique number such that  \(\int_{\Tt^d} \bar m\, dx =1\), and $G^*$ is the Legendre transform of $G$; that is, $G^*(q)=\sup\limits_{m\geq 0} \{q\cdot m -G(m)\}$.

Note that if $G'(0+)=-\infty$ then $(G^*)'(q)>0$ for $q\in \Rr$, and $\bar{m}(x)>0,~x\in\Tt^d$ independently of $V$ and $\Hh$. Thus, the triplet $(0,\bar{m},\Hh)$ defined by \eqref{eq:explweaksolHgeneral} is the unique classical solution of \eqref{main2}.

Alternatively, if $G'(0+)>-\infty$ then $(G^*)'(q)\geq 0$ for $q\in\Rr$, with equality if and only if $q\leq G'(0)$. Therefore, $\bar{m}$ defined by \eqref{eq:explweaksolHgeneral} may vanish at some points for a particular choice of $V$. More precisely, $\bar{m}(x)=0$ at points $x\in \Tt$ for which $V(x) - \Hh\leq G'(0)$.

Next, we take \(G(m)=\frac{m^2}{2}\) and calculate the solution, $\bar{m}$, of \eqref{mwhenP0} using \eqref{eq:explweaksolHgeneral}. We use this solution in
Section~\ref{num}  to validate our numerical method. For this coupling, \(G\), we have that $G^*(q)=\frac{(q^+)^2}{2},~q\in \Rr$. Therefore, according to \eqref{eq:explweaksolHgeneral}, we obtain
\begin{equation}\label{eq:explweaksolH}
\begin{aligned}
\bar m(x) = (V(x) - \overline H)^+,
\end{aligned}
\end{equation}
where \(\overline H \in \Rr\) is such that \(\int_{\Tt^d} \bar m\, dx =1\). In particular, if  \(\inf_{\Tt^d} V \geq \int _{\Tt^d}  V\, dx -1\), then 
\begin{equation*}
\begin{aligned}
\bar m(x) = V(x) + 1 - \int _{\Tt^d}  V\, dx.
\end{aligned}
\end{equation*}
Moreover, if  \(\inf_{\Tt^d} V > \int _{\Tt^d}  V\, dx -1\), then \[(\bar u, \bar m,\overline H)= \bigg(0,V+1-\int _{\Tt^d}  V\, dx, -1+\int _{\Tt^d}  V\, dx\bigg)\] is  the unique classical solution of \eqref{main2} in light of Corollary~\ref{Cor:uniqmin}.

\section{The two-dimensional case}
\label{2dcase}

The variational approach considered in the preceding section requires $1<\alpha\leq \gamma$.  However, in the two-dimensional case, if \(0<\alpha<1\) and \(\gamma> 1\), we can use properties of divergence-free vector fields to deduce an
equivalent equation for which our variational method can be applied. 

In this section, the space dimension is always $d=2$. Moreover,  given \(Q=(q_1, q_2) \in\Rr^2\), we set \(Q^\perp=(-q_2,q_1)\).


From the second equation in \eqref{main}, there exists a constant vector, $Q\in \Rr^2$, and a scalar function, $\psi$, such that
\begin{equation}\label{divfree1}
m^{1-\alpha}|P+Du|^{\gamma-2}(P+Du)=      Q^\perp+(D\psi)^\perp.
\end{equation}
Consequently,
\begin{equation*}
m^{1-\alpha}|P+Du|^{\gamma-1}=|Q+D\psi|.
\end{equation*}
Raising the prior expression to the power $\gamma'$, where $\gamma'=\frac{\gamma}{\gamma-1}$, and rearranging the terms, we obtain
\[
\frac{|P+Du|^{\gamma}}{m^\alpha}=\frac{|Q+D\psi|^{\gamma'}}{m^{\alpha-(\alpha-1)\gamma'}}.
\] 
Therefore,
\[
\frac{|P+Du|^{\gamma}}{\gamma m^\alpha}+V(x)-g(m)-\overline{H}
=
\frac{|Q+D\psi|^{\gamma'}}{\gamma m^{\talpha}}+V(x)-g(m)-\overline{H}, 
\]
with
\[
\talpha=\alpha-(\alpha-1)\gamma'.
\]
Moreover, from \eqref{divfree1}, we have 
\begin{equation}\label{eq:PperpDuPerp}
P^\perp+(Du)^\perp=
m^{1-\talpha} |Q+D\psi|^{\gamma'-2} (Q+D\psi).
\end{equation}
Accordingly, 
\[
\div(m^{1-\talpha} |Q+D\psi|^{\gamma'-2} (Q+D\psi) )=0.
\]

Thus, \eqref{main} can be rewritten as
\[
\begin{cases}
\frac{|Q+D\psi|^{\gamma'}}{\gamma'm^{\talpha}}+\frac{\gamma}{\gamma'}V(x)- \frac{\gamma}{\gamma'}g(m)=\frac{\gamma}{\gamma'}\overline{H}\\
\div(m^{1-\talpha} |Q+D\psi|^{\gamma'-2} (Q+D\psi) )=0.
\end{cases}
\]
Finally, we notice that if $0<\alpha<1$ and $\gamma>1$, we have
$1<\talpha<\gamma'$. That is, we obtain an equation of the form of
\eqref{main} with exponents $\talpha$ and $\gamma'$ in the place of 
$\alpha$ and $\gamma$. Furthermore, $\talpha$ and $\gamma'$ now belong
to the range where our prior results apply.

\section{Transformations of second-order MFGs}
\label{tsom}

Now, we discuss a method to transform second-order MFGs into a scalar PDE. For $\alpha=0$, we recover the Hopf-Cole transformation, used in the context of
MFGs in \cite{LCDF,ll1, GeRC}, and \cite{MR2928382}, for example,
	and further generalized in \cite{MR3377677}. Moreover, we obtain extensions of these transformations for the case $\alpha>0$. We also make connections between these systems and the calculus of variations.
In Section~\ref{soquadH}, we examine problems with a  
quadratic Hamiltonian and with \(\alpha=1\). Then, in Section~\ref{otherH}, we extend our analysis to more general Hamiltonians and any congestion parameter.

\subsection{Quadratic Hamiltonian and $\alpha=1$}\label{soquadH}

Here, we examine  the following elliptic version of \eqref{main} for $\alpha=1$
and $\gamma=2$: \begin{equation}
        \label{mainel}
        \begin{cases}
                -\Delta u+ \frac{|P+Du|^2}{2 m}+V(x)=g(m)+\Hh\\
                -\Delta m -\Delta u=0.   
        \end{cases}
\end{equation}
From the second equation and the periodicity, we get
\[
u=\mu-m 
\]
for some real $\mu$. Accordingly, we replace $u$\ in the first equation in
\eqref{mainel} and obtain
\begin{equation}
        \label{mmm}
        \Delta m+ \frac{|P-Dm|^2}{2m}+V(x)=g(m)+\Hh.
\end{equation}
As we show next, if  $P=0$, the preceding equation is equivalent to the Euler-Lagrange
equation of an integral functional. First, we take $P=0$ and multiply \eqref{mmm}
by $m^{1/2}$. Then, \eqref{mmm} becomes
\[
m^{1/2}\Delta m+ \frac{|Dm|^2}{2m^{1/2}}+V(x)m^{1/2}=m^{1/2}g(m)+\Hh m^{1/2}.
\]
Now, we set $\psi=m^{3/2}$ and conclude that 
\[
\frac 2 3 \Delta \psi+V(x) \psi^{1/3}=g(\psi^{2/3})  \psi^{1/3}+\Hh  \psi^{1/3}.
\]
The foregoing equation is the Euler--Lagrange equation of the functional
\[
\hat J(\psi)=\int_{\Tt^d} \frac{|D\psi|^2}{3}-\frac 3 4 V(x) \psi^{4/3}
+\hat G(\psi)+\frac 4 3\Hh \psi^{4/3}\, dx, 
\]
where, for \(z\geq0\),
\begin{equation*}
\begin{aligned}
\hat G(z) = \int_0^z g(r^{2/3})r^{1/3}\, dr.
\end{aligned}
\end{equation*}

\subsection{Other Hamiltonians}\label{otherH}
Here, we consider the system
\begin{equation}\label{Eq:statcong2}
\begin{cases}
-\Delta u+ m^{\alpha}H\!\left(\frac{Du+P}{m^{\alpha}}\right)=g(m)+\overline{H}-V(x)\\
-\Delta m-\div \left(m D_pH\!\left(\frac{Du+P}{m^{\alpha}}\right)\right)=0,
\end{cases}
\end{equation}
where the Hamiltonian,  $H\colon\Rr^d\to\Rr$, is the Legendre transform of a strictly convex and coercive Lagrangian, $L\colon\Rr^d\to\Rr$; that is, 
\[
H(p)=\sup_{v\in\Rr^d} \{-v\cdot p - L(v)\}.
\]
In the preceding definition,
the maximizer, $v^*$, is given by
\begin{equation}
\label{eq:LTmax}
\begin{aligned}
 p=-D_vL(v^*(p)) \hbox{ and } v^*(p)=-D_pH(p).
\end{aligned}
\end{equation}
Hence,
\begin{equation}
\label{eq:HbyLT}
\begin{aligned}
H(p)=v^*(p)\cdot D_vL(v^*(p))-L(v^*(p)).
\end{aligned}
\end{equation}

Next, we relax \eqref{Eq:statcong2} by replacing $Du$ by a function, $w\colon\Tt^d\to\Rr^d$. Then, 
the second equation in \eqref{Eq:statcong2} can be written as
\begin{equation*}
\div\left(Dm+mD_pH\!\left(\frac{w+P}{m^{\alpha}}\right)\right)=0.
\end{equation*}
Accordingly, we introduce the divergence-free vector field $Q\colon\Tt^d\to\Rr^d $ given by
\[
Q:=-Dm-mD_pH\!\left(\frac{w+P}{m^{\alpha}}\right)\!.
\] 
Therefore, we obtain the system
\begin{equation}\label{Eq:statcong2'}
\begin{cases}
-\div(w)+ m^{\alpha}H\!\left(\frac{w+P}{m^{\alpha}}\right)=g(m)+\overline{H}-V(x),\\
-Dm-mD_pH\!\left(\frac{w+P}{m^{\alpha}}\right)=Q,\\
\div(Q)=0.
\end{cases}
\end{equation}
Note that if  $(u,m, \Hh)$ is a solution of \eqref{Eq:statcong2}, then  $(w, m,\Hh)=(Du, m,\Hh)$ solves \eqref{Eq:statcong2'}. The converse implication does  not necessarily hold and will be discussed in Remark~\ref{Rmk:converse} below. 

Now, we
note that the second equation in  \eqref{Eq:statcong2'} gives
\[
\frac{Dm+Q}{m}=-D_pH\!\left(\frac{w+P}{m^{\alpha}}\right)\!.
\] 
Hence, using the second identity in \eqref{eq:LTmax} with
  $p=\frac{w+P}{m^{\alpha}}$, we obtain  $\frac{Dm+Q}{m}=v^*(\frac{w+P}{m^{\alpha}})$.
Consequently, the first identity in \eqref{eq:LTmax} gives
\[
\frac{w+P}{m^{\alpha}}=-D_vL\!\left(\frac{Dm+Q}{m}\right)
\]
and, in view of \eqref{eq:HbyLT},
\[
 H\!\left(\frac{w+P}{m^{\alpha}}\right)=\frac{Dm+Q}{m}\cdot D_vL\!\left(\frac{Dm+Q}{m}\right)-L\!\left(\frac{Dm+Q}{m}\right)\!.
\]
Using the two preceding identities in \eqref{Eq:statcong2'}, 
we obtain
\begin{equation}\label{Eq:statcong2b}
\begin{cases}
\div \left(m^{\alpha} D_vL\!\left(\frac{Dm+Q}{m}\right)\right)+ m^{\alpha}\left(\frac{Dm+Q}{m}\cdot D_vL\!\left(\frac{Dm+Q}{m}\right)-L\!\left(\frac{Dm+Q}{m}\right)\right)=g(m)+\overline{H}-V(x),\\
w=-P-m^{\alpha}D_vL(\frac {Dm+Q} m),\\
\div(Q)=0.
\end{cases}
\end{equation}

Next, we consider a few cases when the system \eqref{Eq:statcong2b} can be further simplified.

Assume that  $Q=0$ and 
$H(p)=\frac{1}{\gamma}|p|^{\gamma}$.
In this case,  we obtain a solution of \eqref{Eq:statcong2}  by solving \eqref{Eq:statcong2'}. To see this, we observe first that 
$L(v)=\frac{1}{\gamma'}|v|^{\gamma'}$,
where \(\gamma'=\tfrac{\gamma}{\gamma-1}\); consequently, \eqref{Eq:statcong2b} 
becomes
\begin{equation*}
\div\left(m^{\alpha-\gamma'+1}|Dm|^{\gamma'-2} Dm\right)+\frac 1 {\gamma} m^{\alpha-\gamma'}|Dm|^{\gamma'}=g(m)+\overline{H}-V(x).
\end{equation*}
Next, for $\beta$ to be selected later, we use the change of variables $m=\psi^{\beta}$ to get
\[
\beta^{\gamma'-1}\div\left(\psi^{\alpha\beta-\gamma'+1}|D\psi|^{\gamma'-2} D\psi\right)+
\frac {\beta^{\gamma'}} {\gamma} \psi^{\alpha\beta-\gamma'}|D\psi|^{\gamma'}=g(\psi^{\beta})+\overline{H}-V(x).
\]
We rewrite the preceding equation as
\begin{equation*}
\begin{aligned}
&\beta^{\gamma'-1}\psi^{\alpha\beta-\gamma'+1}\div\left(|D\psi|^{\gamma'-2} D\psi\right)+\beta^{\gamma'-1}\left(\alpha\beta-\gamma'+1+ \tfrac {\beta} {\gamma}\right) \psi^{\alpha\beta-\gamma'}|D\psi|^{\gamma'}\\
&\quad=g(\psi^{\beta})+\overline{H}-V(x).
\end{aligned}
\end{equation*}
Now, we choose $\beta$ such that the second term on the left-hand side of
the previous identity vanishes; that is, 
\begin{equation}
\label{eq:betaQ0}
\begin{aligned}
\beta=\frac{\gamma'-1}{\alpha+1/\gamma}
= \frac{\gamma'}{\alpha\gamma + 1}.
\end{aligned}
\end{equation}
Accordingly, we obtain
\begin{equation}\label{eq.pLap}
\beta^{\gamma'-1}\Delta_{\gamma'} \psi=\left(g(\psi^{\beta})+\overline{H}-V(x)\right)\psi^{\frac{\beta}{\gamma}},
\end{equation}
where $\Delta_{p}$ is the $p$-Laplacian operator,  $\Delta_{p} \psi=\div (|D\psi|^{p-2}D\psi)$. We note that \eqref{eq.pLap}
is the Euler-Lagrange equation of the functional
\[
\hat J[\psi]=\int_{\Tt^d} \bigg[\beta^{\gamma'-1}\frac{|D\psi|^{\gamma'}}{\gamma'} +\hat{G}(\psi)-\frac{\gamma}{\beta+\gamma}(V(x)-\overline{H})
{\psi^{\frac{\beta +\gamma}{\gamma}} }\bigg] dx 
\]
where \(\beta\) is given by \eqref{eq:betaQ0}
and  $\hat{G}(z):=\int_0^z g(r^{\beta})r^{\frac{\beta}{\gamma}}\,dr.$ Note  that the unknown $\overline{H}$ is  determined by the constraint 
\begin{equation*}
\begin{aligned}
\int_{\Tt^d} \psi^{\beta}\,dx=1.
\end{aligned}
\end{equation*}

In particular, for $\gamma=2$, we have
\(\beta=\tfrac2{2\alpha +1}\) and \(m=\psi^{\tfrac2{2\alpha +1}}\), where  $\psi$ solves
\begin{equation}\label{Eq:statcongquad}
\tfrac 2 {(2\alpha+1)}\Delta \psi=\Big(g\big(\psi^{\frac 2 {2\alpha+1}}\big)+\overline{H}-V(x) \Big)\psi^{\frac 1 {2\alpha+1}}.
\end{equation}
As before, \eqref{Eq:statcongquad} is the Euler-Lagrange equation of the functional
\[
\hat J[\psi]=\int_{\Tt^d} \bigg[\frac{|D\psi|^2}{2\alpha+1} +\hat{G}(\psi)-\frac{2\alpha+1}{2(\alpha+1)}(V(x)-\overline{H})\psi^{\frac{2(\alpha+1)}{2\alpha+1}}  \bigg] dx \]
where $\hat{G}(z)=\int_0^z g(r^{\frac{2}{2\alpha+1}})r^{\frac{1}{2\alpha+1}}\,dr$,
and $\Hh$ is chosen such that the constraint
\begin{equation*}
\begin{aligned}
\int_{\Tt^d} \psi^{\frac
        2 {2\alpha+1}}\,dx=1
\end{aligned}
\end{equation*}
holds. 

Furthermore, for $\alpha=0$, \eqref{eq.pLap} corresponds to the generalized Hopf-Cole transformation from \cite{MR3377677}.

\begin{remark}        
In the case without congestion, which
corresponds to  $\alpha=0$, and without any restrictions on either  $Q$ or $L$,
 \eqref{Eq:statcong2b} has the form
        \begin{equation*}
        \div \left( D_vL\!\left(\tfrac{Dm+Q}{m}\right)\right)+ \tfrac{Dm+Q}{m}\cdot D_vL\!\left(\tfrac{Dm+Q}{m}\right)-L\!\left(\tfrac{Dm+Q}{m}\right)=g(m)+\overline{H}-V(x).
        \end{equation*}
This equation is the Euler-Lagrange equation of the  functional
        \[
       \tilde J[m]=\int_{\Tt^d} \left[mL\!\left(\frac{Dm+Q}{m}\right)-G(m)-V(x)m \right]dx  
        \]
subjected to the constraint  \(\int_{\Tt^d} m\,dx=1\).     
\end{remark}

        %

\begin{remark}\label{Rmk:converse}
        From a solution $(m, w,\Hh)$ to \eqref{Eq:statcong2b}, we recover a solution to \eqref{Eq:statcong2} if and only if $w=-P-m^{\alpha}D_vL(\frac {Dm+Q} m)$ is a gradient of a function,  $u:\Tt^d \to \Rr$. There are two instances when this holds easily; namely:
        \begin{itemize}
        \item[(i)] Assume that  $d=1$
        and $(m, w,\Hh)$
solves \eqref{Eq:statcong2b} with  $\int_{\Tt} w\,dx =0$. Then,  identifying functions
on \(\Tt\)
with periodic functions on  \([0,1]\)
and setting \(u(x):=\int_0^x w(t)\,dt\),
we conclude that $(m, u,\Hh)$
solves \eqref{Eq:statcong2}.
                  
          \item[(ii)] 
                Assume that $Q\equiv0$,   the Lagrangian is quadratic, and
  $\psi>0$  solves \eqref{Eq:statcongquad}.
In this case,  $w=-P-m^{\alpha}D_vL(\frac {Dm} m)=-P-m^{\alpha-1}Dm$ is a gradient if and only if $P=0$. Then, assuming
further that \(P=0\), setting $m:=\psi^{
\frac{2\alpha +1}{2}}$ and  
$u:=-\frac{m^{\alpha}}{\alpha}+c$,
\(c\in\Rr\), we conclude that  $(u,m,
\Hh)$ solves \eqref{Eq:statcong2}, where
\(\Hh\)  is determined by the constraint
\(\int_{\Tt^d} \psi^{\frac
2 {2\alpha+1}}\,dx=1 \). This transformation of $m$ and $u$ generalizes the well-known Hopf-Cole transform to the congestion case. Finally, note that if \(\alpha=1\), we recover the case treated in Section~\ref{soquadH}.
     \end{itemize}
                However, in general, the condition for $w$ to be a gradient is more restrictive. For instance, assume
that $Q\equiv 0 $ and that the Lagrangian is radial,
$L(v)=l(|v|)$ with $l:\Rr_0^+\to \Rr$ of class $C^2$ and $rl''(r)-l'(r)\neq 0$
for all $r>0$. Then, \[w=-P-m^{\alpha}D_vL\Big(\frac {Dm} m\Big)\] is a gradient only if 
        \[
        0=(w_i)_{x_j}-(w_j)_{x_i}=m^{\alpha}\frac{l'(r)-r l''(r)}{r^3}(v_i (v\cdot v_{x_j})-v_j (v\cdot v_{x_i}))
        \]
        for all $ i,j \in \{1,...,
d\}$, where $r=|Dm|/m$ and $v:=Dm$. Consequently,
        $v_i (v\cdot v_{x_j})-v_j (v\cdot v_{x_i})=0$ for all $ i,j \in\{1,..., d\}$, which implies that there is a
scalar function $\lambda\colon\Tt^d\to
\Rr$ such that $v\cdot v_{x_i}=\lambda v_i$    for all \(i\in\{1,...d\}\). Hence, $m$ must satisfy the identity
        \begin{equation*}
        D^2mDm= \lambda Dm.
        \end{equation*}
        This identity is rather restrictive in higher dimensions; thus, in general, the solutions to \eqref{Eq:statcong2b} may not correspond to solutions to \eqref{Eq:statcong2}.
        
        Finally, note that for radially symmetric Lagrangians, $L(v)=l(|v|)$, the only case when $w$ is automatically a gradient is the case when $r l''(r)=l'(r)$ for $r>0$, which corresponds to the quadratic Lagrangian discussed in (ii) above.
\end{remark}

\section{Numerical solution for the first-order MFGs with congestion, with $1<\alpha\leq\gamma$
}
\label{num}

In this section, we compute and analyze numerical solutions for the variational problem \eqref{mmz}. First, in Section~\ref{sub:discretization},
  we describe the numerical scheme. Then, in Section~\ref{sub:num1d}, we  examine problems in one dimension and discuss the corresponding numerical experiments.
At last, in Section~\ref{sub:num2d}, we perform and discuss our numerical experiments for the two-dimensional case. 

\subsection{Discretization}
\label{sub:discretization}

Here, we detail our numerical scheme. For simplicity, we consider the two-dimensional setting; our methods can easily be adapted to other dimensions. 

We fix an integer, $N\in\Nn$,  and denote by $\Tt_N^2$ the square grid in the two-dimensional torus, $\Tt^2$, with grid size $h=\frac 1 N$. 
 Let $x_{i,j}\in \Tt_N^2$ represent a point with coordinates $(ih,jh)$.
 A grid function is a vector $\varphi  \in\Rr^{N^2}$ whose components,  
 $\varphi_{i,j}$, are determined by the values of a function \(\tilde \varphi:\Tt^2_N \to\Rr\) at \(x_{i,j}\); that is, \(\varphi_{i,j}= \tilde \varphi (x_{i,j})\) for $i,j\in~\{0,\ldots,N-1\}$. Because we are in the periodic setting, we let $\varphi_{i+N,j}=\varphi_{i,j+N}=\varphi_{i,j}$.

For $\varphi\in\Rr^{N^2}$, we define the 5-point stencil, central-difference scheme
\begin{align}\label{diffschemes}
\begin{split}
(D_{1}^h \varphi)_{i,j}&=\frac{-\varphi_{i+2,j}+8\varphi_{i+1,j}-8\varphi_{i-1,j} -\varphi_{i-2,j}}{12h}\\
(D_{2}^h \varphi)_{i,j}&=\frac{-\varphi_{i,j+2}+8\varphi_{i,j+1}-8\varphi_{i,j-1} -\varphi_{i,j-2}}{12h}.
\end{split}
\end{align}
The discrete gradient vector is defined by
\begin{equation}\label{gradCentralDiff}
[D^h\varphi]_{i,j} = \left((D_1^h \varphi)_{i,j},(D_2^h \varphi)_{i,j}\right) \in \Rr^2.
\end{equation}

Let $u,m,V\in \Rr^{N^2}$ be grid functions and $P=(p_1,p_2)\in \Rr^2$ be a point. We
discretize   $\bar f$ in \eqref{barf}  as follows. 
 We define
a function  $f_h:\Rr^{2N^2}\times\Rr^{N^2}\to \Rr^{N^2}$ by setting, for $i,j\in\{0,\ldots,N-1\}$, 
\[
( f_h([D_h u],m))_{i,j}=
\begin{cases}
\frac{1}{m_{i,j}^{\alpha -1}\gamma(\alpha -1)  }
\left(\left(p_1+(D^h_1 u)_{i,j}\right)^2\right.\\
\left.\qquad\qquad +\left(p_2+(D^h_2 u)_{i,j}\right)^2\right)
^{\gamma /2} & \quad \mbox{if $m_{i,j}\neq 0$}\\
\infty &  \quad \mbox{if $m_{i,j}= 0$ and $(D^hu)_{i,j} \neq -P$}\\
0  & \quad \mbox{if $m_{i,j}= 0$ and $(D^hu)_{i,j}=-P$}\\
\end{cases}
\]


Next, we construct a discrete version of $\bar J$, $ J_h:\Rr^{N^2}\times \Rr^{N^2}\to\Rr$, in the following manner. For $u,m\in \Rr^{N^2}$, we define
\[
 J_h(u,m)=h^2\sum_{i,j=0}^{N-1}\big( (f_h([D_h u],m))_{i,j}- V_{i,j}m_{i,j}+G(m_{i,j})\big).
\]
Accordingly, in the discrete setting, the problem \eqref{mmz}  becomes
\begin{equation}\label{mmzdiscrete}
\min_{(u,m)\in \mathcal A_h} J_h(u,m),
\end{equation}
where 
\begin{equation}\label{eq:setAh}
\Aa_h=\bigg\{(u,m)\in \Rr^{2N^2}:h^2\sum_{i,j=0}^{N-1} u_{i,j} =0,h^2\sum_{i,j=0}^{N-1} m_{i,j}=1,m_{i,j}\geq 0\,\,  \forall i,j\in\{0,\ldots,N-1\}\bigg\}.
\end{equation}

\begin{remark}
        In alternative to the 5-point stencil, centered-differences, we can discretize $\bar J$ using monotone finite differences
        \begin{align*}
        (D^h_1 \varphi)^{+}_{i,j}=\frac{\varphi_{i+1,j}-\varphi_{i,j}}{h}, \quad (D^h_2 \varphi)^{+}_{i,j}=\frac{\varphi_{i,j+1}-\varphi_{i,j}}{h}. 
        \end{align*}
        Here, $|P+Du|$ is discretized as follows: 
        \begin{align*}
        |P+D^h \varphi |_{i,j}=&\max(-p_1-(D^h_1 \varphi)^{+}_{i,j},0)+\max(p_1+(D^h_1 \varphi)^{+}_{i-1,j},0)\\&+\max(-p_2-(D^h_2 \varphi)^{+}_{i,j},0)+\max(p_2+(D^h_2 \varphi)^{+}_{i,j-1},0).
        \end{align*}
        However, in our numerical experiments, the results with the 5-point centered difference discretization and the ones with this monotone discretization were similar. 
        Moreover, our numerical tests were faster using \eqref{gradCentralDiff} than using the monotone discretization. Therefore, 
        in our numerical computations we use the central difference
        scheme in \eqref{diffschemes}.
\end{remark}

\subsection{Numerical experiments in one dimension}\label{sub:num1d}

To validate our approach, we start by   considering the one-dimensional version of the discretized variational problem  \eqref{mmzdiscrete} with \(P=0\) and \(G(m) = \frac{m^2}{2}\). In this case, the unique solution of \eqref{mmz} is (see Section~\ref{varexplicit})
\begin{equation}\label{explmmzP0}
\begin{aligned}
(\bar u, \bar m)= \big(0, (V(x) - \overline H)^+\big),
\end{aligned}
\end{equation}
where \(\overline H\) is such that \(\int_{\Tt} \bar m(x)\, dx = 1\). 

As a first example, we choose the congestion
exponent $\alpha=1.5$, $\gamma=2$,
and the potential 
\[
V(x)=\frac 1 2 \cos\left(2\pi \left(x-\frac 1 4\right)\right),
\] 
as shown in Fig.~\ref{fig:plotVval}. 
   Note that $\int_{\Tt}V\,dx=0 $ and \(\inf_{\Tt} V = -\frac{1}{2}\).
As proved in Section~\ref{varexplicit}, \((\bar u,\bar m) = (0, V+1)\) is the minimizer of \eqref{mmz}
and \((\bar u,\bar m, \overline H) = (0, V+1, -1)\) is the classical solution of \eqref{main2}.

The numerical solution (with \(N=200)\) for the density, $m$, and the value function, $u$, are shown, respectively, in Figs.~\ref{fig:plotmval} and \ref{fig:plotuval}. As expected, $u\equiv 0$ and, because the potential does not have regions where it is {too negative}, the graph of \(m\) resembles that of \(V\) everywhere. In Fig.~\ref{fig:plotmvVal}, we depict the absolute error between the numerical solution $m(\cdot)$ and the explicit solution $\bar m(\cdot)=V(\cdot)+1$. The maximum absolute error between these two functions is of order $10^{-8}$.

Next, we slightly modify the potential but in such a way that we are in the borderline case for which \eqref{explmmzP0} does not provide a classical solution 
of \eqref{main2}. More precisely, we consider the potential 
\[
V(x)=\cos\left(2\pi \left(x-\frac 1 4\right)\right)
\] 
as shown in Fig.~\ref{fig:plotVvalb}. 
   In this case, $\int_{\Tt}V\,dx=0 $,  \(\inf_{\Tt} V =V(\frac{3}{4})=-1\), and \((\bar u,\bar m) = (0, V+1)\) is
the minimizer of \eqref{mmz}. Note that \(\bar m (\frac{3}{4})=0\); elsewhere in \(\Tt\), \(\bar m\) is positive.
The numerical solution (with \(N=200)\) for the density, $m$, and the value function, $u$,
are shown, respectively, in Figs.~\ref{fig:plotmvalb} and \ref{fig:plotuvalb}.
As expected, $u\equiv 0$ and, as before, the graph of \(m\) resembles that of \(V\)
everywhere. In Fig.~\ref{fig:plotmvValb}, we depict the absolute
error between the numerical solution $m(\cdot)$ and the explicit solution $\bar
m(\cdot)=V(\cdot)+1$. The maximum absolute error
between these two functions is of order $10^{-6}$.

Finally, concerning the validation of our method in the one-dimensional case, we consider a potential for which \eqref{explmmzP0} is far from being a classical solution of \eqref{main2}. Namely, we consider the potential, \[
V(x)=10\cos\left(2\pi \left(x-\frac 1 4\right)\right),
\] 
shown in Fig.~\ref{fig:plotVvalc}.    
The numerical solution for the density, $m$, and the value function, $u$,
are shown, respectively, in Figs.~\ref{fig:plotmvalc} and \ref{fig:plotuvalc}.
As expected, $u\equiv 0$ and, in the regions where the potential  is not too negative, the graph of \(m\) resembles that of \(V\). 

In Fig.~\ref{fig:plotmvValc}, we depict the absolute
error between the numerical solution $m(\cdot)$ and the explicit solution $\bar
m(\cdot)$ given by \eqref{explmmzP0}. 
To understand the error between \(m\) and \(\bar m\), we performed the numerical simulation for several grid sizes, \(h=\tfrac1N\). In Table~\ref{Table1d}, we present the numerical error and the corresponding running time for \(N\in\{100, 200, 400, 600, 800, 1000\}\). We observe  that the error is linear in grid size and the running time quadratic in the number of nodes,  as expected for a minimization
algorithm where inversions of matrices are used.  
The simulations were performed in a laptop with a 2.5 GHz Intel Core i7 processor, and 16 GB 1600 MHz DDR3 memory.

Next, we show the effect of the  preferred direction, $P$, by plotting the behavior of $u$ and $m$ for $P\in\{0,1,2,3,4\} $ in 
Fig.~\ref{fig:solValidationExp2}. We recall that,  if \(P\not=0\), we are not aware of closed-form solutions of \eqref{mmz}. For the example in  Fig.~\ref{fig:solValidationExp2},  we chose   $V(x)= e^{-(x-\frac{1}{2})^2}$, $G(m) = m^2$,
$\alpha= 1.5$, and $\gamma=2$. We observe that the bigger the \(P\), the flatter the graph of \(m\). In other words, as \(P\) grows, agents prefer to move more and their distribution becomes more homogeneous. This effect was first observed for MFGs without congestion in \cite{Gomes2016b}.

Finally, in Fig.~\ref{fig:solExpr134}, we illustrate the dependence of the solution on the congestion parameter, \(\alpha\). Precisely, we depict the behavior of the solution for the potential \[V(x)=10\sin\left(2\pi\left(x+\frac 1 4\right)\right)\] 
and for  $\alpha\in\{1.001, 1.2, 1.4, 2\},$  $P=1$,  $\gamma=2$, and  \(G(m) = m^3\).  
As expected, the density  resembles the potential in the regions where the potential is not \textit{too negative}. This explains the formation of regions with almost no agents (see Fig.~\ref{fig:plotmls4}).

Note that the congestion exponent, $\alpha$, determines the strength of the congestion effects. For instance, in the regions where the potential is too negative and, therefore, with fewer agents  (see Fig.~\ref{fig:plotmls4}), we see that the higher  the congestion exponent is, the lower  the density value is. However, these differences are compensated in regions where the potential is positive, where we see that the higher  the congestion exponent is, the higher  the density value is. This effect is expected because the mass of the system is preserved.

\begin{figure}[htb!]
        \centering
        \begin{subfigure}[b]{\sizefigure\textwidth}
                \includegraphics[width=\textwidth]{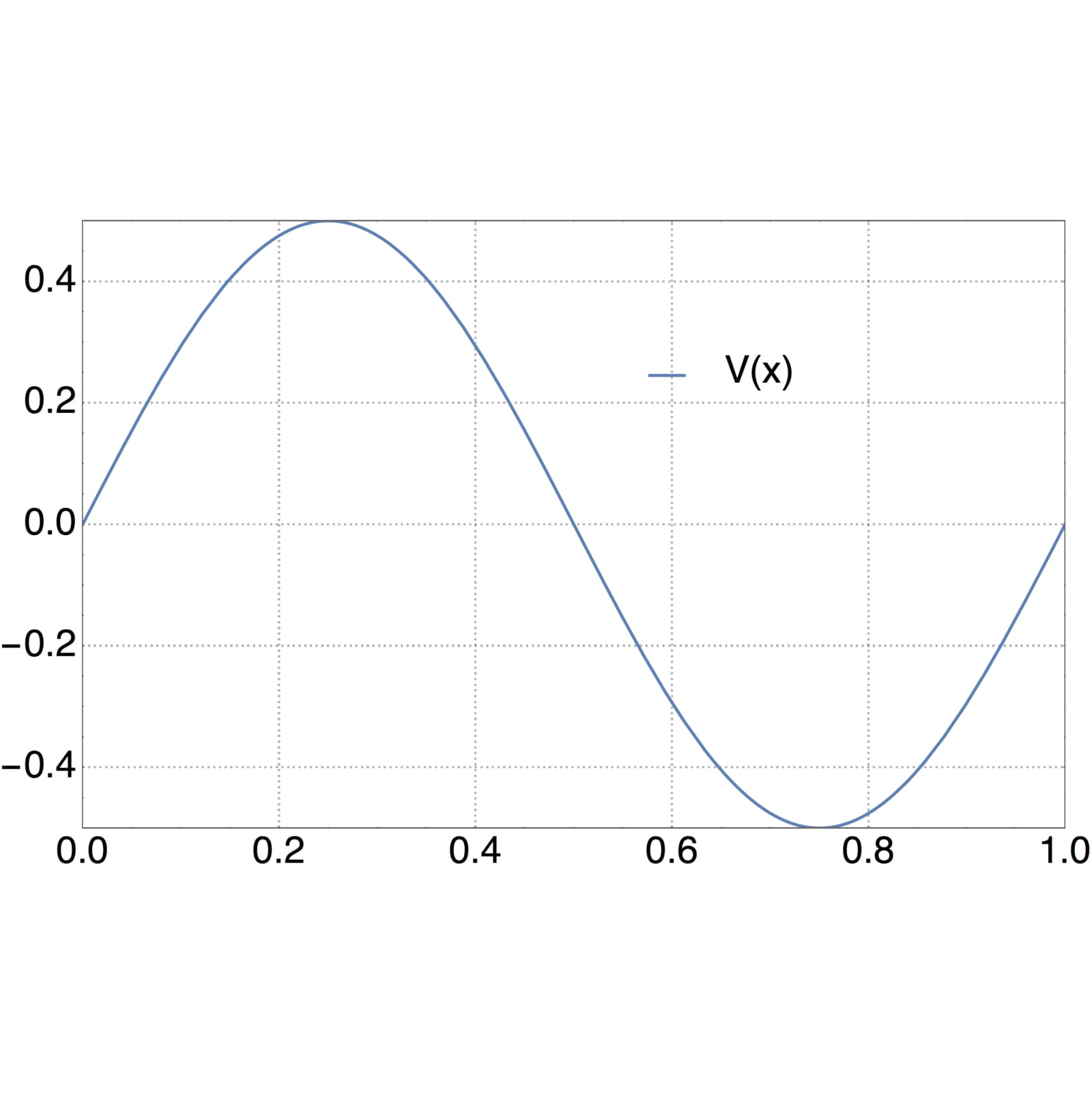}
                \caption{Potential $V$}
                \label{fig:plotVval}
        \end{subfigure}
        \begin{subfigure}[b]{\sizefigure\textwidth}
                \includegraphics[width=\textwidth]{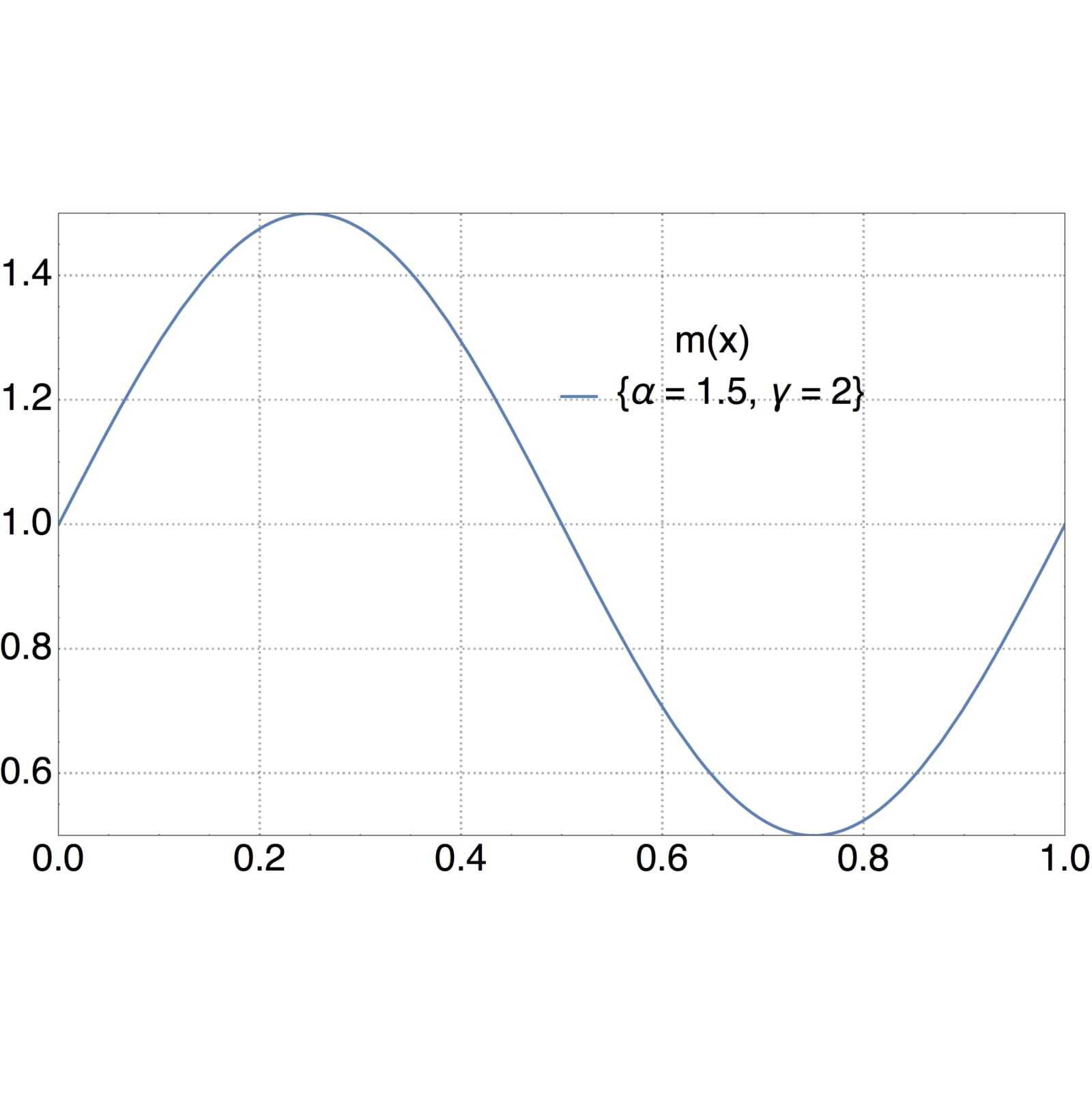}
                \caption{Density $m$}
                \label{fig:plotmval}
        \end{subfigure}\\
        \begin{subfigure}[b]{\sizefigure\textwidth}
                \includegraphics[width=\textwidth]{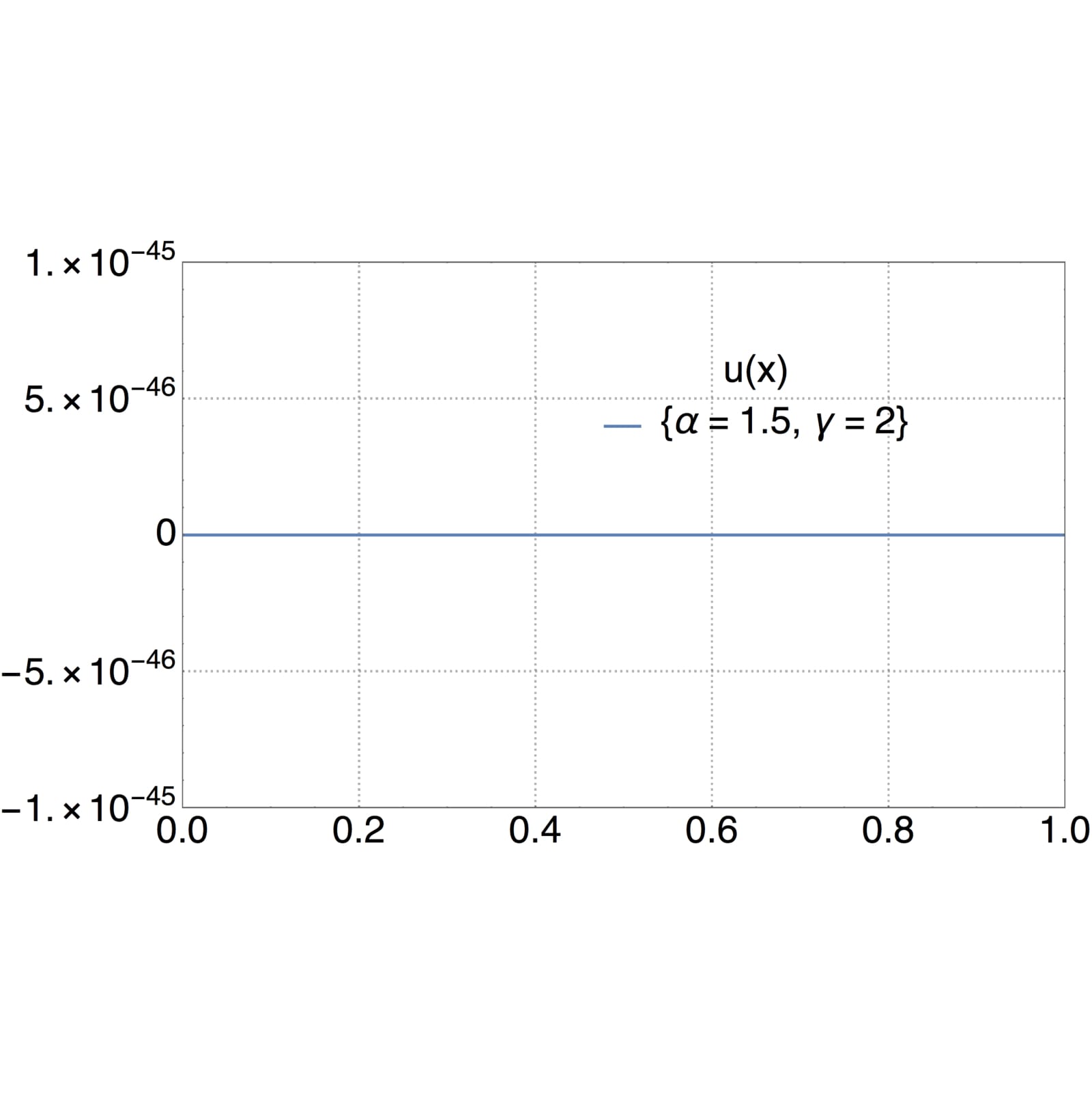}
                \caption{Value function $u$}
                \label{fig:plotuval}
        \end{subfigure}
        \begin{subfigure}[b]{\sizefigure\textwidth}
                \includegraphics[width=\textwidth]{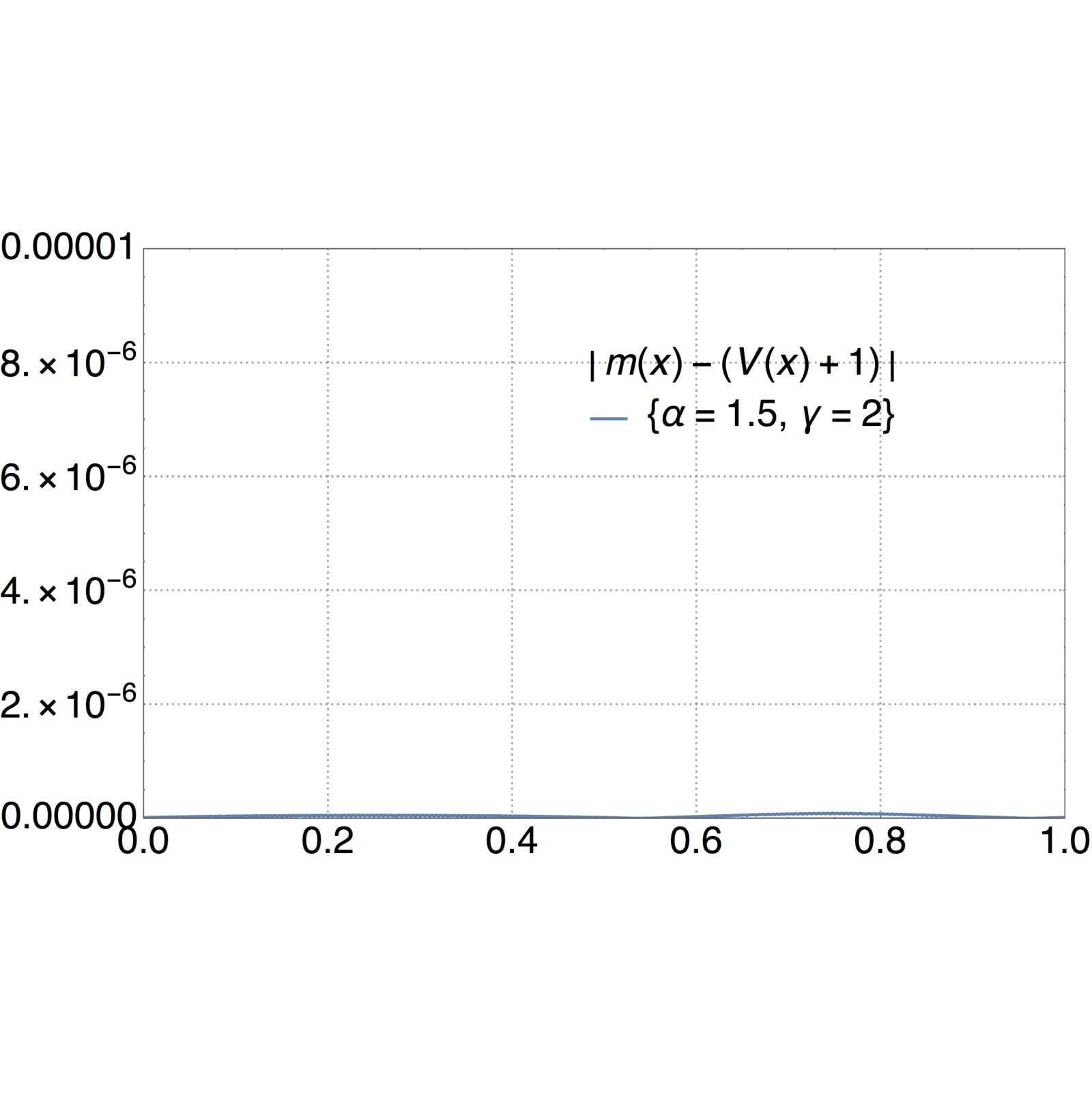}
                \caption{Absolute numerical error}
                \label{fig:plotmvVal}
        \end{subfigure}
        ~ 
        \caption{Numerical solution of the variational problem \eqref{mmz} with \(N=200\) and for \(d=1\),  $P=0$, $V(x)=\frac{1}{2} \cos\left(2\pi \left(x-\frac 1 4\right)\right)$, $G(m) = \frac{m^2}{2}$, $\alpha= 1.5$, and $\gamma=2$.}\label{fig:solValidation}
\end{figure}

\begin{figure}[htb!]
        \centering
        \begin{subfigure}[b]{\sizefigure\textwidth}
                \includegraphics[width=\textwidth]{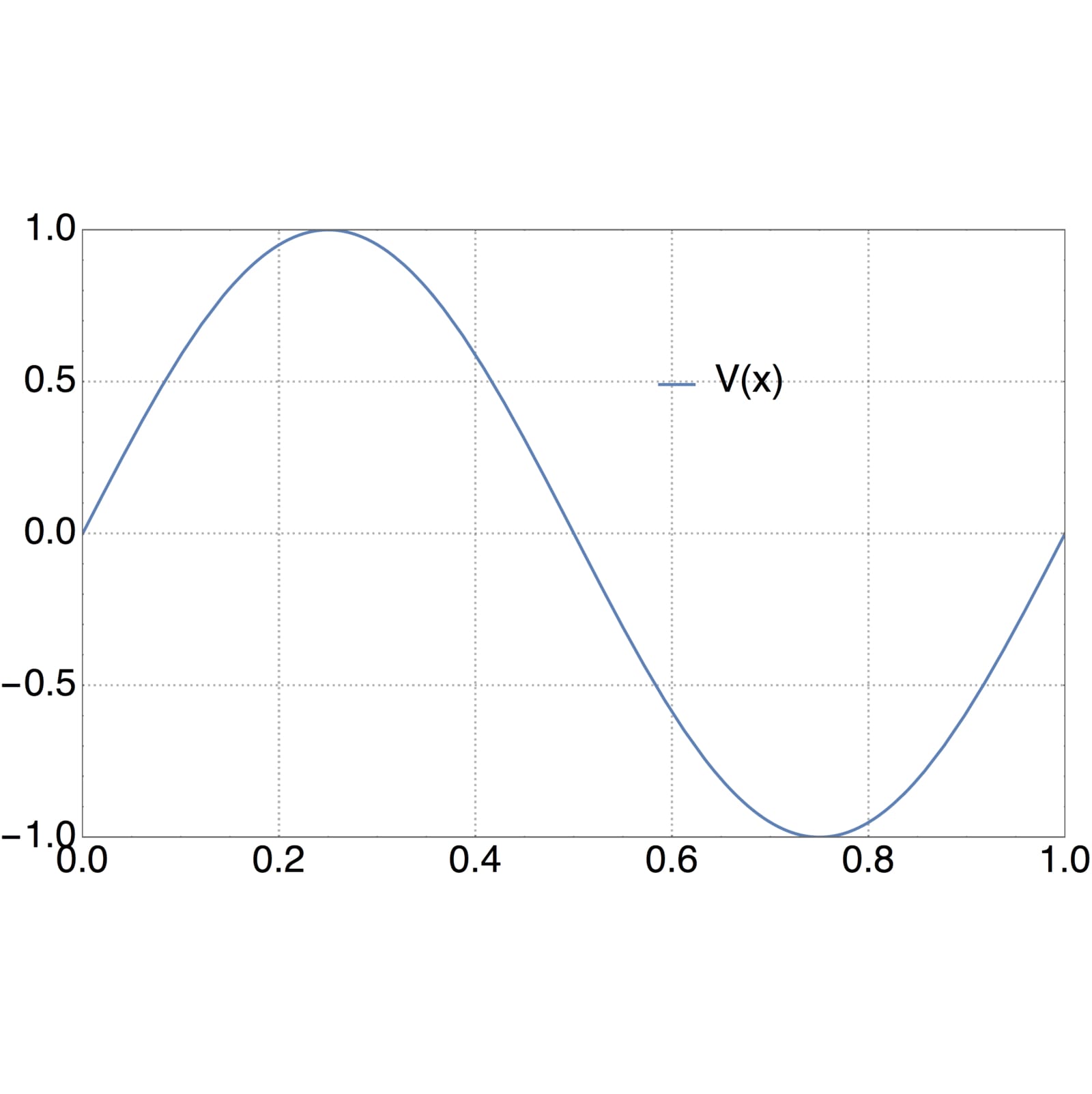}
                \caption{Potential $V$}
                \label{fig:plotVvalb}
        \end{subfigure}
        \begin{subfigure}[b]{\sizefigure\textwidth}
                \includegraphics[width=\textwidth]{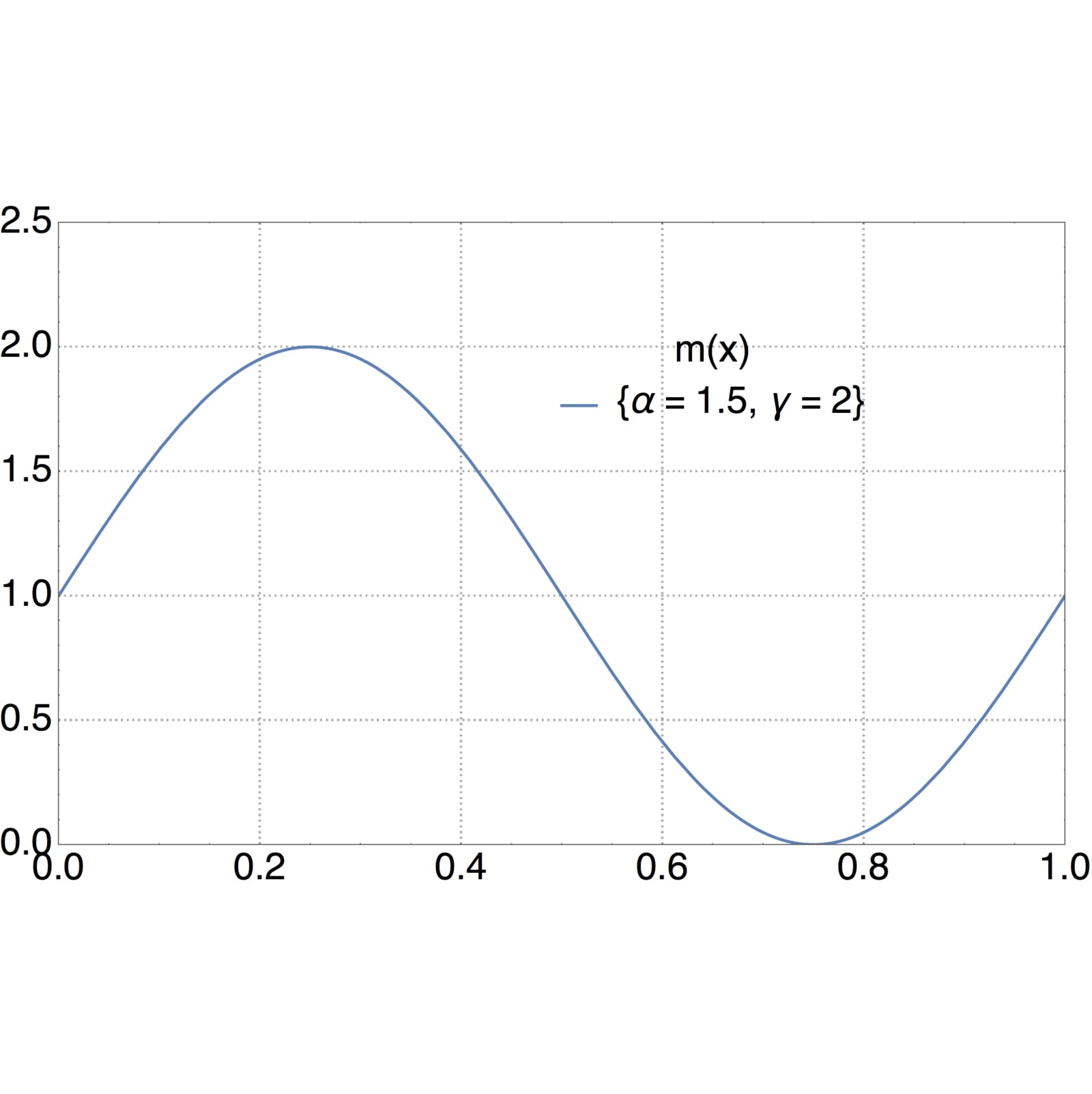}
                \caption{Density $m$}
                \label{fig:plotmvalb}
        \end{subfigure}\\
        \begin{subfigure}[b]{\sizefigure\textwidth}
                \includegraphics[width=\textwidth]{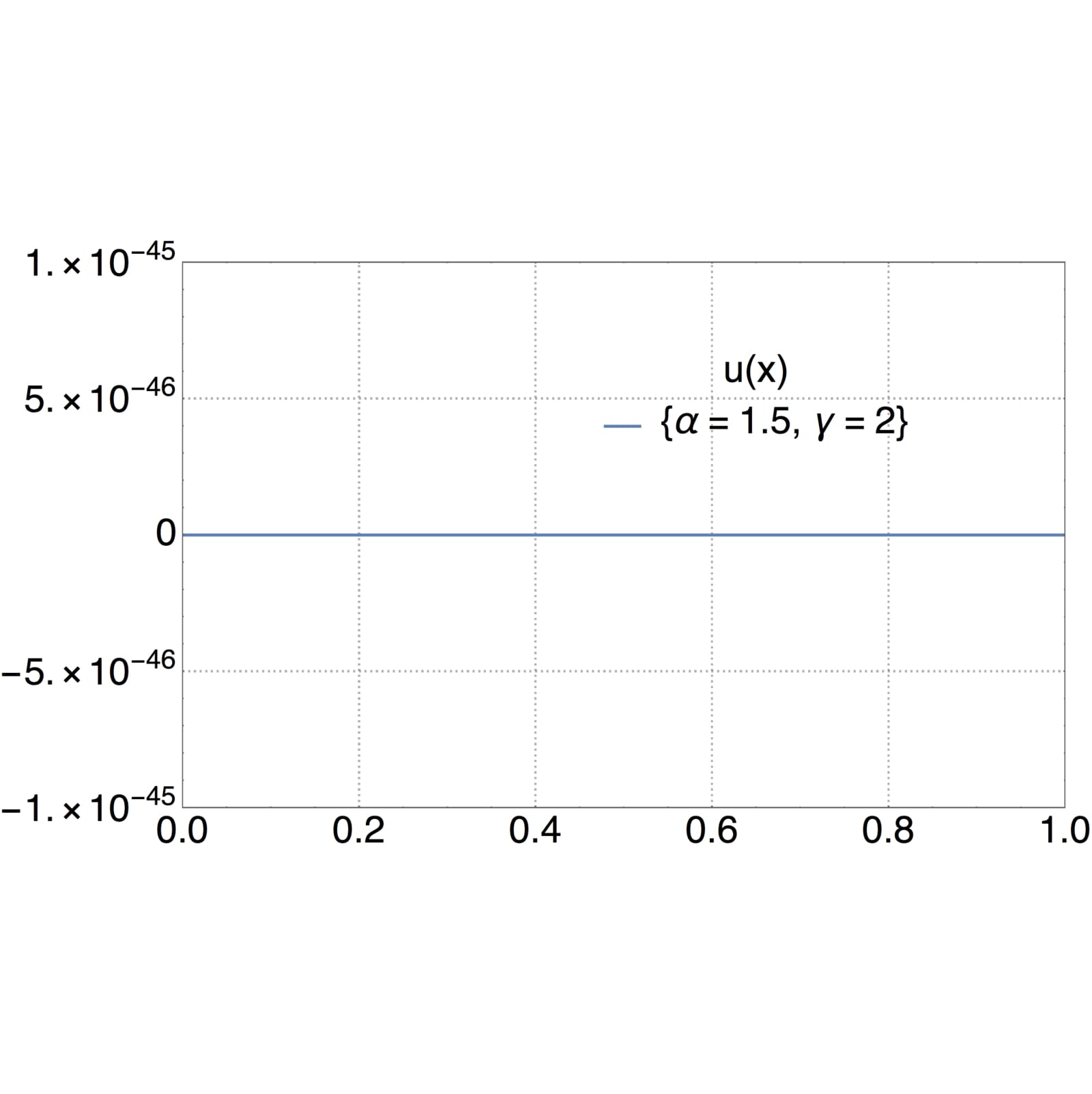}
                \caption{Value function $u$}
                \label{fig:plotuvalb}
        \end{subfigure}
        \begin{subfigure}[b]{\sizefigure\textwidth}
                \includegraphics[width=\textwidth]{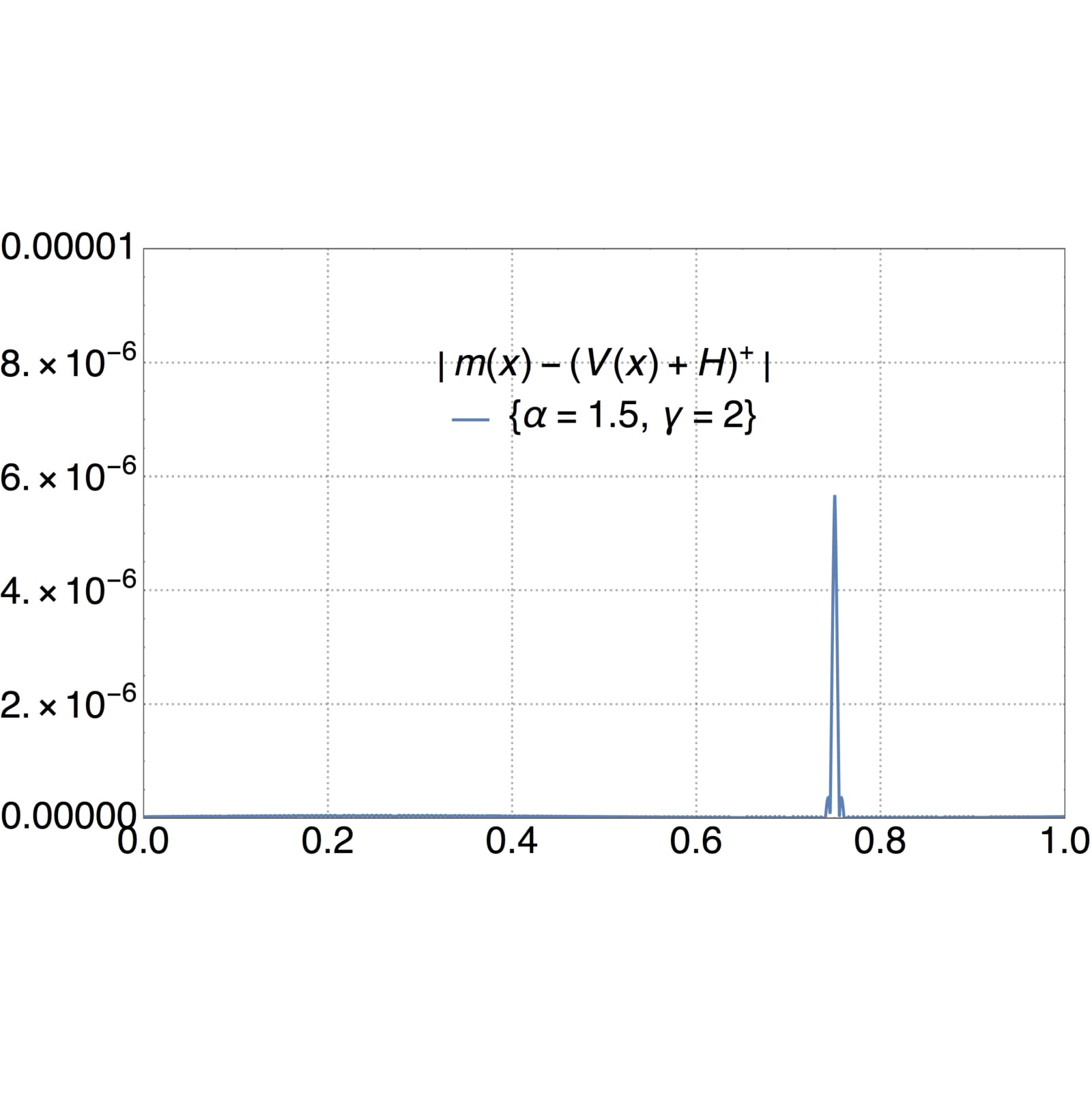}
                \caption{Absolute numerical error}
                \label{fig:plotmvValb}
        \end{subfigure}
        ~ 
        \caption{Numerical solution of the variational problem \eqref{mmz} with \(N=200\) and for \(d=1\), $P=0$, $V(x)= \cos\left(2\pi \left(x-\frac 1 4\right)\right)$, $G(m) = \frac{m^2}{2}$, $\alpha= 1.5$, and $\gamma=2$.}\label{fig:solValidationb}
\end{figure}

\begin{figure}[htb!]
        \centering
        \begin{subfigure}[b]{\sizefigure\textwidth}
                \includegraphics[width=\textwidth]{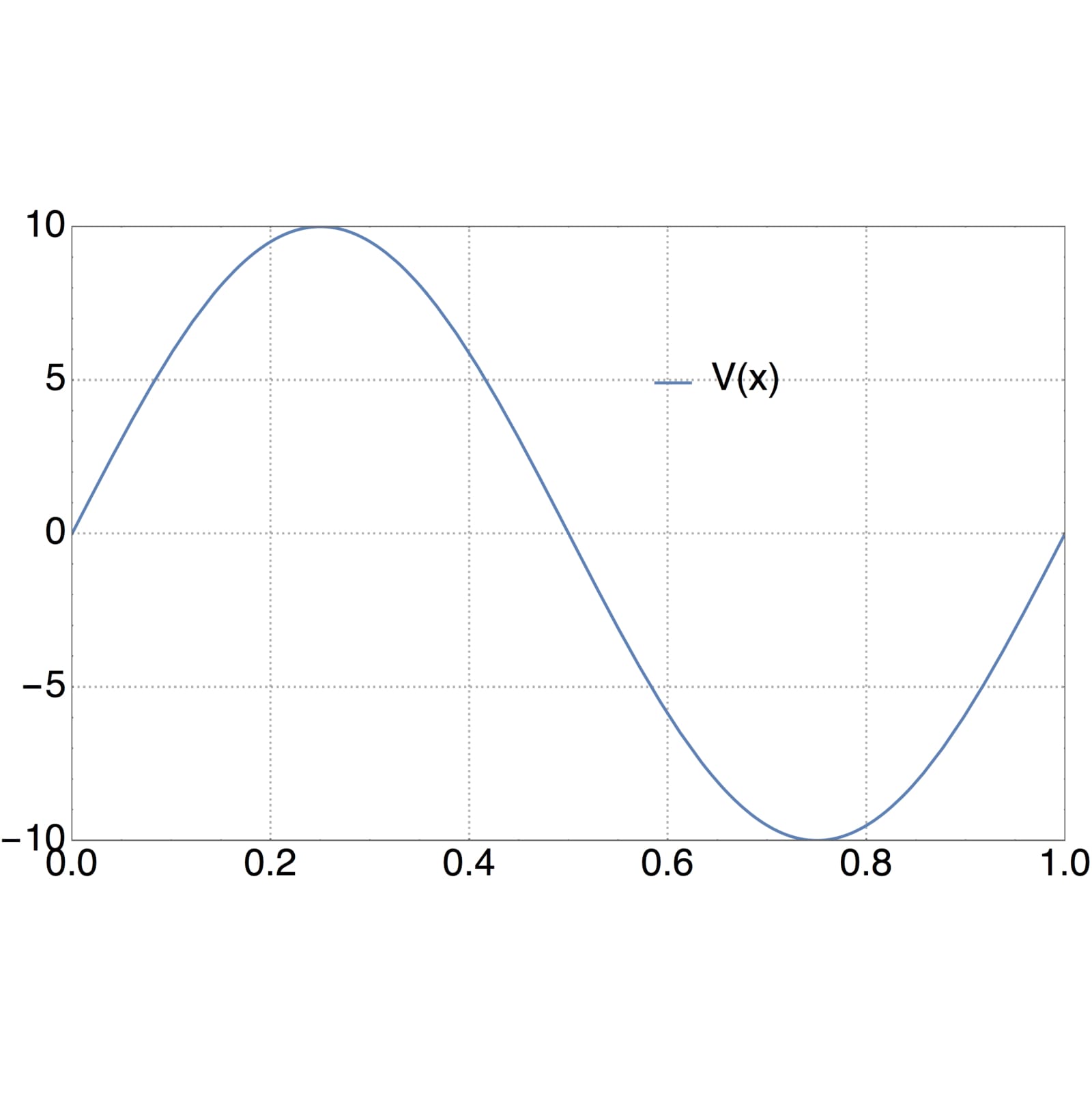}
                \caption{Potential $V$}
                \label{fig:plotVvalc}
        \end{subfigure}
        \begin{subfigure}[b]{\sizefigure\textwidth}
                \includegraphics[width=\textwidth]{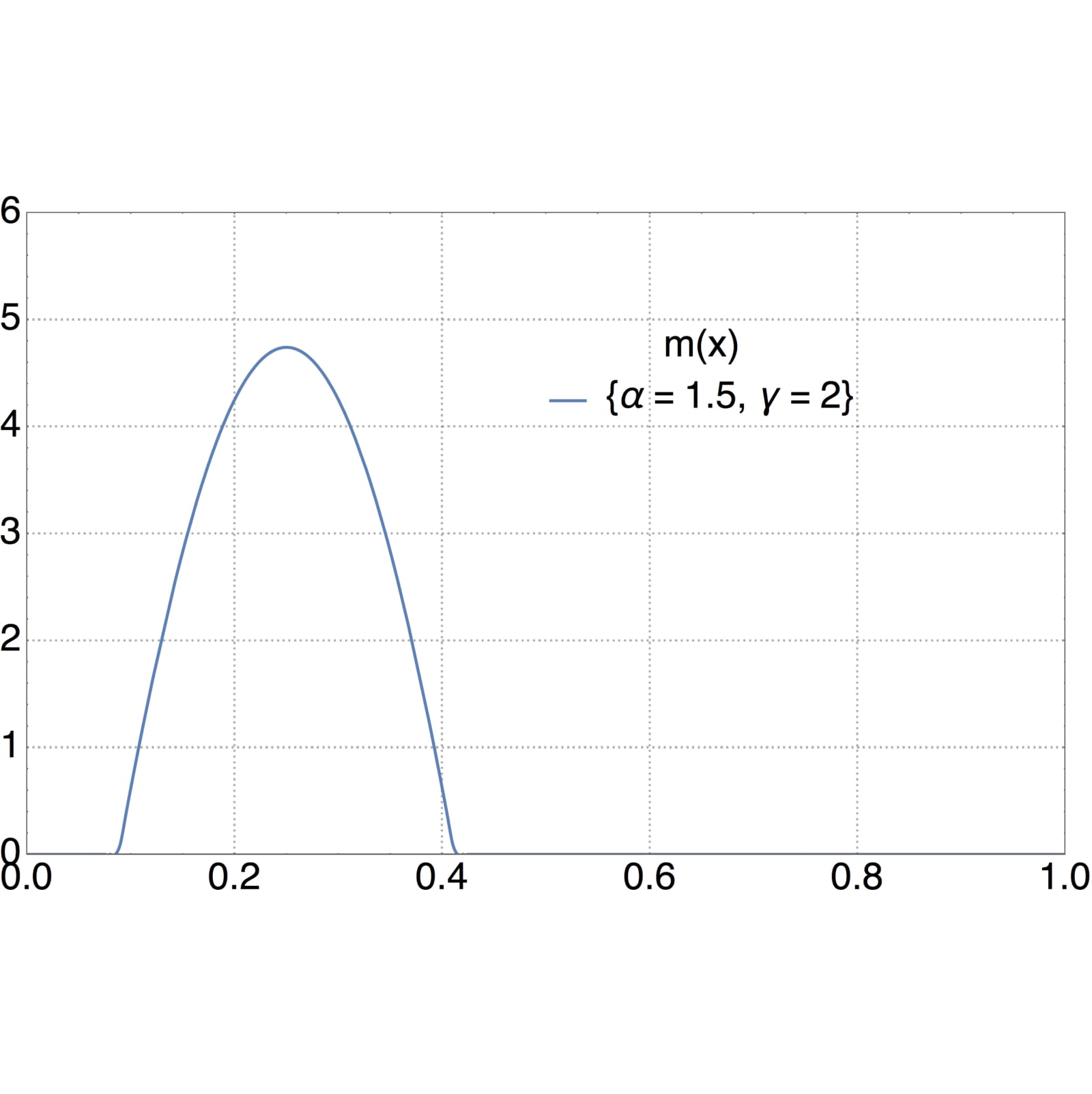}
                \caption{Density $m$}
                \label{fig:plotmvalc}
        \end{subfigure}\\
        \begin{subfigure}[b]{\sizefigure\textwidth}
                \includegraphics[width=\textwidth]{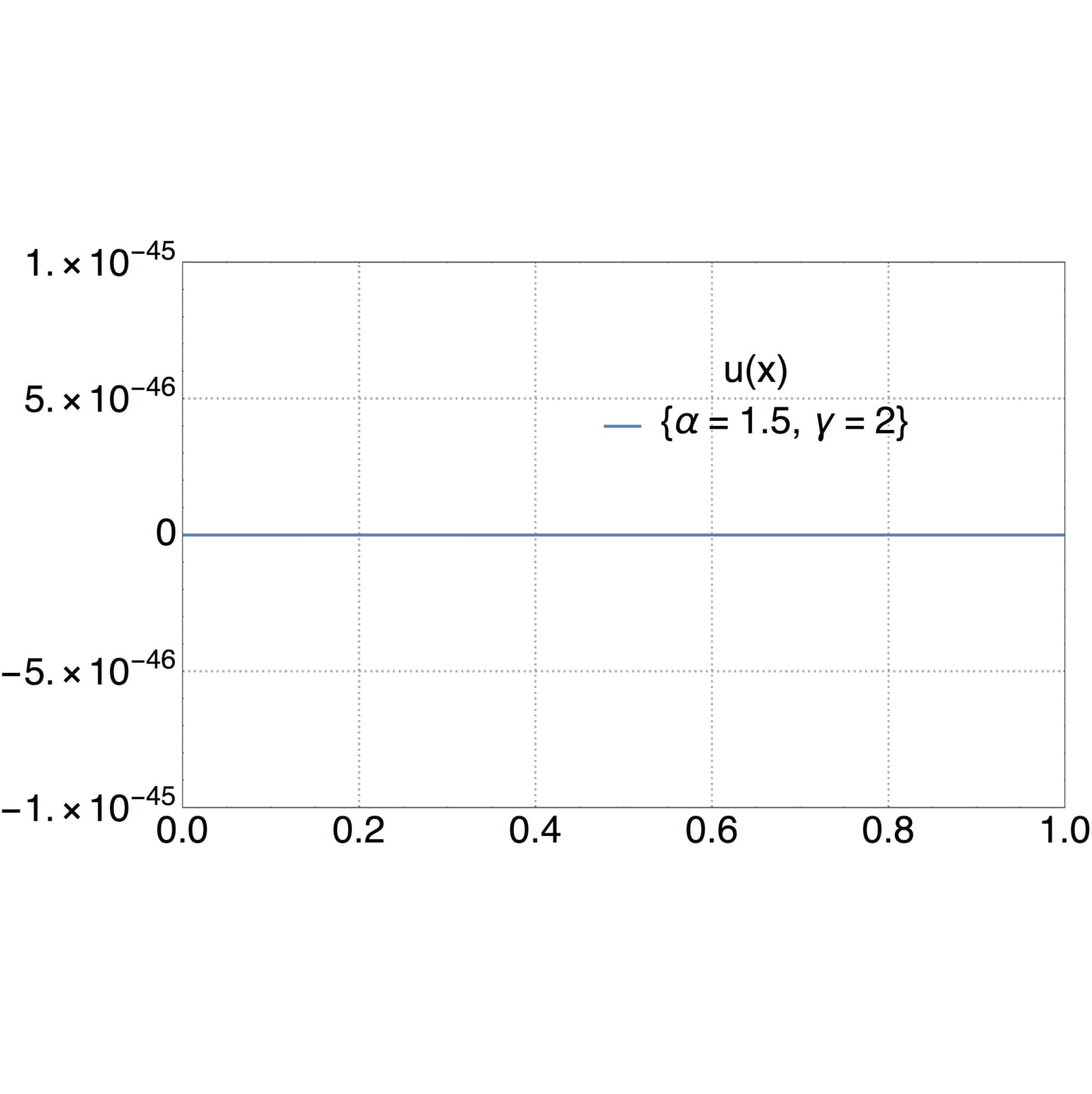}
                \caption{Value function $u$}
                \label{fig:plotuvalc}
        \end{subfigure}
        \begin{subfigure}[b]{\sizefigure\textwidth}
                \includegraphics[width=\textwidth]{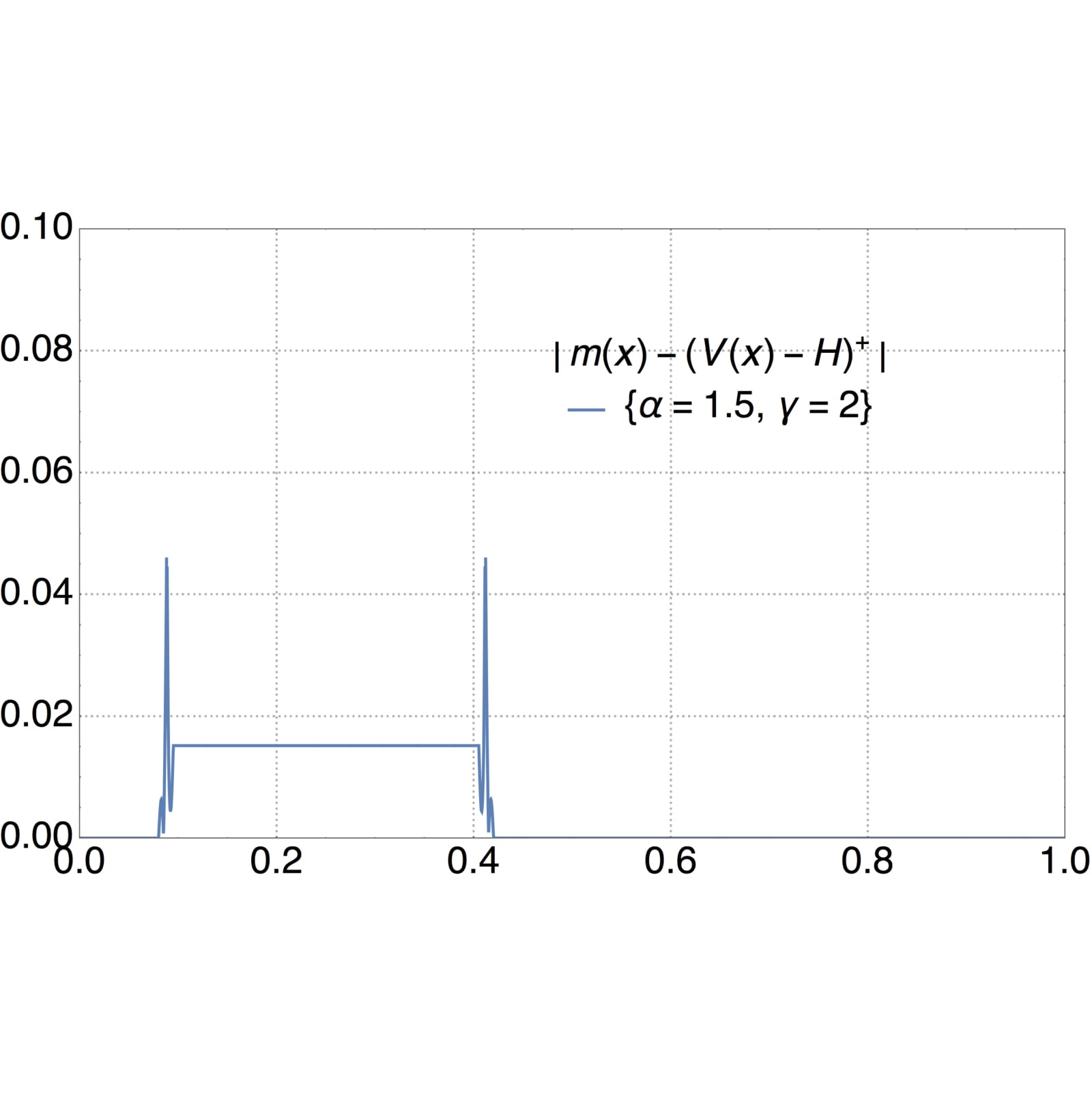}
                \caption{Absolute numerical error}
                \label{fig:plotmvValc}
        \end{subfigure}
        ~ 
        \caption{Numerical solution of the variational problem \eqref{mmz} with \(N=200\) and for \(d=1\), $P=0$, $V(x)=10 \cos\left(2\pi \left(x-\frac 1 4\right)\right)$, $G(m) = \frac{m^2}{2}$, $\alpha= 1.5$, and $\gamma=2$.}\label{fig:solValidationc}
\end{figure}

\begin{table}
\begin{center}
        \begin{tabular}{ | l | | c | c | c| }
                \hline
                \textbf{Grid size} & \textbf{Max abs. error} & \textbf{Mean abs. error} & \textbf{Running time (s)} \\      
                \hline
                \hline
               $N=100$ & 0.03083580&0.010075100 & 4.12853\\ \hline
               $N=200$ & 0.01517990 & 0.004908920 & 6.41835\\ \hline
               $N=400$ &   0.00768403 & 0.002471920 & 24.1912 \\ \hline
               $N=600$ &   0.00522876 & 0.001681400  &72.1434 \\ \hline
               $N=800$ &   0.00385371 & 0.001246080 &105.094 \\ \hline
               $N=1000$ &  0.00308331 & 0.000994913  &182.188 \\ \hline
        \end{tabular}
\end{center}
\caption{Numerical error and running time computation. Comparison between the closed-form solution, $\bar m(x)=(V(x)-\Hh)$, where $\Hh\in\Rr$ is such that $\int_{\Tt}\bar m\, dx=1$, and the numerical solution of the variational problem \eqref{mmz}. Here, $G(m) = \frac{m^2}{2}$, $P=0$, $V(x)=10 \cos\left(2\pi \left(x-\frac 1 4\right)\right)$, $\alpha= 1.5$, and $\gamma=2$.}\label{Table1d}
\end{table}

\begin{figure}[htb!]
        \centering
   \begin{subfigure}[b]{\sizefigure\textwidth}
        \includegraphics[width=\textwidth]{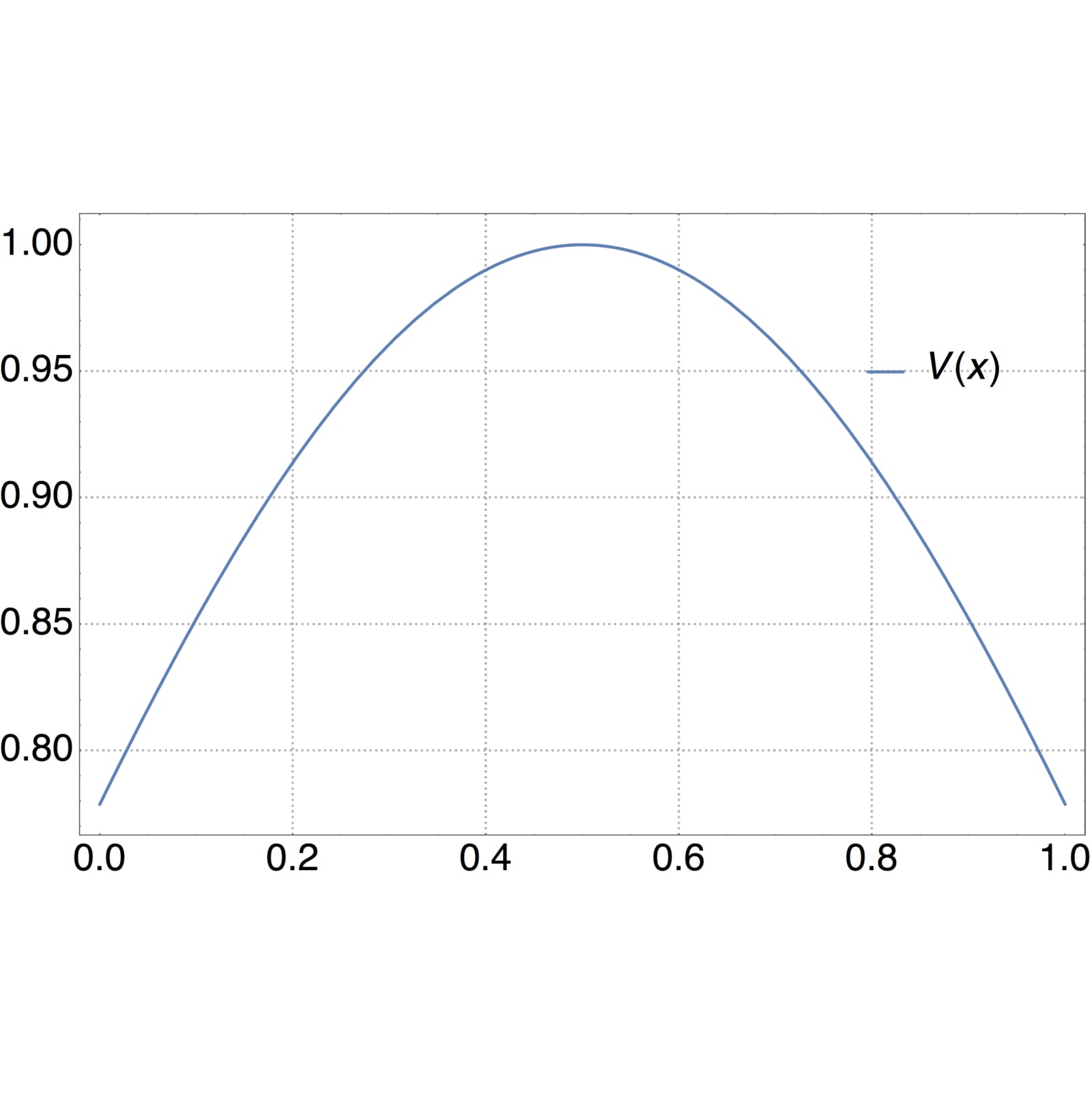}
        \caption{Potential $V$}
        \label{fig:plotVvalExp2}
  \end{subfigure}\\
        \begin{subfigure}[b]{\sizefigure\textwidth}
                \includegraphics[width=\textwidth]{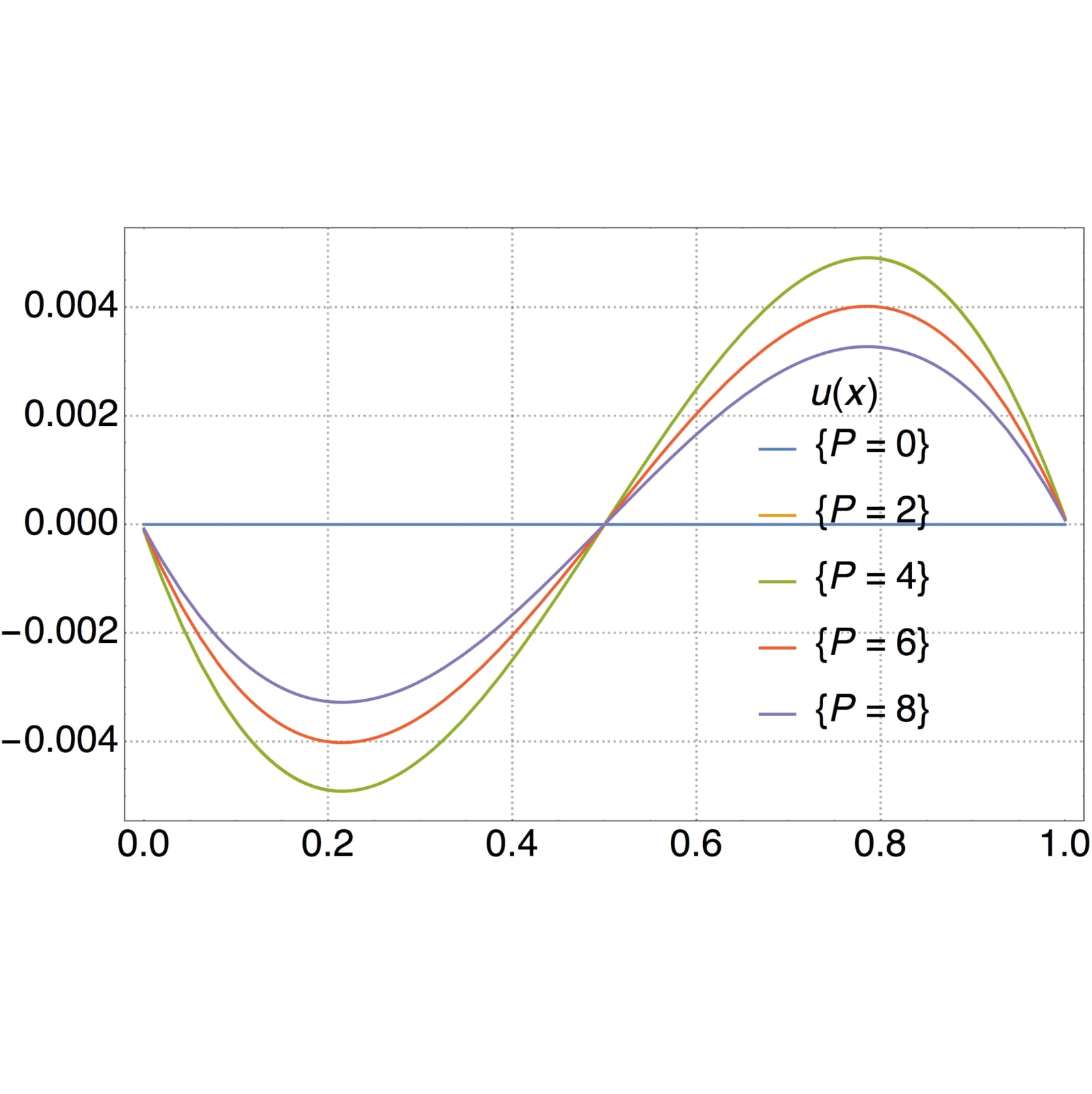}
                \caption{Value function $u$}
                \label{fig:plotuvalExp2}
        \end{subfigure}
        \begin{subfigure}[b]{\sizefigure\textwidth}
                \includegraphics[width=\textwidth]{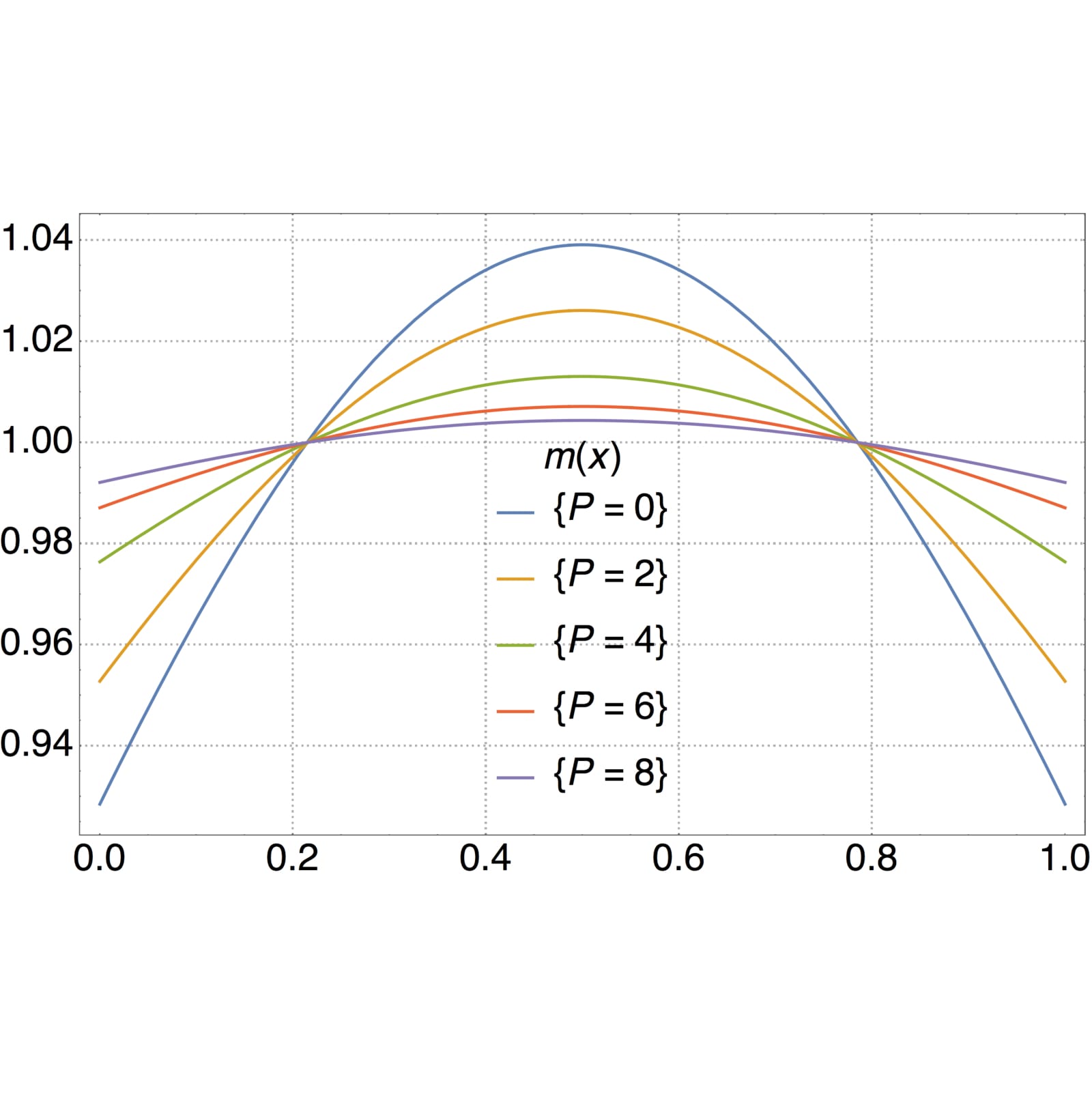}
                \caption{Density $m$}
                \label{fig:plotmvValExp2}
        \end{subfigure}
        
        ~ 
        \caption{Numerical solution of the variational problem \eqref{mmz} with \(N=200\) and  for  \(d=1\), $P\in\{0,2,4,6,8\}$, 
$V(x)= e^{-(x-\frac{1}{2})^2}$, $G(m) = m^2$, $\alpha= 1.5$, and $\gamma=2$.}\label{fig:solValidationExp2}
\end{figure}

%
%
%
%
%
%
%
%
%

\begin{figure}[h!]
        \centering
        \begin{subfigure}[b]{\sizefigure\textwidth}
                \includegraphics[width=\textwidth]{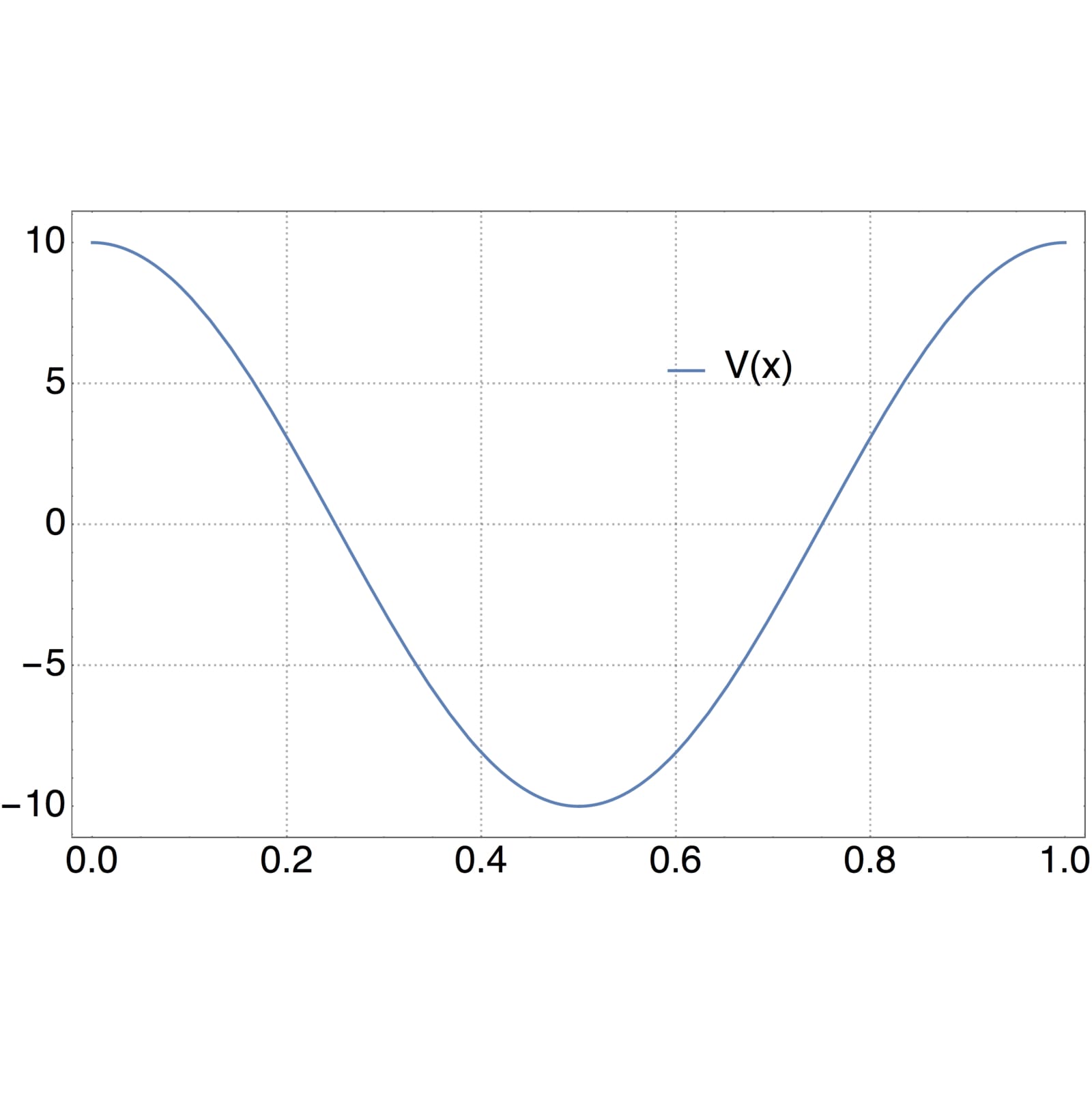}
                \caption{Potential $V$}
                \label{fig:plotV3}
        \end{subfigure}\\
        \begin{subfigure}[b]{\sizefigure\textwidth}
                \includegraphics[width=\textwidth]{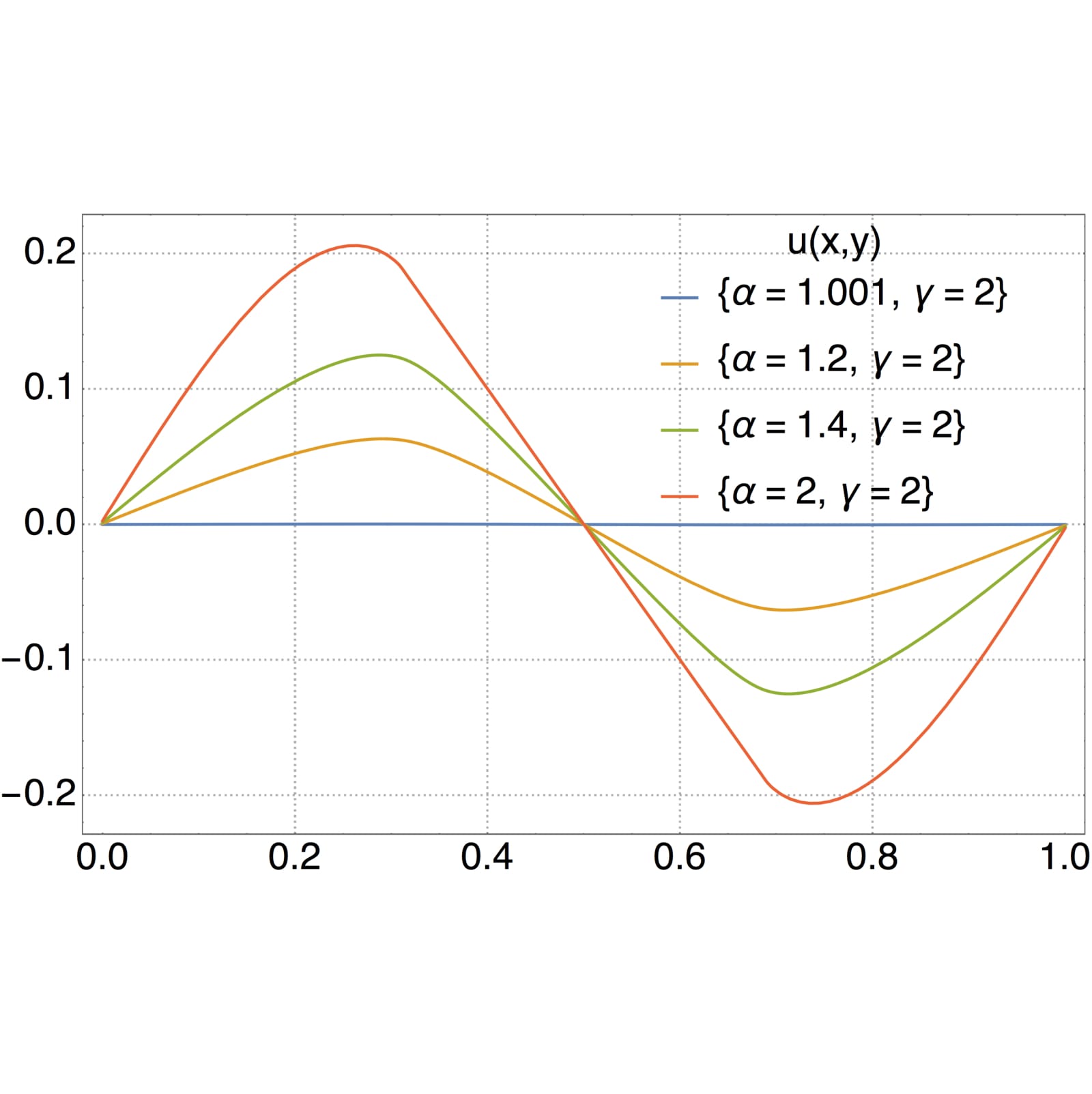}
                \caption{Value function $u$}
                \label{fig:plotuls4}
        \end{subfigure}
                \begin{subfigure}[b]{\sizefigure\textwidth}
                        \includegraphics[width=\textwidth]{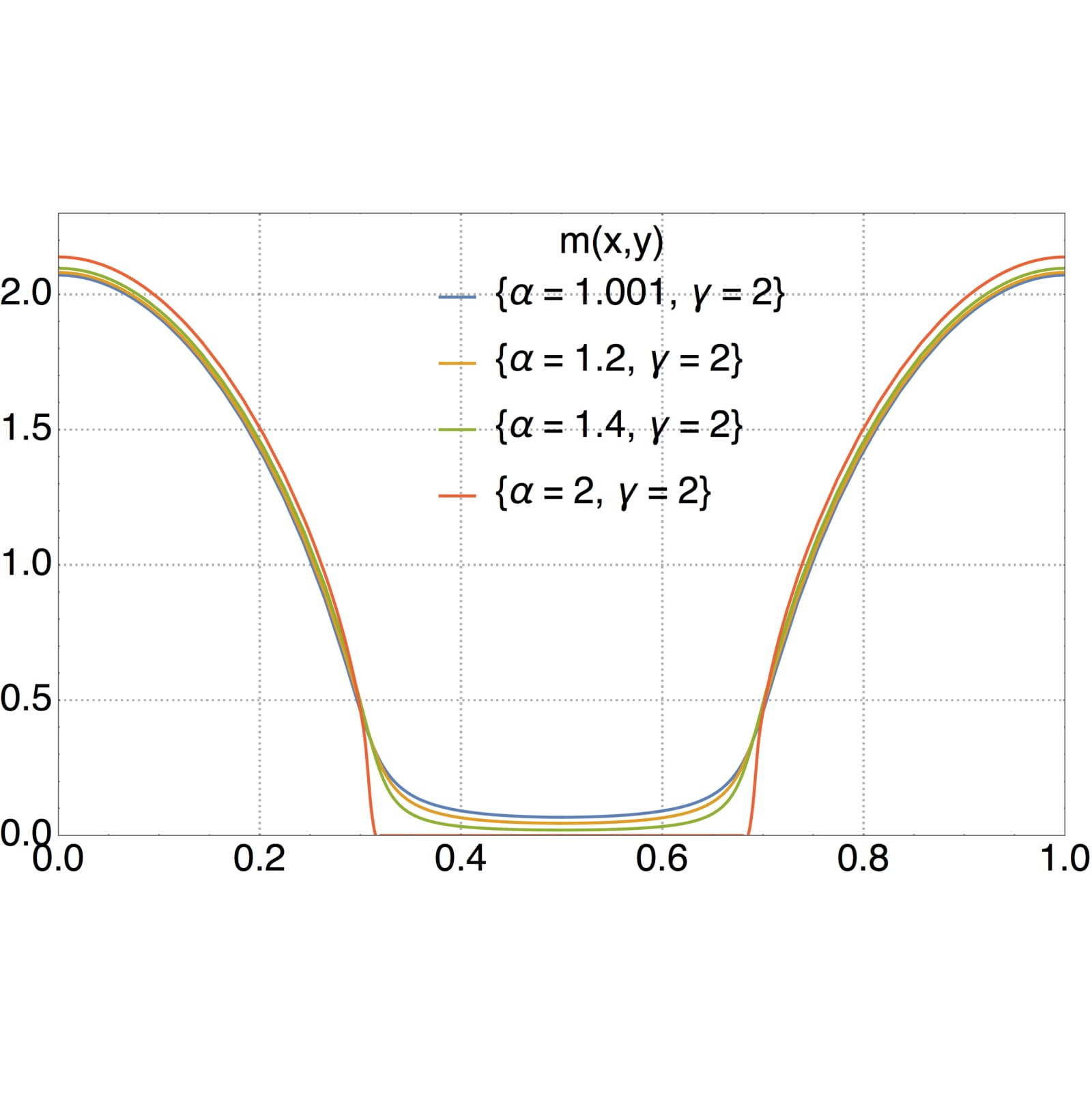}
                        \caption{Density $m$}
                        \label{fig:plotmls4}
                \end{subfigure}
        ~ 
        \caption{Numerical solution of the variational problem \eqref{mmz} with \(N=200\) and for \(d=1\), $P=1$, $V(x)=10 \sin\left(2\pi\left(x+\frac 1 4 \right)\right)$, $G(m) = m^3$, $\alpha\in \{1.001, 1.2, 1.4, 2\}$, and $\gamma=2$.}
        \label{fig:solExpr134}
        
\end{figure}

\newpage
\subsection{Numerical experiments in two dimensions}\label{sub:num2d}

Here, we solve the discretized variational problem~\eqref{mmzdiscrete}. We start by validating our numerical method by considering first the \(P=(0,0)\) case. We recall that in this case, we have determined closed-form solutions of \eqref{mmz} in Section~\ref{varexplicit}. Precisely, we treat first the example in which  we choose the potential
\begin{equation}\label{eq:potential2DValid}
V(x,y)= 10\sin\left(2\pi \left(x+\frac 1 4\right)\right)\cos\left(2\pi \left(y+\frac 1 4\right)\right),
\end{equation}
the coupling \(G(m)=\frac{m^2}{2}\), \(P=(0,0)\), $\alpha=
1.5$, and $\gamma=2$ (see Fig.~\ref{fig:solExpr2c}). The numerical solution (with
$N=50$) for the density, $m$, and the value function, $u$, are shown, respectively, in Figs.~\ref{fig:plotmlsvc} and \ref{fig:plotulsvc}. As expected, $u\equiv 0$ and, in the regions where the potential  is not
\textit{too negative}, the graph of \(m\) resembles that of \(V\).  In 
Fig.~\ref{fig:plotmVvc}, we depict the absolute error between the numerical solution, $m$, and the explicit solution of \eqref{mmz}, $\bar m$, given by \eqref{eq:explweaksolH}. Moreover,  we performed the numerical
simulation for several grid sizes, \(h=\tfrac1N\). In Table~\ref{Table2d},
we present the numerical error between \(m\) and \(\bar m\) and the corresponding running time for \(N\in\{20, 40,  80\}\). We observe that the error decreases in a sublinear manner in $\frac 1 N$ and
the run time is increasing somewhat faster than quadratic in $N^2$, the total number of nodes.  

Next, in Fig.~\ref{fig:solExpr1}, we choose  $P=(p_1,p_2)=(1,3)$,  $V(x,y)= \sin\left(2\pi
\left(x+\frac 1 4\right)\right)\cos\left(2\pi \left(y+\frac 1 4\right)\right)$, $G(m)=m^3$,
$\alpha=1.5$, and $\gamma=2$. The density, $m$, and the value function, $u$, for this example are displayed in Figs.~\ref{fig:plotmls} and \ref{fig:plotuls}, respectively. In this figure, we note the asymmetry caused by the vector $P$, which introduces a preferred direction of motion.

Finally, in Fig.~\ref{fig:solExpr12}, we illustrate the solution  of \eqref{mmzdiscrete} for  $P=(-1,3)$, $G(m) = \frac{m^2}{2}+ m^3$, $V(x,y)= e^{-\sin\left(2\pi \left(x+\frac 1 4\right)\right)^2}\cos\left(2\pi \left(y-\frac 1 4\right)\right)$, $\alpha= 2$, and $\gamma=2.5.$ 
Note that the numerical solution is robust with respect to the choice of the parameters, and the functions $V$ and $G$. Moreover, the density resembles the potential in all the considered experiments.

\begin{figure}[htb!]
        \centering
        \begin{subfigure}[b]{\sizefigure\textwidth}
                \includegraphics[width=\textwidth]{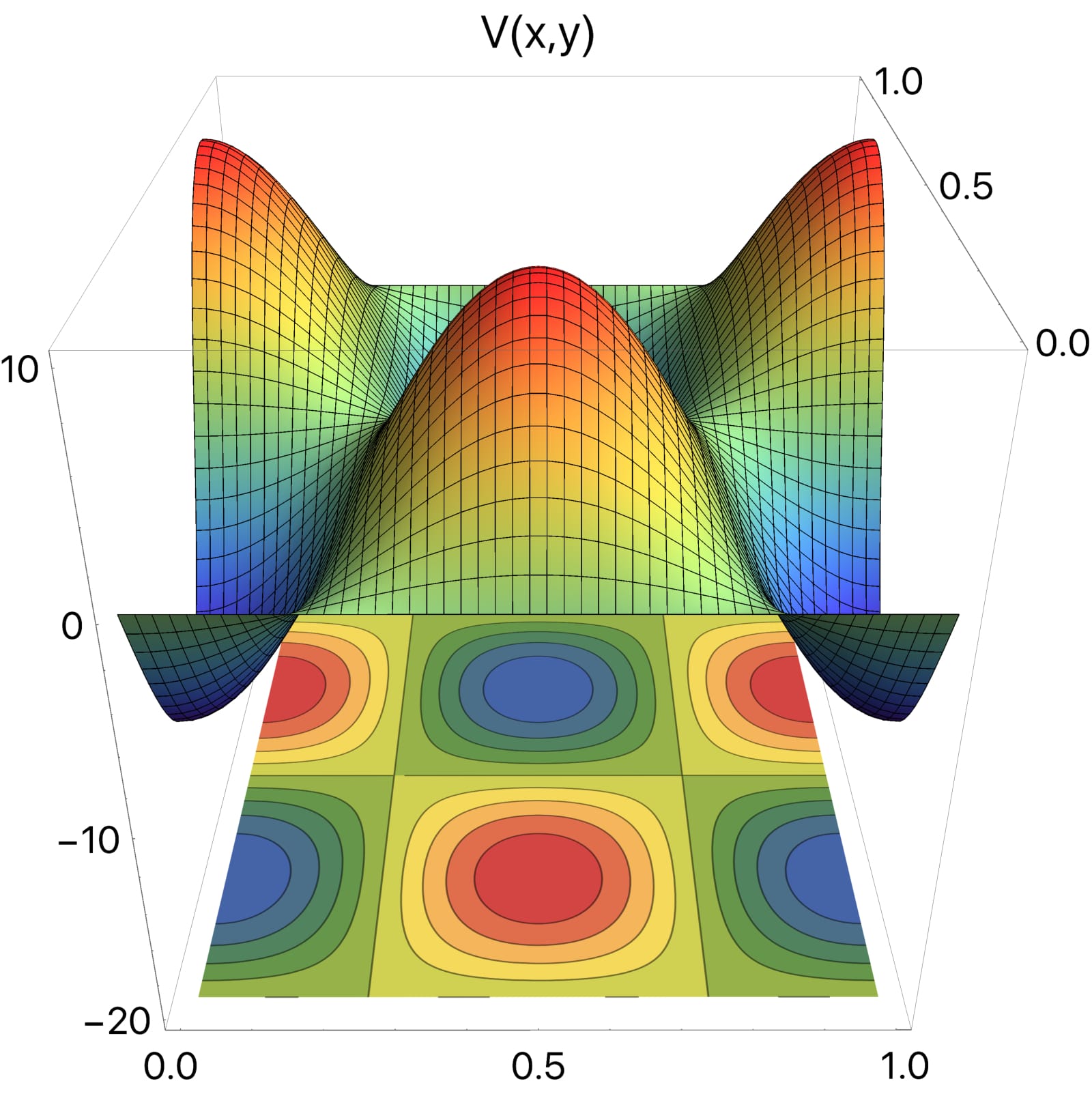}
                \caption{Potential $V$}
                \label{fig:plotVvc}
        \end{subfigure}
        ~\bigskip\bigskip
        \begin{subfigure}[b]{\sizefigure\textwidth}
                \includegraphics[width=\textwidth]{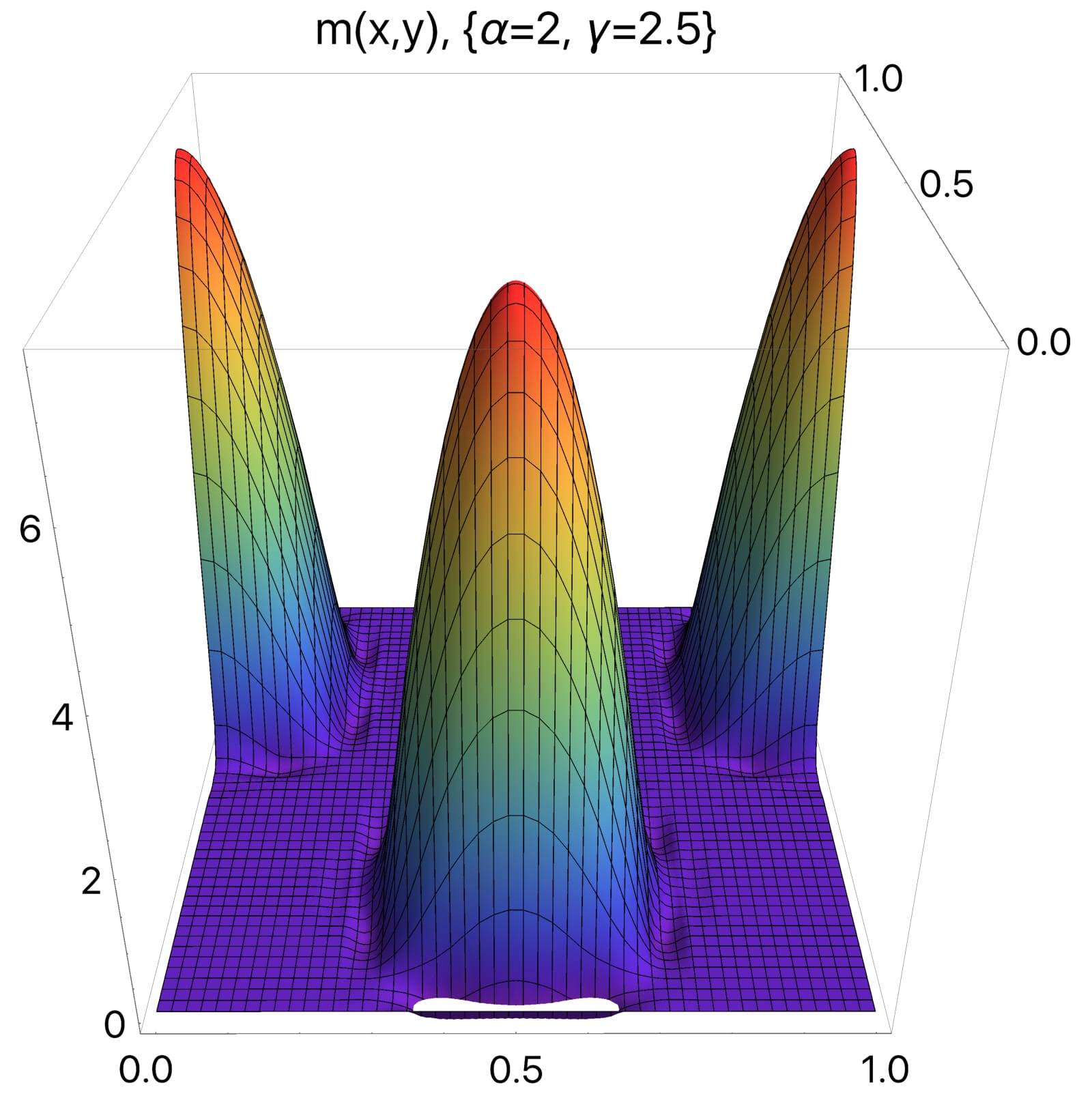}
                \caption{Density $m$}
                \label{fig:plotmlsvc}
        \end{subfigure}
        \begin{subfigure}[b]{\sizefigure\textwidth}
                \includegraphics[width=\textwidth]{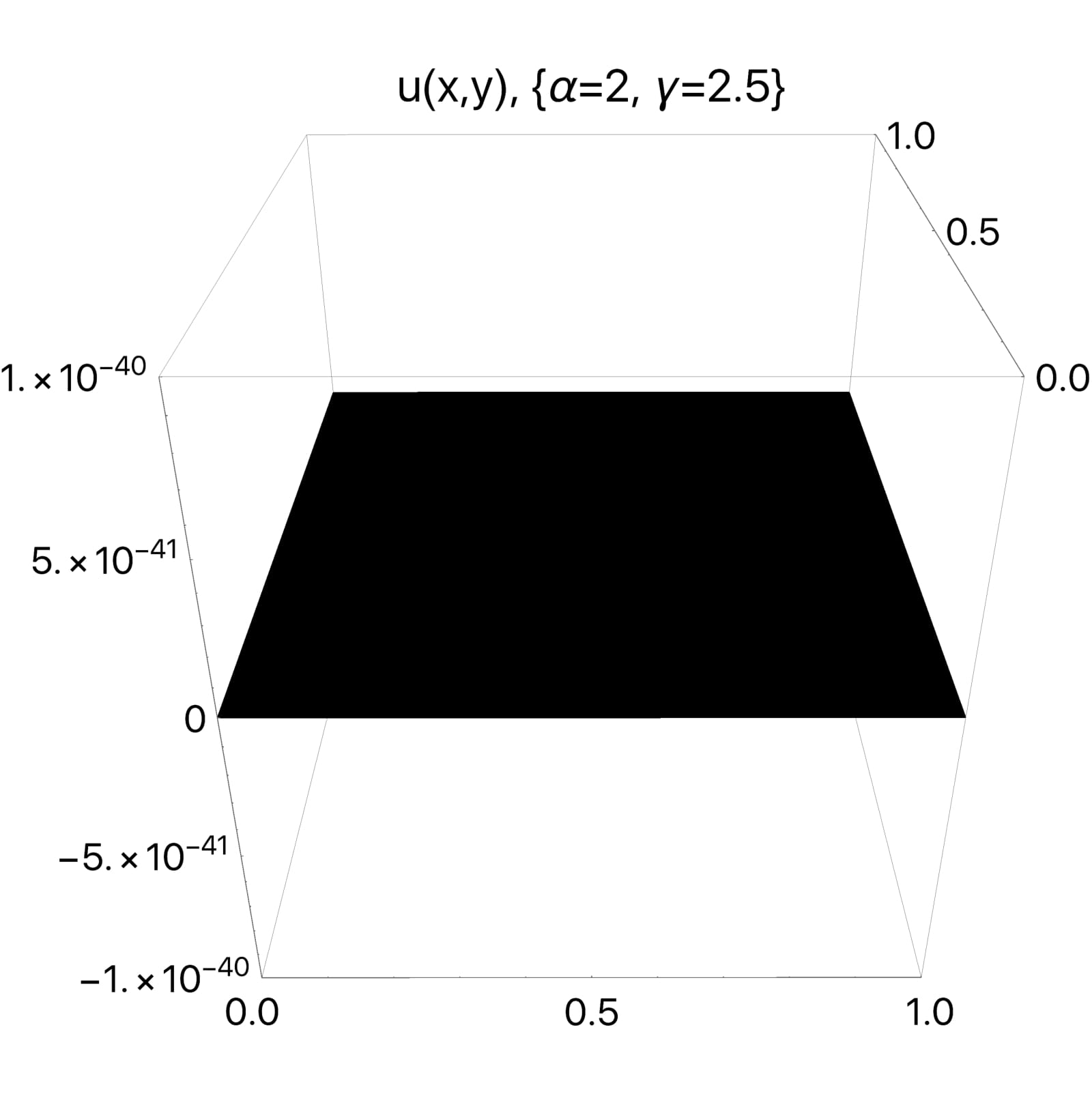}
                \caption{Value function $u$}
                \label{fig:plotulsvc}
        \end{subfigure}
        \begin{subfigure}[b]{\sizefigure\textwidth}
                \includegraphics[width=\textwidth]{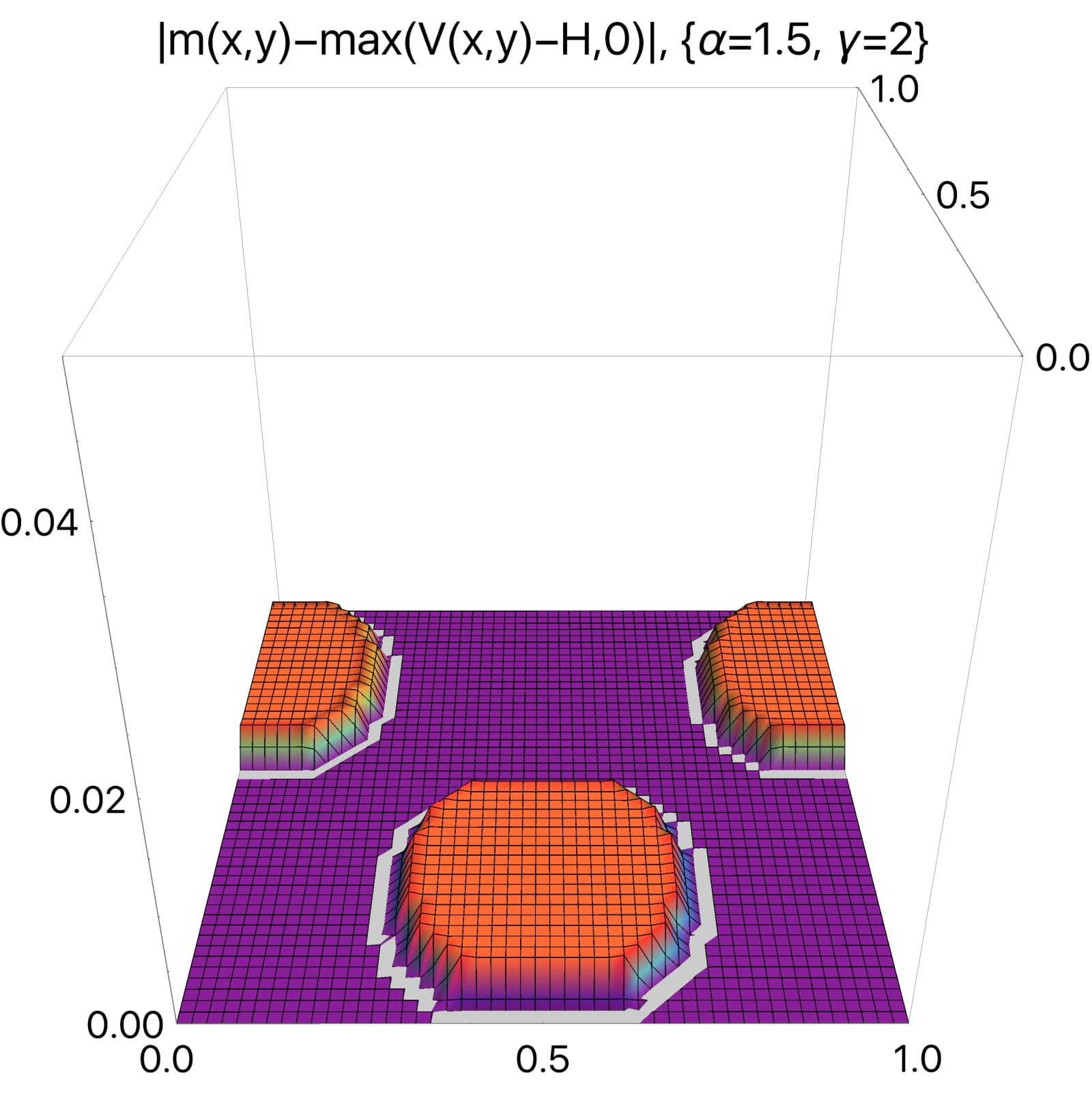}
                \caption{Absolute numerical error}
                \label{fig:plotmVvc}
        \end{subfigure}
        ~ 
        \caption{ Numerical solution of the variational problem \eqref{mmz} with \(N=50\) and for \(d=2\), $P=(0,0)$, $V(x,y)= 10\sin\left(2\pi \left(x+\frac 1 4\right)\right)\cos\left(2\pi \left(y+\frac 1 5\right)\right)$, $G(m)=\frac{m^2}{2}$, $\alpha= 1.5$, and $\gamma=2$.}\label{fig:solExpr2c}
\end{figure}
%
%
%
%
%

\begin{figure}[htb!]
        \centering
        \begin{subfigure}[a]{\sizefigure\textwidth}
                \includegraphics[width=\textwidth]{plotVlsExprValidc.jpg}
                \caption{Potential $V$}
                \label{fig:plotVvc2}
        \end{subfigure}\\
        \begin{subfigure}[b]{\sizefigure\textwidth}
                \includegraphics[width=\textwidth]{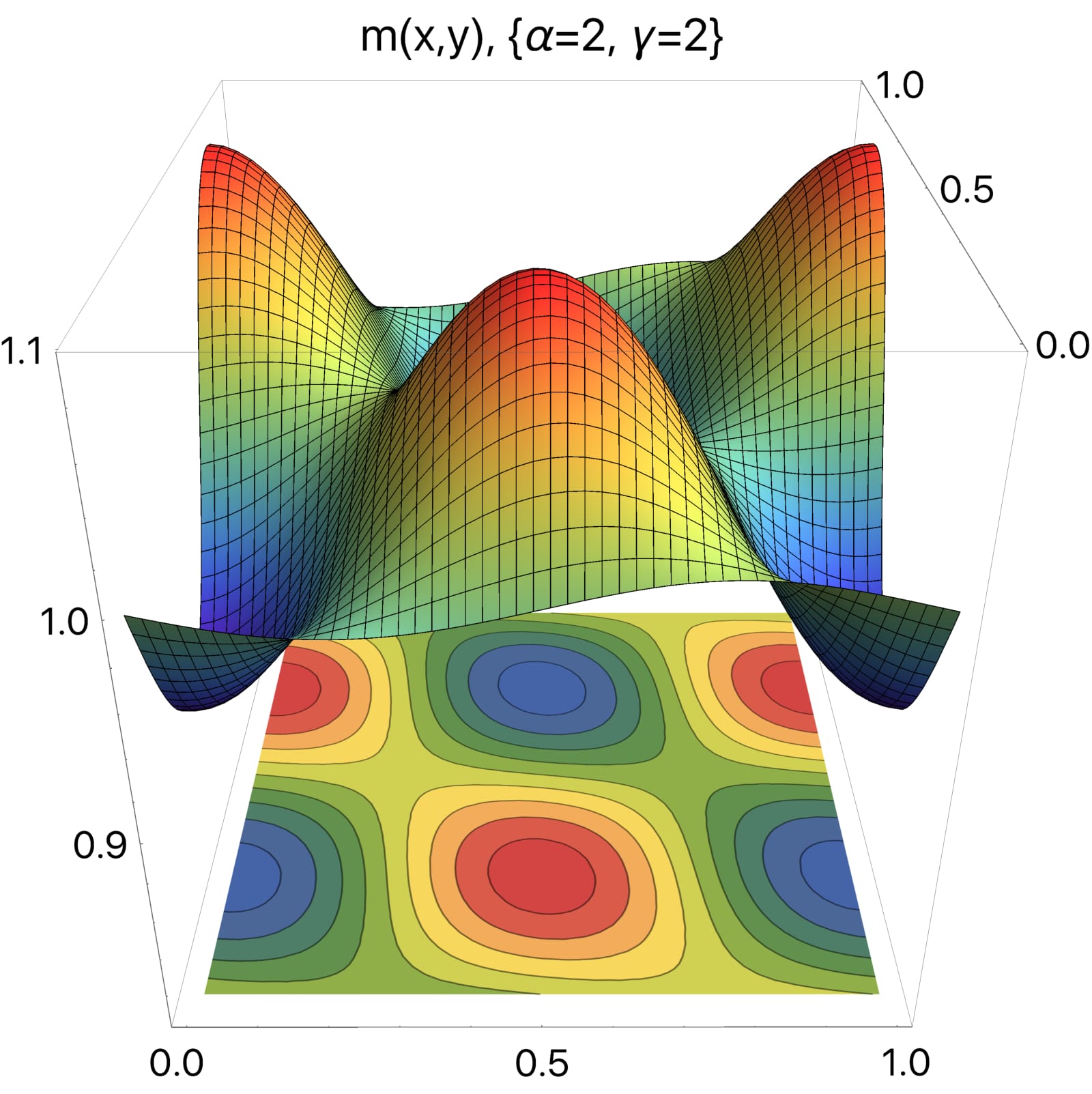}
                \caption{The density $m$}
                \label{fig:plotmls}
        \end{subfigure}
        \begin{subfigure}[b]{\sizefigure\textwidth}
                \includegraphics[width=\textwidth]{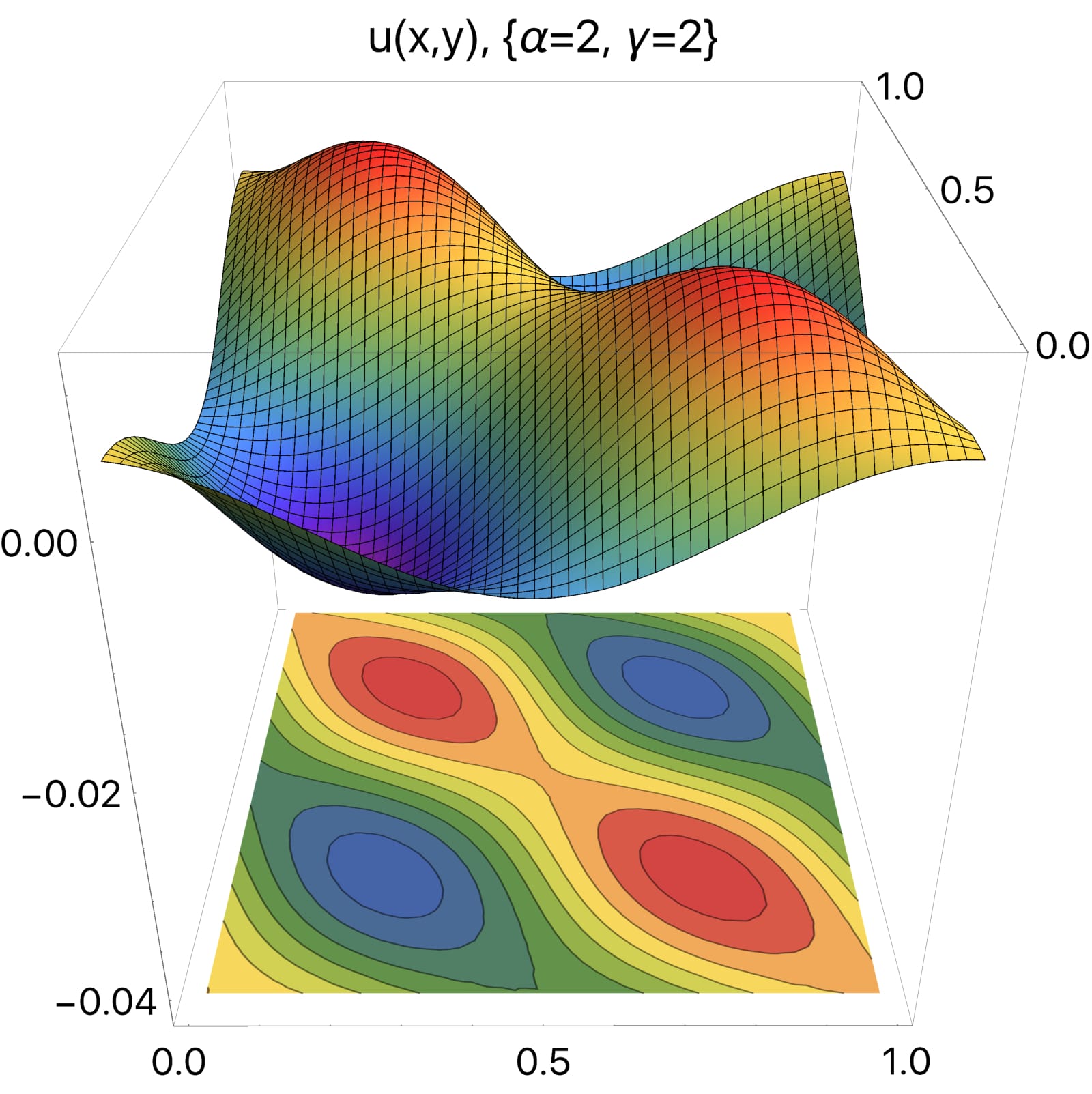}
                \caption{The value function $u$}
                \label{fig:plotuls}
        \end{subfigure}
        ~ 
        \caption{Numerical solution of the variational problem \eqref{mmz} with \(N=50\) and for \(d=2\), $P=(1,3)$, $V(x,y)= \sin\left(2\pi \left(x+\frac 1 4\right)\right)\cos\left(2\pi \left(y+\frac 1 4\right)\right)$, $G(m) = m^3$, $\alpha= 1.5$, and $\gamma=2$.}\label{fig:solExpr1}
\end{figure}

\begin{table}
        \begin{center}
                \begin{tabular}{ | l | | c | c | c| }
                        \hline
                        \textbf{Grid size} & \textbf{Max abs. error} & \textbf{Mean abs. error} & \textbf{Running time (s)} \\      
                        \hline
                        \hline
                        $N=20$ & 0.03982920& 0.011831300 & 3.39795 \\ \hline
                        $N=40$ & 0.00691211 & 0.002116950  &222.576 \\ \hline
                        $N=80$ & 0.00108425& 0.000318745 & 6032.06  \\ \hline
                \end{tabular}
        \end{center}
        \caption{Numerical error and running time computation. Comparison between the closed-form solution, $\bar m(x)=(V(x)-\Hh)^+$, where $\Hh\in\Rr$ is such that $\int_{\Tt^2}\bar m\, dx=1$, and the numerical solution of the variational problem \eqref{mmz}. Here, $P=(0,0)$, $V(x)=10 \sin\left(2\pi \left(x+\frac 1 4\right)\right)\cos\left(2\pi \left(y+\frac 1 4\right)\right)$, $G(m) = \frac{m^2}{2}$,  $\alpha= 1.5$, and $\gamma=2$.}\label{Table2d}
\end{table}

%
%
%
%
%
\begin{figure}[htb!]
        \centering
        \begin{subfigure}[b]{\sizefigure\textwidth}
                \includegraphics[width=\textwidth]{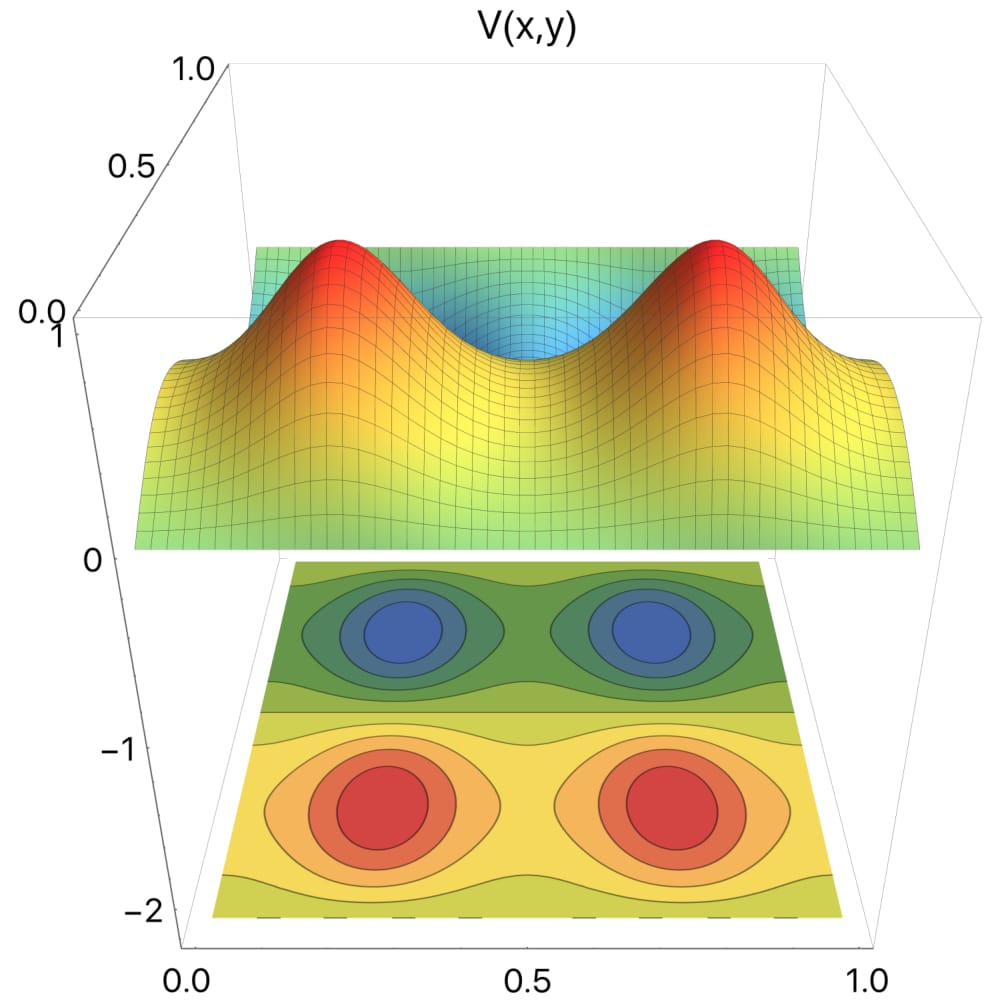}
                \caption{Potential $V$}
                \label{fig:plotV2}
        \end{subfigure}\\
        \begin{subfigure}[b]{\sizefigure\textwidth}
                \includegraphics[width=\textwidth]{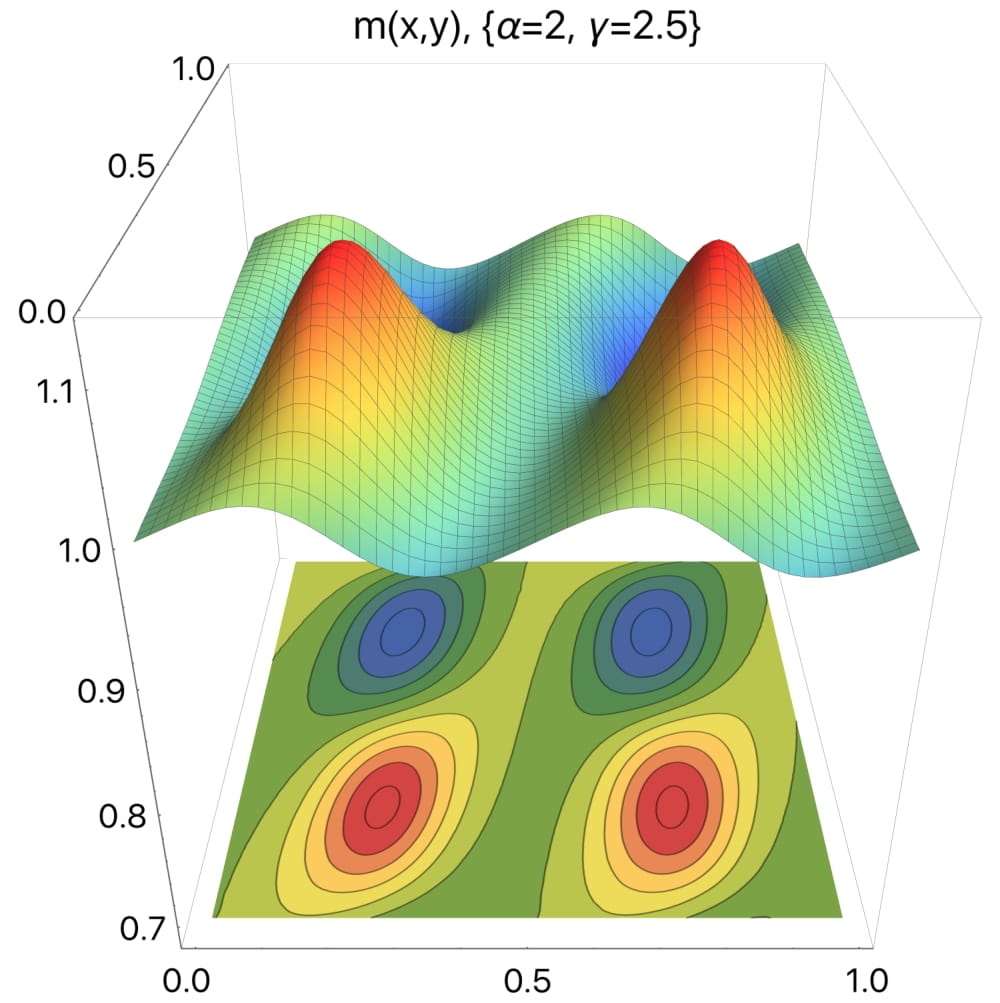}
                \caption{Density $m$}
                \label{fig:plotmls2}
        \end{subfigure}
        \begin{subfigure}[b]{\sizefigure\textwidth}
                \includegraphics[width=\textwidth]{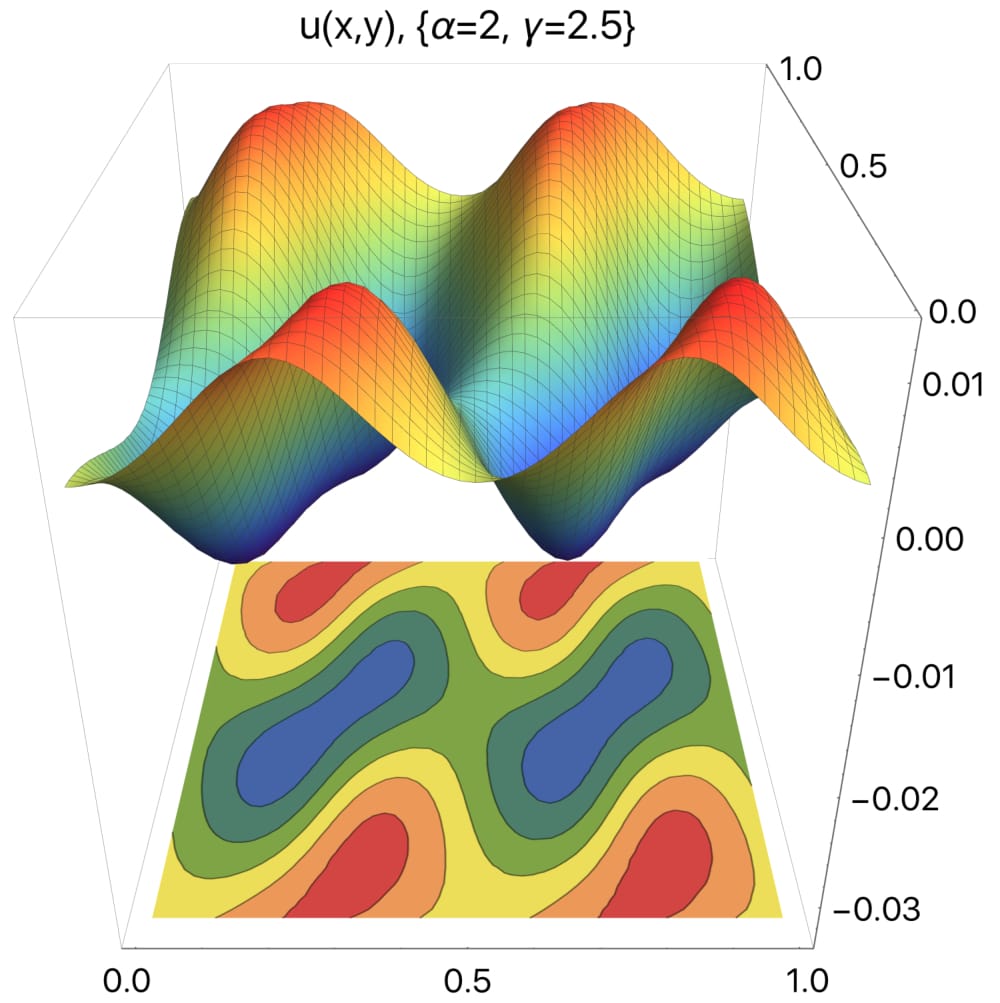}
                \caption{Value function $u$}
                \label{fig:plotuls2}
        \end{subfigure}
        ~ 
        \caption{Numerical solution of the variational problem \eqref{mmz} with \(N=50\) and for \(d=2\), $P=(-1,3)$, $V(x,y)= e^{-\sin\left(2\pi \left(x+\frac 1 4\right)\right)^2}\cos\left(2\pi \left(y-\frac 1 4\right)\right)$, $G(m) = \frac{m^2}{2}+ m^3$, $\alpha= 2$, and $\gamma=2.5$.}\label{fig:solExpr12}
\end{figure}
\section{Numerical solution for the two-dimensional first-order MFGs with congestion, with $0<\alpha<1$ and $\gamma>1$}\label{num8}
In this section, we numerically solve the following MFG system, introduced in Section \ref{2dcase}:
\begin{equation}\label{eq:hjb2d}
\begin{cases}
\frac{|P+Du|^{\gamma}}{\gamma m^{\alpha}}+V(x)-g(m)=\overline{H}&\quad \mbox{in }\; \Tt^2\\
\div(m^{1-\alpha} |P+Du|^{\gamma-2} (P+Du) )=0 &\quad \mbox{in }\; \Tt^2\\
m>0, \int_{\Tt^2} m(x)\,dx=1
\end{cases}
\end{equation}
for  $0<\alpha<1$ and $\gamma>1$. In the sequel, we describe a procedure to obtain a numerical solution of \eqref{eq:hjb2d}. 
\bigskip
\paragraph{\textbf{Step 1.}} First, we define
\[
\gamma'=\frac{\gamma}{\gamma-1} \quad\mbox{and}\quad \tilde{\alpha}=\alpha-(\alpha-1)\gamma'.
\]
Note that $\gamma'>1$ and $1<\talpha<\gamma'$.
\bigskip
\paragraph{\textbf{Step 2.}} Next, using the results from Section \ref{2dcase}, we transform \eqref{eq:hjb2d} into

\begin{equation}\label{eq:Transformed}
\begin{cases}
\frac{|Q+D\psi|^{\gamma'}}{\gamma' m^{\talpha}}+\frac{\gamma}{\gamma'}V(x)-\frac{\gamma}{\gamma'}g(m)=
{\frac{\gamma}{\gamma'}} \overline{H}&\quad \mbox{in }\; \Tt^2\\
\div(m^{1-\talpha} |Q+D\psi|^{\gamma'-2} (Q+D\psi) )=0&\quad \mbox{in }\; \Tt^2\\
m>0,\int_{\Tt^2} m(x)\,dx=1.
\end{cases}
\end{equation}
Then, for each $Q$, we solve the above system. For this, we formulate a variational principle analogous to \eqref{mmz}. Namely, for fixed $Q$, we minimize the functional
\begin{equation}\label{eq:Jpsiandm}
J[\psi,m] = \int_{\Tt^2}\left( \frac{|Q+D\psi|^{\gamma'}}{\gamma' (\talpha - 1) m^{\talpha-1}} -\frac{\gamma}{\gamma'}V m +\frac{\gamma}{\gamma'}G(m)\right) dx
\end{equation}
under the constraints $\int_{\Tt^2}\psi \,dx=0$, $\int_{\Tt^d}m\, dx= 1$, and $m>0$,
using the corresponding discretization for $\psi$ and $m$. That is, considering the grid functions, $\psi,m\in \Rr^{N^2}$, we numerically solve
\begin{equation}\label{eq:minJpsiandm}
\min_{(\psi,m)\in \mathcal A_h} J_h[\psi,m],
\end{equation}
where
\begin{equation}\label{eq:discreteJpsiandm}
 J_h[\psi,m]=h^2\sum_{i,j=0}^{N-1}\left( ( f_h([D_h \psi],m))_{i,j} - \frac{\gamma}{\gamma'}V_{i,j}m_{i,j}+\frac{\gamma}{\gamma'}G(m_{i,j})\right)
\end{equation}
with
\[
 ( f_h([D_h \psi],m))_{i,j}=\frac{1}{m_{i,j}^{\talpha-1 }\gamma'(\talpha -1)  }
\left(\left(q_1+(D^h_1 \psi)_{i,j}\right)^2
+\left(q_2+(D^h_2 \psi)_{i,j}\right)^2\right)
^{\gamma' /2}
\]
for $m_{i,j}\not=0$ and $i,j\in\{0,\ldots,N-1\}$. The set $\mathcal{A}_{h}$ is the same as in \eqref{eq:setAh}, replacing $u$ by $\psi$. Moreover, the discretization scheme is the one given in \eqref{diffschemes}--\eqref{gradCentralDiff}.

\bigskip
\paragraph{\textbf{Step 3.}} So far, for each \(Q\), we  have $\psi$ and $m$ satisfying \eqref{eq:minJpsiandm}. Next, we determine the corresponding vector \(P\). By \eqref{eq:PperpDuPerp}, we have
\[
P^\perp+(Du)^\perp=
m^{1-\talpha} |Q+D\psi|^{\gamma'-2} (Q+D\psi).
\]
Because $u$ is periodic, by integrating the previous expression in $\Tt^2$, we obtain
\begin{equation}\label{eq:Pperp}
P^\perp = \int_{\Tt^2}m^{1-\talpha} |Q+D\psi|^{\gamma'-2} (Q+D\psi)\,dx.
\end{equation}


We observe that  $P^\perp$ can be obtained using the functional $J$.
In fact, let  
\[
\vartheta[\psi,m]=\frac{|Q+D\psi|^{\gamma'}}{\gamma' (\talpha - 1) m^{\talpha-1}} -\frac{\gamma}{\gamma'}V m +\frac{\gamma}{\gamma'}G(m)
\]
denote the integrand of $J$. Taking the variational derivative of $\vartheta$  with respect to $m$ in the direction of $1$, we get
\[
\frac{\delta}{\delta m} \vartheta [\psi,m]=\lim_{\epsilon\to 0}\frac{\vartheta [\psi,m+\epsilon]-\vartheta[\psi,m]}{\epsilon}=   -\frac{|Q+D\psi|^{\gamma'}}{\gamma'  m^{\talpha}} -\frac{\gamma}{\gamma'}V +\frac{\gamma}{\gamma'}g(m).
\]
Next,  differentiating the  expression above with respect to $Q$, we obtain
\[
 -\frac{|Q+D\psi|^{\gamma'-2}(Q+D\psi)}{ m^{\talpha}}.
\]
Finally, multiplying this last expression by $m$ and integrating over $\Tt^2$, we obtain
the right-hand side of \eqref{eq:Pperp}. 
Therefore,  $P^\perp$ is alternatively given by
\begin{equation}
\label{eq:formPper}
\begin{aligned}
P^\perp = -\int_{\Tt^2}\frac{\partial}{\partial Q}\left(\frac{\delta}{\delta m} \vartheta[\psi,m](Q)\right)m\,dx.
\end{aligned}
\end{equation}

Recalling that  $P^\perp = (-p_2,p_1)$,
 we observe that the discrete version of \eqref{eq:formPper} is
\begin{align}\label{eq:p1p2}
\begin{split}
p_1 &= h^2\sum_{i,j=0}^{N-1}(\mathbf{D}^h_{q_2}(\mathbf{D}^h_m (\vartheta[\psi,m])(q_2))_{i,j}m_{i,j})_{i,j}\,dx\\
p_2 &= -h^2\sum_{i,j=0}^{N-1}\mathbf{D}^h_{q_1}(\mathbf{D}^h_m (\vartheta[\psi,m])_{i,j}(q_1))_{i,j}m_{i,j}\,dx.
\end{split}
\end{align}
\begin{remark}
The notation $(\mathbf{D}^h_{x}(f(x)))_{i,j}$ must be understood as the $i,j$-node of the symbolic derivative of a grid function, $f$, with respect to $x$,   whereas $(\mathbf{D}_f^h g[f])_{i,j}$
is the symbolic variational derivative of $g$ with respect to $f$ in the direction of $1$.
In our implementation,
we use symbolic calculus to compute those derivatives
 in a straightforward and automated fashion.
Here, we obtain \eqref{eq:p1p2} using standard symbolic manipulations applied to $J$. 
\end{remark}
%


\paragraph{\textbf{Step 4.}} Finally, we use $m$ and \(P\) from the previous steps and solve, for $u$, the Hamilton--Jacobi equation, 
\begin{equation}\label{eq:hjbEffHamiltgamma}
\frac{|P+Du|^{\gamma}}{\gamma m^{\alpha}}+V(x)-g(m)=\overline{H} \quad \mbox{in }\;\Tt^2.
\end{equation}
Under the notation of Section \ref{sub:discretization}, we use the following monotone scheme for $|P+D^h u |_{i,j}^{\gamma}$, $i,j\in\{0,\ldots, N-1\}$:
\begin{align*}
|P+D^h u|_{i,j}^{\gamma}=&\max(-p_1-(D^h_1 u)^{+}_{i,j},0)^{\gamma}+\max(p_1+(D^h_1 u)^{+}_{i-1,j},0)^{\gamma}\\&+\max(-p_2-(D^h_2 u)^{+}_{i,j},0)^{\gamma}+\max(p_2+(D^h_2 u)^{+}_{i,j-1},0)^{\gamma}.
\end{align*}
From \cite{LPV}, we have that
\[
\beta u^{(\beta)} + \frac{|P+Du^{(\beta)}|^{\gamma}}{\gamma m^{\alpha}}+V(x)-g(m)=0 \quad \mbox{in }\;\Tt^2
\]
converges to \eqref{eq:hjbEffHamiltgamma} when $\beta\to 0$; that is, $\beta u^{(\beta)}$ converges uniformly to $-\Hh$, and $u^{(\beta)}-\max_{x\in\Tt^2} u(x)$ converges uniformly, up to subsequences, to $u$.
We then use the notation for grid functions from Section~\ref{sub:discretization}, and numerically solve, for small $\beta>0$, the  discrete problem 
\[
\beta u^{(\beta)}_{i,j} + \frac{|D^h u^{(\beta)} + P|_{i,j}^{\gamma}}{\gamma m^\alpha_{i,j}} +V_{i,j}-g_{i,j} = 0\quad \mbox{in $\Tt^2_N$}.
\]
Hence, we set 
\[
\overline{H}^{(\beta)} = \max\{u^{(\beta)}_{i,j},\; i,j\in\{0,\ldots, N-1\} \}.
\]
Finally, the solution of \eqref{eq:hjbEffHamiltgamma} is approximated by
\[
u_{i,j}\cong u^\beta_{i,j} - \overline{H}^\beta.
\]



\subsection{Numerical experiments}
In Fig.~\ref{fig:solTransf}, we illustrate the solution of \eqref{eq:hjb2d}  for $\alpha=0.8$, $\gamma=2$, and $P=(1,3)$. Moreover, we use the coupling $G(m) = m^3$ and the potential $$V(x,y)= \sin\left(2\pi \left(x+\frac 1 4\right)\right)\cos\left(2\pi \left(y+\frac 1 4\right)\right).$$ 
In this example, we have $Q=(3,-1)$. The value function, $u$, and the density, $m$, are displayed in Figs.~\ref{fig:plotuTransf} and \ref{fig:plotmTransf}, respectively. Note that, even with $\alpha<1$, the density is similar to the density in Fig. \ref{fig:plotmls}. However, for this example, the value function behaves differently from what we observed in Fig. \ref{fig:plotuls}.
\begin{figure}[h!]
        \centering
        \begin{subfigure}[b]{\sizefigure\textwidth}
                \includegraphics[width=\textwidth]{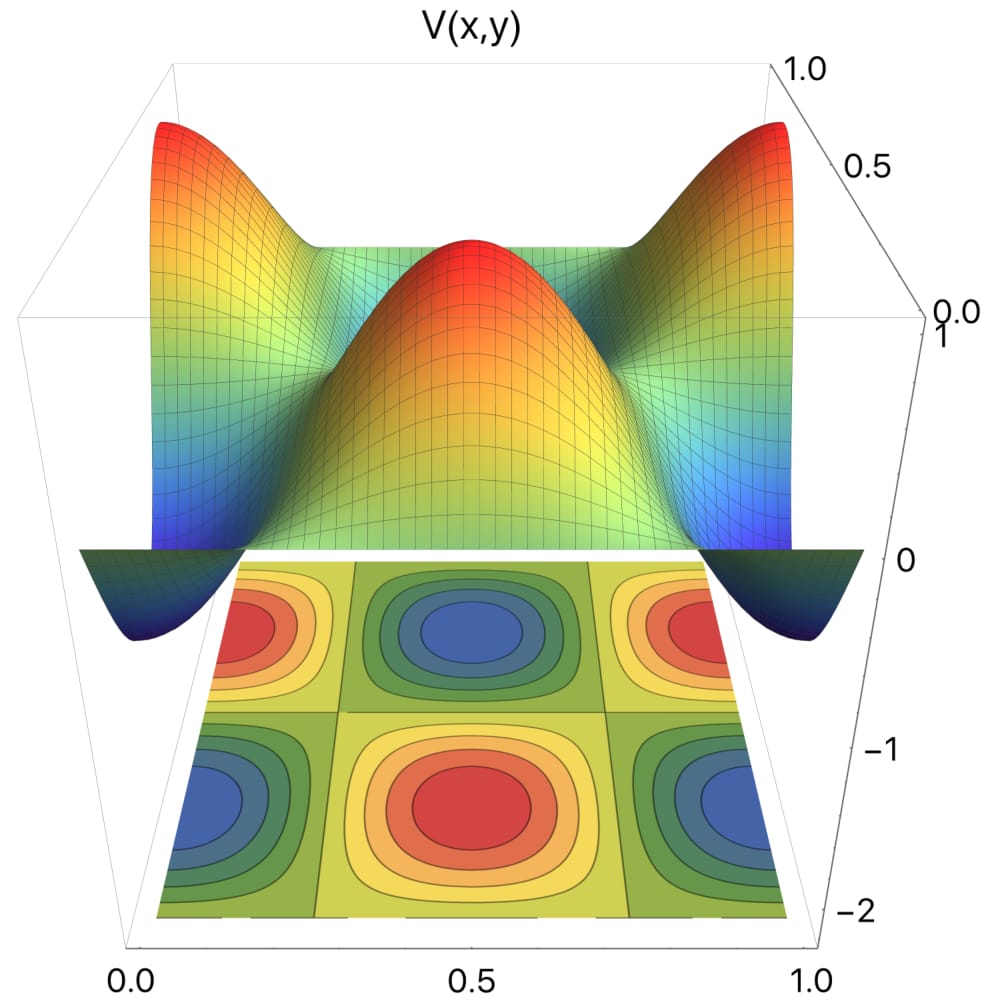}
                \caption{Potential $V$}
                \label{fig:plotVTransf}
        \end{subfigure}\\
        \begin{subfigure}[b]{\sizefigure\textwidth}
                \includegraphics[width=\textwidth]{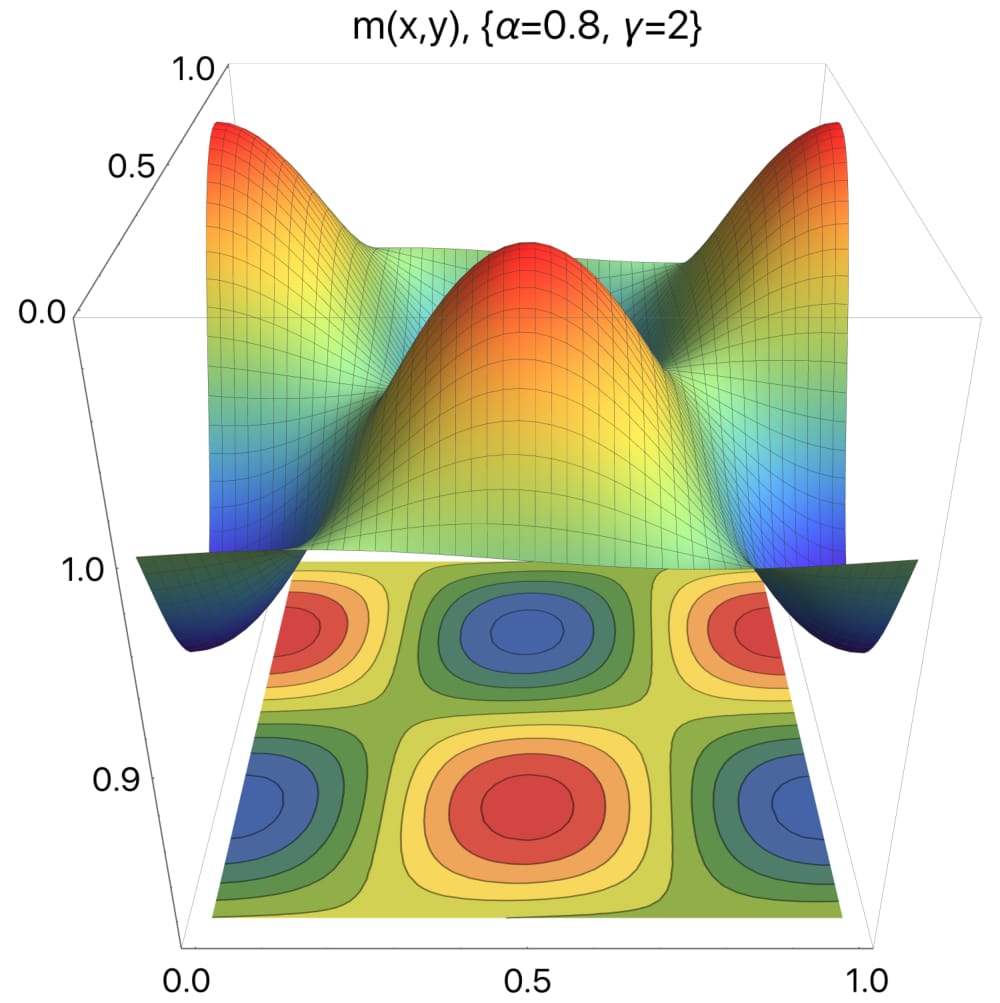}
                \caption{Density $m$}
                \label{fig:plotmTransf}
        \end{subfigure}
        \begin{subfigure}[b]{\sizefigure\textwidth}
                \includegraphics[width=\textwidth]{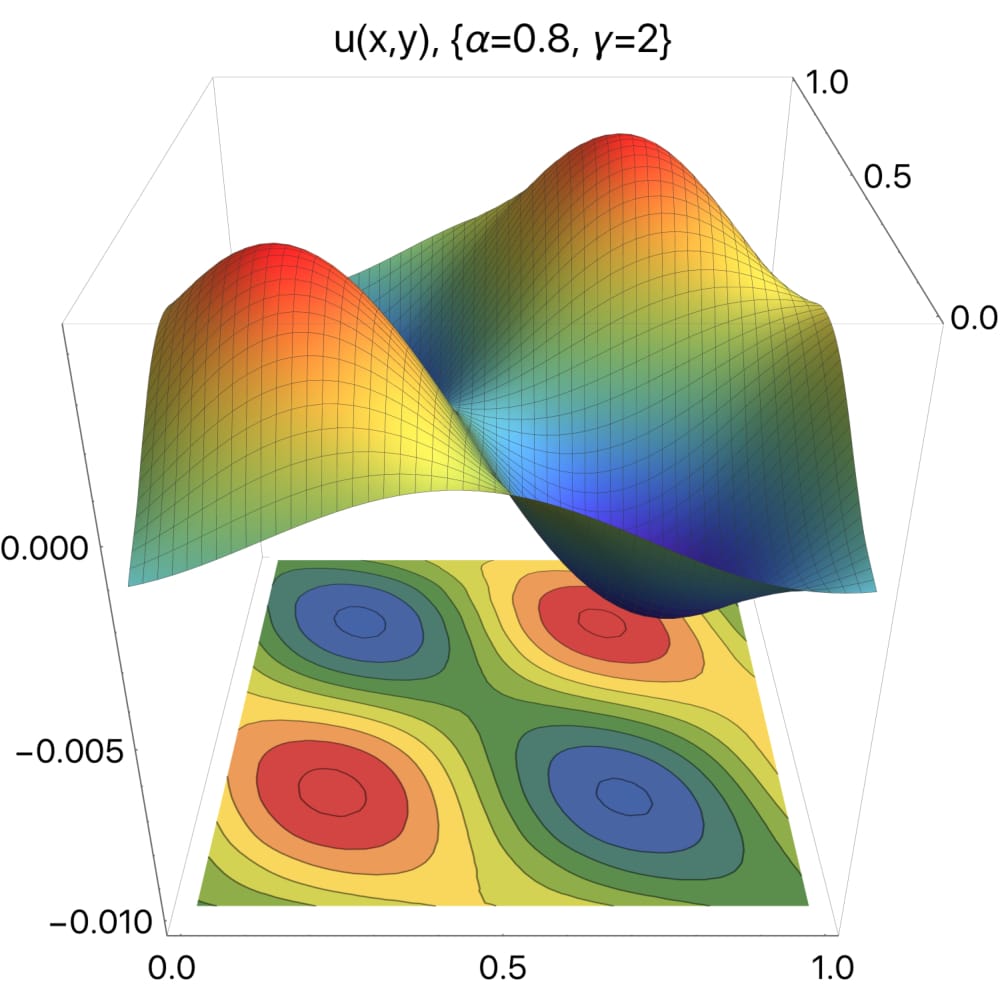}
                \caption{Value function $u$}
                \label{fig:plotuTransf}
        \end{subfigure}
        ~ 
        \caption{Numerical solution of \eqref{eq:hjb2d} with \(N=50\) and for  $P=(1,3)$, $V(x,y)= \sin\left(2\pi \left(x+\frac 1 4\right)\right)\cos\left(2\pi \left(y+\frac 1 4\right)\right)$, $G(m) = m^3$, $\alpha= 0.8$, and $\gamma=2$.}\label{fig:solTransf}
\end{figure}


\section{Numerical solution for the second-order MFGs with congestion, with $\gamma=2$ and $1<\alpha<\gamma$}
\label{num9}

In this section, for illustration purposes, we numerically solve a variational problem introduced in Section \ref{otherH}, which corresponds to a second-order MFG with congestion in the quadratic case. More precisely, we consider the minimization problem
\begin{equation}\label{eq:varfomul2ndord}
\min_{\psi \geq 0}\hat J[\psi], \quad\hat J[\psi]=\int_{\Tt^d} \bigg[\frac{|D\psi|^2}{2\alpha+1} +\hat{G}(\psi)-\frac{2\alpha+1}{2(\alpha+1)}(V(x)-\overline{H})\psi^{\frac{2(\alpha+1)}{2\alpha+1}}  \bigg] dx, 
\end{equation}
where, for \(z\geq0\),
\begin{equation}\label{eq:g9.1}
\hat{G}(z)=\int_0^z g(r^{\frac{2}{2\alpha+1}})r^{\frac{1}{2\alpha+1}}\,dr
\end{equation}
and $\Hh$ is chosen such that the constraint
\begin{equation}\label{eq:intpsi1}
\begin{aligned}
\int_{\Tt^d} \psi^{\frac
        2 {2\alpha+1}}\,dx=1
\end{aligned}
\end{equation}
holds. Recall  that the Euler--Lagrange equation of the functional in \eqref{eq:varfomul2ndord} is  \eqref{Eq:statcongquad}. Moreover,  \eqref{Eq:statcongquad} corresponds to \eqref{eq.pLap} for $\gamma=2$. 
 
Using the notation for grid functions introduced in Section \ref{sub:discretization}, we present the discretized version of \eqref{eq:varfomul2ndord} in the two-dimensional case as follows. For a grid function, $\psi\in \Rr^{N^2}$, we numerically solve 
\begin{equation}
\min_{\psi\geq 0} J_h[\psi],
\end{equation}
where
\[
J_h[\psi]=h^2\sum_{i,j=0}^{N-1}\left( ( f_h([D_h \psi]))_{i,j}+G(\psi_{i,j})- \frac{2\alpha +1 }{2(\alpha +1)}(V_{i,j} - \overline{H})\psi_{i,j}^{\frac{2(\alpha+1)}{2\alpha+1}}\right)
\]
and
\[
 ( f_h([D_h \psi]))_{i,j}=\frac{1}{(2\alpha + 1)}
\left(\left((D^h_1 \psi)_{i,j}\right)^2
+\left((D^h_2 \psi)_{i,j}\right)^2\right).
\]
Moreover, as in the previous sections, the discretization scheme follows \eqref{diffschemes}--\eqref{gradCentralDiff}. Recall that  $P=(0,0)$ here. In the following subsection, we perform a numerical experiment to illustrate our method.


\subsection{Numerical experiment in the two-dimensional case}
Here, we depict a numerical  solution of \eqref{eq:varfomul2ndord} for the potential 
 \[
 V(x,y)=  e^{-\sin\left(2\pi \left(x+\frac 1 4\right)\right)^2}\sin\left(2\pi \left(y-\frac 1 4\right)\right),
 \]
 the coupling $g(r) = r^3$ (see \eqref{eq:g9.1}), and $\alpha= 1.5$ (see Fig.~\ref{fig:solSecOr}).
 For this example, the value of $\overline{H}$ such that \eqref{eq:intpsi1} is satisfied is approximately $\overline{H} \cong -3.001$.  The density,
$m$, is displayed in Fig.~\ref{fig:plotmSecOr}.
In this example, we note that the density still resembles the potential. However, diffusion tends to spread the distribution of the agents making it more uniform. 
\begin{figure}[h!]
        \centering
        \begin{subfigure}[b]{\sizefigure\textwidth}
                \includegraphics[width=\textwidth]{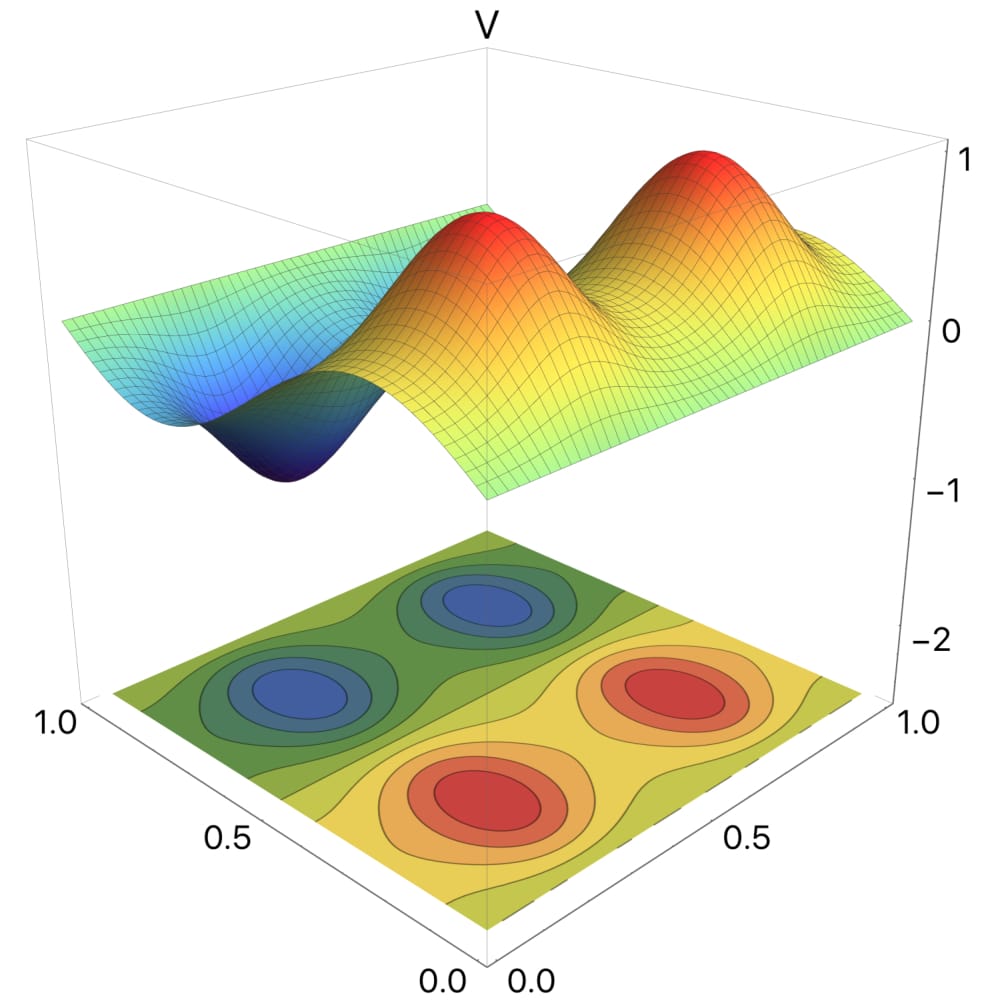}
                \caption{Potential $V$}
                \label{fig:plotVSecOr}
        \end{subfigure}
        \begin{subfigure}[b]{\sizefigure\textwidth}
                \includegraphics[width=\textwidth]{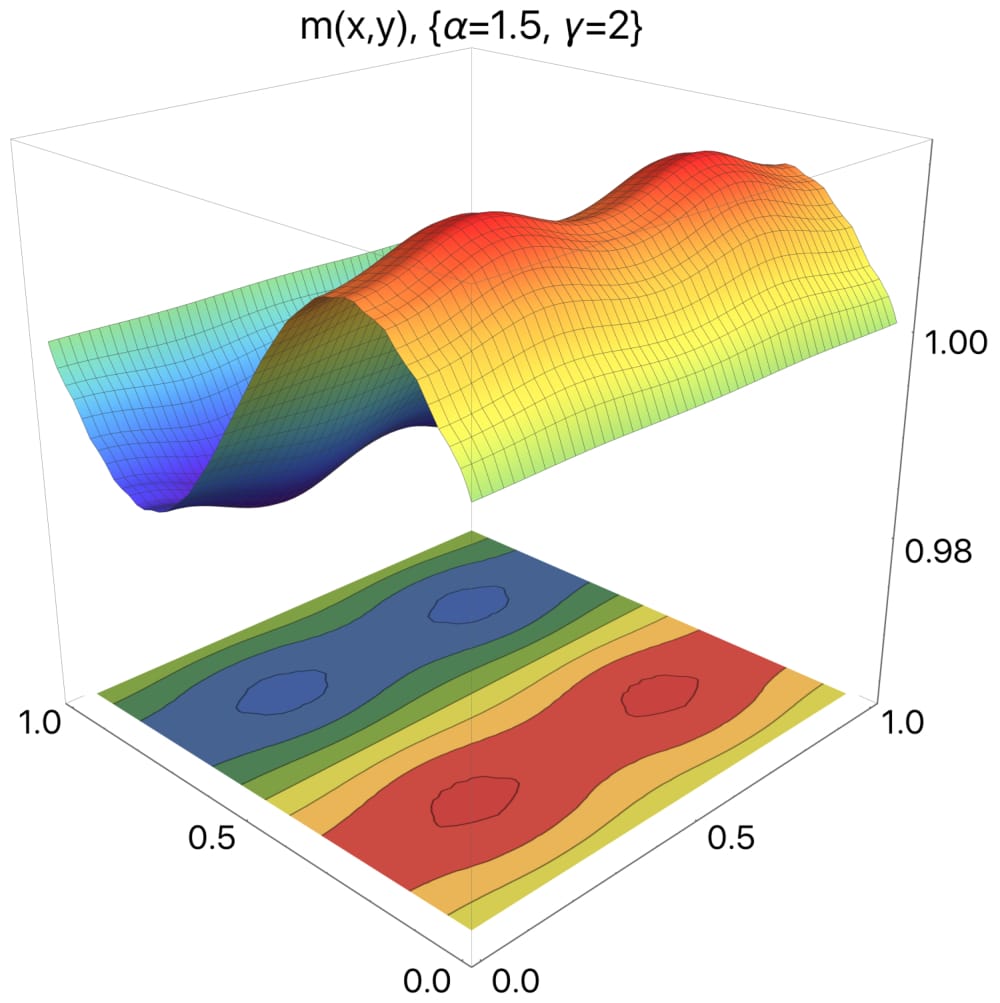}
                \caption{Density $m$}
                \label{fig:plotmSecOr}
        \end{subfigure}
        ~ 
        \caption{Numerical solution  of \eqref{eq:varfomul2ndord} with \(N=50\) and for \(d=2\),   $V(x,y)=  e^{-\sin\left(2\pi \left(x+\frac 1 4\right)\right)^2}\sin\left(2\pi \left(y-\frac 1 4\right)\right)$, $g(m) = m^3$, and $\alpha= 1.5$.}\label{fig:solSecOr}
\end{figure}

\section{Conclusions}
In this paper, we develop a new variational formulation for systems of first-order MFGs with congestion. This variational principle provides a new construction of weak solutions and leads to a novel numerical approach through an optimization problem. 

Even though the variational structure strongly depends on the form of the MFGs, it is often possible to modify the variational principle. This approach was
followed  in Section~\ref{2dcase} for the first-order case and in Section~\ref{tsom} for the second-order case. Moreover, in Sections \ref{num}--\ref{num9}, we use these results to obtain new numerical solutions for MFGs with congestion.
 
\def\cprime{$'$}

\end{document}